\documentclass[matprg]{svjour3}
\usepackage{graphics,graphicx,epsf,subfigure,epstopdf}
\usepackage{amsmath,amsxtra,amsfonts,amscd,amssymb,bm}
\usepackage{algorithm}
\usepackage{algpseudocode}
\usepackage[top=2.5cm,bottom=2.5cm,right=2.5cm,left=2.5cm]{geometry}
\usepackage[mathscr]{eucal}
\usepackage{color}
\numberwithin{equation}{section}

\usepackage{hyperref}

\DeclareMathOperator*{\argmin}{arg\,min}

\def\eqnok#1{(\ref{#1})}

\newcommand{\tsum}{\textstyle\sum}

\newcommand{\beq}{\begin{equation}}
\newcommand{\eeq}{\end{equation}}
\newcommand{\beqa}{\begin{eqnarray}}
\newcommand{\eeqa}{\end{eqnarray}}
\newcommand{\beqas}{\begin{eqnarray*}}
\newcommand{\eeqas}{\end{eqnarray*}}

\newtheorem{assumption}{Assumption}

\newcommand{\nn}{\nonumber}

\def\bS{{\mathbb{S}}}
\def\bA{{\mathbb{A}}}

\def\vgap{\vspace*{.1in}}

\global\long\def\R{\mathbb{R}}%
\global\long\def\E{\mathbb{E}}%
\global\long\def\P{\mathbb{P}}%

\newcommand{\cO}{\mathcal{O}}

\newcommand{\tr}{\intercal}

\title{Value Mirror Descent for Reinforcement Learning
\thanks{This research was partially supported by Office of Naval Research grant N000142412654.}}

\author{Zhichao Jia \quad Guanghui Lan}

\institute{Zhichao Jia \\ H. Milton Stewart School of Industrial and Systems Engineering, Georgia Institute of Technology, Atlanta, GA, 30332. \\ Email: {\tt zjia75@gatech.edu} \\ \\ Guanghui Lan \\ H. Milton Stewart School of Industrial and Systems Engineering, Georgia Institute of Technology, Atlanta, GA, 30332. \\ Email: {\tt george.lan@isye.gatech.edu}}

\date{}

\begin{document}

\maketitle

\begin{abstract}
Value iteration-type methods have been extensively studied for computing a nearly optimal value function in reinforcement learning (RL). Under a generative sampling model, these methods can achieve sharper sample complexity than policy optimization approaches, particularly in terms of their dependence on the discount factor. In practice, they are often employed for offline training or in simulated environments. In this paper, we consider discounted Markov decision processes with state space $\bS$, action space $\bA$, discount factor $\gamma \in (0,1)$ and costs in $[0,1]$. We introduce a novel value optimization method, termed value mirror descent (VMD), which integrates mirror descent from convex optimization into the classical value iteration framework. In the deterministic setting with known transition kernels, we show that VMD converges linearly.
For the stochastic setting with a generative model, we develop a stochastic variant, SVMD, which incorporates variance reduction commonly used in stochastic value iteration-type methods. For RL problems with general convex regularizers, SVMD attains a near-optimal sample complexity of $\Tilde{\cO}(|\bS||\bA|(1-\gamma)^{-3}\epsilon^{-2})$. Moreover, we establish that the Bregman divergence between the generated and optimal policies remains bounded throughout the iterations, even under the presence of model misspecification. This property is absent in existing stochastic value iteration-type methods but is important for enabling effective online (continual) learning following offline training.
Under a strongly convex regularizer, SVMD achieves sample complexity of $\Tilde{\cO}(|\bS||\bA|(1-\gamma)^{-5}\epsilon^{-1})$, improving performance in the high-accuracy regime. Furthermore, we prove convergence of the generated policy to the optimal policy. Overall, the proposed method, its analysis, and the resulting guarantees, including improved sample complexity under strong convexity and bounded or convergent Bregman divergence, constitute new contributions to the RL and optimization literature.
\end{abstract}
\vgap

{\noindent\keywordname reinforcement learning $\cdot$ value iteration-type methods $\cdot$ value mirror descent}
\vgap

{\noindent\subclassname 90C30 $\cdot$ 90C39 $\cdot$ 90C40 $\cdot$ 68Q25}

\section{Introduction}

Markov Decision Process (MDP) is a well-studied mathematical framework for sequential decision-making under uncertainty. At each time step, an agent observes the current state of the environment and selects an action according to a policy. A cost is then incurred for the resulting state-action pair, and the environment transitions to a new state according to a probability transition kernel determined by the chosen action. The goal of solving an MDP is to find an optimal policy that minimizes the expected discounted cumulative cost over an infinite horizon (or a prescribed finite horizon). Such a policy specifies the best action to take at each state.

Reinforcement learning (RL) problems are MDPs with an unknown probability transition kernel $\P$. Given a state-action pair, instead of having direct access to the transition probabilities, one can only generate samples of the next state according to $\P$. Classical methods for solving RL problems can be broadly divided into three categories: dynamic optimization, linear optimization and nonlinear optimization methods. Dynamic optimization methods are typically categorized into value iteration-type and policy iteration-type methods. Value iteration-type methods, including classical value iteration~\cite{bellman1957markovian,puterman2014markov,shapley1953stochastic} and Q-learning~\cite{watkins1992q}, update value functions, whereas policy iteration-type methods, exemplified by classical policy iteration~\cite{howard1960dynamic,puterman2014markov}, update policies. Their stochastic variants have been extensively developed for RL, including model-based empirical Q-value iteration~\cite{agarwal2020model,gheshlaghi2013minimax}, randomized value iteration~\cite{sidford2023variance}, variance-reduced Q-value iteration~\cite{sidford2018near}, and approximate policy iteration~\cite{bertsekas2011approximate,lagoudakis2003least}. In addition to dynamic optimization methods, linear optimization methods form another important class of approaches for RL, based on linear programming formulations of MDPs; examples include stochastic primal-dual methods~\cite{chen2016stochastic,wang2017primal} for stochastic saddle point formulations of RL problems. More recently, nonlinear optimization techniques have been widely incorporated into RL, such as trust region policy optimization methods~\cite{schulman2015trust}, proximal policy optimization methods~\cite{schulman2017proximal}, and  policy mirror descent first studied in ~\cite{lan2023policy} and later in \cite{ju2022policy,ju2024strongly,li2025stochastic,li2025policy,li2022first,zhan2023policy}, among many others. Here these methods are referred to as policy optimization methods.

Both policy optimization methods and stochastic value iteration-type methods are widely used in reinforcement learning, representing two fundamentally different algorithmic paradigms.
Policy optimization methods offer several significant advantages. First, they are flexible enough to accommodate different sampling regimes, including generative models and online Markovian sampling. Moreover, they can naturally handle, or even benefit from, the presence of (strongly) convex regularizers in RL models. Finally, policy optimization methods are often easier to integrate with on-policy learning~\cite{ju2025auto,li2025policy}, which facilitates action and state exploration in online settings.
In contrast, stochastic value iteration-type methods offer several distinct advantages. First, under a generative model, they enjoy superior sample complexity. For a discounted Markov decision process (DMDP) with state space $\bS$, action space $\bA$, discount factor $\gamma\in(0,1)$, and costs in $[0,1]$, stochastic value iteration methods in~\cite{agarwal2020model,sidford2018near} attain an $\epsilon$-optimal policy with sample complexity $\Tilde{\cO}(|\bS||\bA|(1-\gamma)^{-3}\epsilon^{-2})$, matching the known lower bound in~\cite{gheshlaghi2013minimax} up to logarithmic factors, and improving upon policy optimization methods~\cite{ju2022policy,lan2023policy} in terms of the dependence on $1-\gamma$. Therefore, when a generative model is available, value iteration-type methods admit stronger convergence guarantees. Moreover, they are often simpler to implement, as they do not require repeated policy evaluation steps, which can be computationally involved.

Despite the superior sample complexity of stochastic value iteration-type methods, they still face several important challenges. First, although these methods have been extended to various regularized RL settings~\cite{geist2019theory,haarnoja2017reinforcement}, strong convergence guarantees are still lacking. In particular, it appears difficult for stochastic value iteration-type methods to achieve an $\cO(\epsilon^{-1})$ sample complexity when solving strongly convex regularized RL problems. This limitation highlights a gap compared to policy optimization methods, for which sample complexity guarantees have been more thoroughly developed in both unregularized and regularized settings~\cite{ju2022policy,ju2024strongly,lan2023policy}.
Second, the approximate optimal policies produced by value iteration-type methods can be unstable. Specifically, a policy may achieve a near-optimal value while remaining far from the optimal policy in terms of Bregman divergence. In contrast, access to a stable near-optimal policy is often desirable in practice, for example, in fine-tuning a pre-trained agent. If the pre-trained policy remains close to the optimal one under a suitable Bregman divergence, then the probabilities of selecting optimal actions remain positive, thereby providing a favorable warm start for continual training using policy optimization methods in online and on-policy settings~\cite{ju2025auto,li2025policy}. Consequently, policies produced by existing value iteration-type methods are not always well suited for continual online adaptation. This issue becomes even more salient in the presence of model misspecification, where only an inexact transition probability kernel is accessible. Such model mismatch may arise in different settings: (a) model-based methods, where the transition kernel is first approximated and a deterministic value iteration method is then applied; or (b) model-free methods that rely on an inaccurate simulator or digital twin of the true system.
Although existing value iteration-type methods may still produce policies whose values are near-optimal up to an error induced by the model mismatch, these policies can remain far from the true optimal policy in terms of Bregman divergence.

\subsection{Our Contributions}

Our contributions in this paper are summarized from the following several aspects.

First, we introduce a novel value optimization framework, termed value mirror descent (VMD), which integrates mirror descent from convex optimization into the classical value iteration paradigm. More specifically, our method alternates between a prox-mapping step that computes an improved policy and a value update based on this policy. The prox-mapping step can be viewed as a natural generalization of the classical value iteration update, incorporating appropriately chosen stepsizes and Bregman divergences.
We show that VMD achieves a linear convergence rate, attaining an $\epsilon$-optimal policy within $\cO(\log_2[(1-\gamma)^{-1}\epsilon^{-1}])$ epochs and $\cO((1-\gamma)^{-1})$ iterations per epoch. Consequently, the overall iteration complexity, measured by the number of prox-mapping steps, is $\Tilde{\cO}((1-\gamma)^{-1})$, matching the optimal complexity of classical value iteration while arising from a fundamentally different optimization-based perspective. Furthermore, in a model misspecification scenario where the accessible transition kernel $\P$ is an inexact approximation of the true kernel $\Bar{\P}$, satisfying $\max_{s\in\bS,a\in\bA}\|\P_{s,a}-\Bar{\P}_{s,a}\|_1 \leq \tau$, VMD outputs an $\epsilon$-optimal and stable policy with $\epsilon=\cO(\tau(1-\gamma)^{-2})$, and the Bregman divergence between the output policy and the true optimal policy under $\Bar{\P}$ is bounded by $\cO((1-\gamma)^{-1})$. Such policy stability can facilitate subsequent online training with model-free RL methods suited for on-policy learning~\cite{ju2025auto,li2025policy}, using fresh data collected from a more accurate simulator or directly from the true kernel.

Second, we extend value mirror descent to a stochastic setting for solving RL problems with general convex regularizers, where the transition kernel is unknown but a generative model is available. Our method incorporates variance reduction technique in the sampling procedure, and establishes explicit relations between the computed value vectors and the policy value vectors. As a result, we obtain new convergence guarantees for our method and show that the overall sample complexity for computing an $\epsilon$-optimal policy with high probability is $\Tilde{\cO}(|\bS||\bA|(1-\gamma)^{-3}\epsilon^{-2})$. This matches the sample complexity of variance-reduced Q-value iteration method in~\cite{sidford2018near}, and is optimal up to logarithmic factors in view of the lower bound in~\cite{gheshlaghi2013minimax}. In addition, our method guarantees that the Bregman divergence between the output $\epsilon$-optimal policy and the optimal policy is bounded by $\cO((1-\gamma)^{-1})$, which improves upon $\cO((1-\gamma)^{-2})$ obtained by stochastic policy mirror descent method (see~\cite{li2025policy}). Under a model misspecification setting in which the simulator follows an inexact kernel $\P$ that differs from the true kernel $\Bar{\P}$, satisfying $\max_{s\in\bS,a\in\bA}\|\P_{s,a}-\Bar{\P}_{s,a}\|_1 \leq \tau$, we obtain the same bound $\cO((1-\gamma)^{-1})$ on the Bregman divergence between the output $\epsilon$-optimal policy and the true optimal policy under $\Bar{\P}$, with the same sample complexity and with high probability, where $\epsilon=\cO(\tau(1-\gamma)^{-2})$. 

Third, we establish convergence guarantees for stochastic value mirror descent in RL problems with strongly convex regularizers. Our analysis introduces virtual sequences of value vectors that serve as proxies of the computed value iterates, and exploits the strong convexity underlying the regularizer and Bregman divergence. Based on these techniques, we show that our method requires $\Tilde{\cO}(|\bS||\bA|(1-\gamma)^{-5}\epsilon^{-1})$ samples from a generative model to compute an $\epsilon$-optimal policy with high probability. This improves over the variance-reduced Q-value iteration method in~\cite{sidford2018near} when $\epsilon<\cO((1-\gamma)^2)$. To the best of our knowledge, this is the first value iteration-type method to achieve an $\Tilde{\cO}(\epsilon^{-1})$ dependence, rather than $\cO(\epsilon^{-2})$, in sample complexity for regularized RL problems. Moreover, the Bregman divergence between the output $\epsilon$-optimal policy and the optimal policy is bounded by $\cO(\epsilon)$, indicates the convergence of the generated policy to the optimal policy.

We note that several existing works~\cite{geist2019theory,kozuno2022kl,li2025policy2,moulin2023optimistic,moulin2025optimistically,vieillard2020leverage} have also considered incorporating mirror descent-type policy updates into value iteration-type methods. When the transition kernel is known,~\cite{geist2019theory,vieillard2020leverage} studied a class of mirror descent modified value iteration methods and provided error-propagation analyses. In the unregularized stochastic setting,~\cite{kozuno2022kl} extended this class of methods and claimed to achieve minimax optimality under a generative sampling model. Subsequently,~\cite{moulin2023optimistic,moulin2025optimistically} further developed this line of work for online learning and established regret guarantees. More recently,~\cite{li2025policy2} considered a similar mirror descent-type policy update in conjunction with temporal difference learning. Nevertheless, our algorithms differ from these existing methods in both deterministic and stochastic settings. Moreover, none of these existing methods achieves an $\cO(\epsilon^{-1})$ sample complexity for solving strongly convex regularized RL problems or guarantees policy stability, either under an exact model or in the presence of model misspecification.

The rest of this paper is organized as follows. In Section~\ref{section:background}, we introduce background on DMDPs and the classical value iteration method. Section~\ref{section:deter} presents the value mirror descent method for solving DMDPs and establishes its convergence results. Section~\ref{section:stoc1} and Section~\ref{section:stoc2} develop the stochastic value mirror descent method and its convergence guarantees for RL problems with general convex regularizers and strongly convex regularizers, respectively. Section~\ref{section:conclusion} concludes the paper. Several technical proofs are deferred to the appendix.

\subsection{Notations}

For a finite-dimensional normed space, let $\|\cdot\|$ be the norm and $\|\cdot\|_*$ be its dual norm. For any matrix $A\in\R^{m\times n}$, define $\|A\|_{\max}=\max_{i=1,...,m,j=1,...,n}|A_{i,j}|$. For any finite set $S\subseteq\R^n$, we use $|S|$ to denote its cardinality. We define the simplex in $\R^n$ by $\Delta_n:=\{x\in\R^n\mid 0\leq x_i\leq 1 \ \forall i=1,\cdots,n,\tsum_{i=1}^nx_i=1\}$. For any function $f: X\to\R$, the Clarke subdifferential of $f$ at $x\in X$ is defined as $\partial f(x)=\mathrm{co}\{\lim_{i\to\infty}\nabla f(x_i): f(x) \ \text{is differentiable at any} \ x_i\in X, \text{and} \ x_i\to x\}$, where $\mathrm{co}(\cdot)$ denotes the convex hull. We write $f'(x)\in\partial f(x)$ for an arbitrary Clarke subgradient at $x$. A function $f:X\to\R$ is said to be $L$-Lipschitz continuous if $|f(x_1)-f(x_2)|\leq L\|x_1-x_2\|$ holds for any $x_1,x_2\in X$,
or equivalently, $\|f'(x)\|_*\leq L$ for any $x\in X$. We say that a continuous function $f:X\to\R$ is $\mu$-strongly convex if $f(x_2)\geq f(x_1)+\langle f'(x_1),x_2-x_1\rangle+(\mu/2)\|x_1-x_2\|^2$ is satisfied for any $x_1,x_2\in X$ and $f'(x_1)\in\partial f(x_1)$. Let $\omega:X\to\R$ be a continuously differentiable and $1$-strongly convex distance generating function. Then, for any $x_1,x_2\in X$, the associated Bregman divergence (or prox-function) is defined by $D(x_1,x_2):=\omega(x_2)-\omega(x_1)-\langle\nabla\omega(x_1),x_2-x_1\rangle$, and it satisfies $D(x_1,x_2)\geq\|x_1-x_2\|^2/2$. For any $x\in\R^n$, $|x|$ denotes the elementwise absolute value of $x$.
\section{Background}
\label{section:background}

In this section, we present preliminaries on DMDPs considered in this paper, as well as the classical value iteration method underlying our algorithm.

\subsection{Discounted Markov Decision Processes}

A DMDP can be described by a tuple $(\bS,\bA,\P,c,\gamma)$, where $\bS$ is a finite state space, $\bA$ is a finite action space, $\P\in\R^{|\bS|\times|\bA|\times|\bS|}$ represents the probability transition kernel, $c\in\R^{|\bS|\times|\bA|}$ denotes the cost function, and $\gamma\in (0,1)$ is the discount factor. When action $a\in\bA$ is taken at state $s\in\bS$, the system transitions to the next state $s'\in\bS$ with probability $\P(s'|s,a)$, and incurs cost $c(s,a)$. Throughout the paper, we assume that $c(s,a)\in[0,1]$ for all $(s,a)\in\bS\times\bA$.

For any $(s,a)\in\bS\times\bA$ and $s'\in\bS$, we write $\P_{s,a}(s')=\P(s'|s,a)$, so that $\P_{s,a}\in\R^{|\bS|}$ denotes the transition probability vector associated with $(s,a)$. For any vector $V\in\R^{|\bS|}$, we let $\P V\in\R^{|\bS|\times |\bA|}$ be defined by $(\P V)_{s,a}=\P_{s,a}^{\tr}V$. Furthermore, for any state $s\in\bS$, we define $\P_s V\in\R^{|\bA|}$ by $(\P_s V)(a)=(\P V)_{s,a}$.

 We denote $\Delta_{|\bA|}$ by $\Delta_{\bA}$ (and similarly $\Delta_{|\bS|}$ by $\Delta_{\bS}$). A randomized stationary policy is a mapping $\pi(\cdot|s):\bS\to\Delta_{\bA}$, which assigns to each state $s\in\bS$ a distribution over all actions $a\in\bA$. Accordingly, $\pi(a|s)$ represents the probability of taking action $a$ at state $s$. We let $\Pi$ denote the Cartesian product of $|\bS|$ simplices, representing the set of all stationary policies. For any policy $\pi\in\Pi$, we define $c_{\pi}\in\R^{|\bS|}$ and $\P_{\pi}\in\R^{|\bS|\times |\bS|}$ as the cost vector and probability transition matrix induced by $\pi$, respectively, where $c_{\pi}(s)=\tsum_{a\in\bA}\pi(a|s)c(s,a)$ for all $s\in\bS$, and $\P_{\pi}(s,s')=\tsum_{a\in\bA}\pi(a|s)\P(s'|s,a)$ for all $s,s'\in\bS$. We also use $\mathbf{1}_{\bS}\in\R^{|\bS|},\mathbf{1}_{\bA}\in\R^{|\bA|},\mathbf{0}_{\bS}\in\R^{|\bS|}$ and $\mathbf{0}_{\bA}\in\R^{|\bA|}$ to denote the all-ones and all-zeros vectors of the corresponding dimensions.

The goal of a (regularized) DMDP is to attain an optimal policy $\pi\in\Pi$ that minimizes the expectation of the discounted cumulative cost, which is defined as
\begin{align}
    V^{\pi}(s):=\E\big[\tsum_{t=0}^{\infty}\gamma^t[c(s_t,a_t)+h(\pi(\cdot|s_t))]\mid s_0=s,a_t\sim\pi(\cdot|s_t),s_{t+1}\sim\P(\cdot|s_t,a_t)\big],
    \label{def:Vpi}
\end{align}
where $V^{\pi}:\bS\to\R$ is the state-value function associated with $\pi$. We assume that the regularizer $h:\Delta_{\bA}\to[0,\Bar{h}]$ in~\eqref{def:Vpi} is continuous and convex, and define the vector $h(\pi)\in\R^{|\bS|}$ by $(h(\pi))(s):=h(\pi(\cdot|s))$. For notational simplicity, we suppress the state dependence of $h$. More generally, $h$ in~\eqref{def:Vpi} can be replaced by $h_{s_t}$, where the regularizer $h_s:\Delta_{\bA}\to[0,\Bar{h}]$ may differ for each $s\in\bS$. All subsequent results remain valid in this more general setting, provided that the assumptions on $h$ hold uniformly for $h_s$ over $s\in\bS$. Moreover, since $c(s,a)\in[0,1]$ for all $(s,a)\in\bS\times\bA$, it follows that $V^{\pi}(s)\in [0,(1+\Bar{h})/(1-\gamma)]$ for all $s\in\bS$.

For any $s\in\bS$ and policies $\pi$ and $\pi'$, we let $D^{\pi'}_{\pi}(s)=D(\pi(\cdot|s),\pi'(\cdot|s))$ denote the corresponding Bregman divergence. Common examples of Bregman divergences include the squared Euclidean distance $D_{\pi}^{\pi'}(s)=\|\pi'(\cdot|s)-\pi(\cdot|s)\|_2^2/2$ associated with the norm $\|\cdot\|=\|\cdot\|_2$ and the distance generating function $\omega_{Euc}(\cdot)=\|\cdot\|_2^2/2$, the Kullback-Leibler (KL) divergence $D_{\pi}^{\pi'}(s)=\tsum_{a\in\bA}\pi'(a|s)\ln[\pi'(a|s)/\pi(a|s)]$ associated with the norm $\|\cdot\|=\|\cdot\|_1$ and the distance generating function $\omega_{KL}(\pi(\cdot|s))=\tsum_{a\in\bA}\pi(a|s)\log(\pi(a|s))$, and the Tsallis divergence $D_{\pi}^{\pi'}(s)=[p(1-p)]^{-1}\tsum_{a\in\bA}[-\pi'(a|s)^p+(1-p)\pi(a|s)^p+p\pi(a|s)^{p-1}\pi'(a|s)]$ for any $p\in(0,1)$ associated with the norm $\|\cdot\|=\|\cdot\|_1$ and the distance generating function $\omega_{Tsa}(\pi(\cdot|s))=-[p(1-p)]^{-1}\tsum_{a\in\bA} \pi(a|s)^p$.

By the convexity of $h$, there exists some $\mu\geq 0$ such that
\begin{align}
    h(\pi'(\cdot|s))-\big[h(\pi(\cdot|s))+\big\langle h'(\pi(\cdot|s)),\pi'(\cdot|s)-\pi(\cdot|s)\big\rangle\big]\geq\mu D^{\pi'}_{\pi}(s).
    \label{hconvex}
\end{align}
We say that $h$ is a $\mu$-strongly convex regularizer if~\eqref{hconvex} holds for some $\mu>0$.

Moreover, we call $\pi^*$ an optimal policy if
\begin{align}
    V^{\pi^*}(s)\leq V^{\pi}(s) \qquad \forall s\in\bS,\forall \pi(\cdot|s)\in\Delta_{\bA}.
    \label{def:pistar}
\end{align}
The existence of $\pi^*$ has been thoroughly studied in~\cite{puterman2014markov}. In our setting, $\pi^*$ exists since $\bS$ and $\bA$ are finite sets, and $|c(s,a)|<\infty$ holds for all $(s,a)\in\bS\times\bA$.

Next, we introduce several additional definitions that will be used throughout the paper. We first define the operator $\Gamma:\R^{|\bS|}\to\R^{|\bS|}$ by
\begin{align}
    (\Gamma V)(s):=\min_{\pi(\cdot|s)\in\Delta_{\bA}}[\langle c(s,\cdot)+\gamma\P_sV,\pi(\cdot|s)\rangle+h(\pi(\cdot|s))]
    \label{def:Gamma_V}
\end{align} 
for all $s\in\bS$ and $V\in\R^{|\bS|}$. Next, for any policy $\pi\in\Pi$, we define the linear operator $\Gamma_{\pi}:\R^{|\bS|}\to\R^{|\bS|}$ by
\begin{align}
    \Gamma_{\pi}V:=c_{\pi}+\gamma\P_{\pi}V+h(\pi)\in\R^{|\bS|}
    \label{def:Gamma_pi_V}
\end{align} 
for all $V\in\R^{|\bS|}$, where $(\Gamma_{\pi}V)(s)=\langle c(s,\cdot)+\gamma\P_sV,\pi(\cdot|s)\rangle+h(\pi(\cdot|s))$ for all $s\in\bS$. It is easy to verify that $\Gamma_{\pi}$ is a contraction mapping, i.e., $\|\Gamma_{\pi} V_1-\Gamma_{\pi} V_2\|_{\infty}\leq\gamma\|V_1-V_2\|_{\infty}$ for any $V_1,V_2\in\R^{|\bS|}$. Moreover, $\Gamma_{\pi}$ is monotone, i.e., $\Gamma_{\pi}V_1\geq\Gamma_{\pi}V_2$ when $V_1\geq V_2$. It is also useful to note that $(I-\gamma\P_{\pi})^{-1}$ is monotone, since $(I-\gamma\P_{\pi})^{-1}V=\tsum_{i=0}^{\infty}(\gamma\P_{\pi})^iV$ for any $V\in\R^{|\bS|}$. In addition, using the same expansion, we can see that $(I-\gamma\P_{\pi})^{-1}V\geq V$ for any $V\geq \mathbf{0}_{\bS}$.

Finally, we let $V^*$ denote the unique fixed point of $V=\Gamma V$. Note that for any $\pi\in\Pi$, $V^{\pi}$ defined in~\eqref{def:Vpi} is the unique fixed point of $V=\Gamma_{\pi}V$. In particular, we have $V^{\pi}\geq V^{\pi^*}=V^*$ for any $\pi\in\Pi$. Finally, for any $\epsilon>0$, we say that a policy $\pi$ and its associated value $V^{\pi}$ are $\epsilon$-optimal if $\|V^{\pi}-V^*\|_{\infty}\leq\epsilon$.

\subsection{Value Iteration}

Value iteration (see~\cite{puterman2014markov}) is the most classical and simplest method for solving DMDPs. Given any initial value $V_0\in\R^{|\bS|}$, the update in each iteration $t=0,1,...$ is
\begin{align}
    V_{t+1}=\Gamma V_t,
    \label{valueiteration}
\end{align}
where the operator $\Gamma$ is defined in~\eqref{def:Gamma_V}. Since $\|V_{t+1}-V^*\|_{\infty}=\|\Gamma V_t-\Gamma V^*\|_{\infty}\leq\gamma\|V_t-V^*\|_{\infty}$, value iteration converges linearly to $V^*$. According to~\cite{puterman2014markov}, the method terminates with an $\epsilon$-optimal policy once $\|V_{t+1}-V_t\|_{\infty}\leq\epsilon (1-\gamma)/(2\gamma)$ is attained. In this paper, we develop a new method for solving DMDPs based on the value iteration framework.
\section{Value Mirror Descent}
\label{section:deter}

Mirror descent steps have been widely used to solve DMDPs efficiently, for example, in policy mirror descent methods~\cite{lan2023policy}. Motivated by this line of work, we develop a new value optimization method that incorporates the idea of mirror descent. In this section, we present the basic algorithmic scheme of value mirror descent (VMD) and establish its convergence properties.

\subsection{Basic Value Mirror Descent}

The basic value mirror descent (BVMD) method updates the policy in a less greedy manner than the value iteration method~\eqref{valueiteration} by performing~\eqref{eq1:algo1}, which incorporates the Bregman divergence $D_{\pi_t}^{\pi}$. Note that the operator $\Gamma_\pi$ used in~\eqref{eq1:algo1} is defined by~\eqref{def:Gamma_pi_V}. The stepsize $\eta_t$ in~\eqref{eq1:algo1} balances $\Gamma_{\pi}V_t$ and $D_{\pi_t}^{\pi}$. After computing $\pi_{t+1}$, we update the value iterate by~\eqref{eq2:algo1}. BVMD differs from policy mirror descent methods in that it does not require policy evaluation at each iteration $t$.

\begin{algorithm}[H]
\caption{Basic Value Mirror Descent (BVMD) for DMDPs}
\label{algo:VMD1}
\begin{algorithmic}
\State{\textbf{Input:} $\pi_0(\cdot|s)=\mathbf{1}_{\bA}/|\bA| \ \forall s\in\bS$, $V_0=(1+\Bar{h})/(1-\gamma)\cdot\mathbf{1}_{\bS}\in\R^{|\bS|}$, $\{\eta_t\}$.}
\For{$t=0,1,\ldots$}
\State{For all $s\in\bS$, update}
\State{\begin{align}
&\pi_{t+1}(\cdot|s)=\argmin_{\pi(\cdot|s)\in\Delta_{\bA}}\big\{\eta_t(\Gamma_{\pi}V_t)(s)+D^{\pi}_{\pi_t}(s)\big\}, \label{eq1:algo1} \\
&V_{t+1}(s)=(\Gamma_{\pi_{t+1}}V_t)(s). \label{eq2:algo1}
\end{align}}
\EndFor
\end{algorithmic}
\end{algorithm}
\vgap

In Algorithm~\ref{algo:VMD1}, we initialize $\pi_0(\cdot|s)=\mathbf{1}_{\bA}/|\bA|$ for all $s\in\bS$, so that there exists $D_0>0$ satisfying $\|D_{\pi_0}^{\pi}\|_\infty\leq D_0$ for any policy $\pi\in\Pi$. It is easy to verify that $D_0=1$ under the Euclidean distance, $D_0=\log|\bA|$ under the KL divergence, and $D_0=(|\bA|^{1-p}-1)/[p(1-p)]$ under the Tsallis divergence with $p\in(0,1)$. Moreover, the initial value $V_0$ serves as an upper bound on $V^{\pi}$ for all $\pi\in\Pi$.
\vgap

To establish the convergence of Algorithm~\ref{algo:VMD1}, we first show the descent property of the prox-mapping step~\eqref{eq1:algo1}, which is well-known for mirror descent methods.

\begin{lemma}
    In any iteration $t\in\{0,1,\ldots\}$ of Algorithm~\ref{algo:VMD1}, for any state $s\in\bS$ and policy $\pi(\cdot|s)\in\Delta_{\bA}$, we have
    \begin{align*}
        \eta_t[\langle c(s,\cdot)+\gamma\P_sV_t,\pi_{t+1}(\cdot|s)-\pi(\cdot|s)\rangle+h(\pi_{t+1}(\cdot|s))-h(\pi(\cdot|s))]+D^{\pi_{t+1}}_{\pi_t}(s)\leq D^{\pi}_{\pi_t}(s)-(1+\eta_t\mu)D^{\pi}_{\pi_{t+1}}(s).
    \end{align*}
    \label{lem:threepoint1}
\end{lemma}
\begin{proof}
    This proof directly follows from the proof of~\cite[Lemma 3.4]{lan2020first}, by using the optimality condition of~\eqref{eq1:algo1}, the property of the Bregman divergence and the ($\mu$-strong) convexity of $h$ in~\eqref{hconvex}.
\end{proof}
\vgap

Built upon Lemma~\ref{lem:threepoint1}, we establish the following convergence property for multiple iterations of Algorithm~\ref{algo:VMD1}.

\begin{lemma}
    For any $T\in\{1,2,\ldots\}$, taking $\eta_t=\eta>0$ for $t=0,1,\ldots,T-1$ in Algorithm~\ref{algo:VMD1} yields
    \begin{align}
        (V_T-V^*)+\tsum_{t=1}^{T-1}(I-\gamma\P_{\pi^*})(V_t-V^*)+\tfrac{1}{\eta}D_{\pi_{T}}^{\pi^*}\leq \tfrac{1}{\eta}D_{\pi_0}^{\pi^*}+\gamma\P_{\pi^*}(V_0-V^*).
        \label{res1:lem3p1}
    \end{align}
    \label{lem:converge_deter}
\end{lemma}
\begin{proof}
    Taking $\pi=\pi^*$ and $\eta_t=\eta$ in Lemma~\ref{lem:threepoint1}, for any $t\in\{0,1,\ldots,T-1\}$ and $s\in\bS$, we have
    \begin{align}
        \eta\left[\left\langle c(s,\cdot)+\gamma\P_sV_t,\pi_{t+1}(\cdot|s)-\pi^*(\cdot|s)\right\rangle+h(\pi_{t+1}(\cdot|s))-h(\pi^*(\cdot|s)\right]+D^{\pi_{t+1}}_{\pi_t}(s)\leq D^{\pi^*}_{\pi_t}(s)-D^{\pi^*}_{\pi_{t+1}}(s).
        \label{eq1:lem_converge_deter}
    \end{align}
    Notice that
    \begin{align}
        \big\langle c(s,\cdot)+\gamma\P_sV_t,\pi^*(\cdot|s)\big\rangle+h(\pi^*(\cdot|s))&=\big\langle c(s,\cdot)+\gamma\P_sV^*,\pi^*(\cdot|s)\big\rangle+h(\pi^*(\cdot|s))+\big\langle\gamma\P_s(V_t-V^*),\pi^*(\cdot|s)\big\rangle \nn\\
        &=V^*(s)+\gamma (\P_{\pi^*})_s(V_t-V^*),
        \label{eq2:lem_converge_deter}
    \end{align}
    where the second equality applies the relation $V^*=c_{\pi^*}+\gamma P_{\pi^*}V^*+h(\pi^*)$. In view of~\eqref{eq2:algo1},~\eqref{eq2:lem_converge_deter} and the fact that $D_{\pi_t}^{\pi_{t+1}}\geq 0$,~\eqref{eq1:lem_converge_deter} becomes
    \begin{align}
        \eta (V_{t+1}(s)-V^*(s))\leq D^{\pi^*}_{\pi_t}(s)-D^{\pi^*}_{\pi_{t+1}}(s)+\eta\gamma(\P_{\pi^*})_s(V_t-V^*).
        \label{eq3:lem_converge_deter}
    \end{align}
    Since~\eqref{eq3:lem_converge_deter} holds for every $s\in\bS$, then
    \begin{align}
        \eta (V_{t+1}-V^*)\leq D^{\pi^*}_{\pi_t}-D^{\pi^*}_{\pi_{t+1}}+\eta\gamma\P_{\pi^*}(V_t-V^*),
        \label{eq4:lem_converge_deter}
    \end{align}
    and we obtain~\eqref{res1:lem3p1} by summing up~\eqref{eq4:lem_converge_deter} for $t=0,1,\ldots,T-1$.
\end{proof}
\vgap

From Lemma~\ref{lem:converge_deter}, we see that multiplying both sides of~\eqref{res1:lem3p1} by the monotone operator $(I-\gamma\P_{\pi^*})^{-1}$ yields an upper bound on the sum of the value gaps $V_t-V^*,t=1,2,\ldots,T-1$, the transformed terminal value gap $(I-\gamma\P_{\pi^*})^{-1}(V_T-V^*)$, and the transformed terminal Bregman distance $(I-\gamma\P_{\pi^*})^{-1}D_{\pi_T}^{\pi^*}$. Next, in order to bound the terminal policy value gap $V^{\pi_T}-V^*$, we establish two monotonicity properties.

\begin{lemma}
    For any $T\in\{0,1,\ldots\}$, the following properties hold for Algorithm~\ref{algo:VMD1}:
    \begin{itemize}
        \item [(a)] $V_0\geq V_1\geq V_2\geq\cdots\geq V_T$;
        \item [(b)] $V_t\geq\Gamma_{\pi_t}V_t\geq V^{\pi_t}\geq V^*$ for $t=0,1,\ldots,T$.
    \end{itemize}
    \label{lem:monotone_deter}
\end{lemma}
\begin{proof}
    We prove the result by mathematical induction. First, we observe that $\Gamma_{\pi_0}V_0=c_{\pi_0}+\gamma\P_{\pi_0}V_0+h(\pi_0)\leq [1+\gamma (1+\Bar{h})/(1-\gamma)+\Bar{h}]\cdot\mathbf{1}_{\bS}=V_0$ holds by using $V_0=(1+\Bar{h})/(1-\gamma)\cdot\mathbf{1}_{\bS}$, and hence $V_0\geq\Gamma_{\pi_0}V_0\geq\Gamma_{\pi_0}^2V_0\geq\cdots\geq V^{\pi_0}\geq V^*$ holds. Now, suppose that for some $t\in\{0,1,\ldots,T-1\}$, $V_0\geq V_1\geq\cdots\geq V_t$ and $V_i\geq\Gamma_{\pi_i}V_i,i=0,1,\ldots,t$ are satisfied. It then remains to show $V_t\geq V_{t+1}$ and $V_{t+1}\geq\Gamma_{\pi_{t+1}}V_{t+1}$. To this end, for any $s\in\bS$, Lemma~\ref{lem:threepoint1} with $\pi=\pi_t$ yields
    \begin{align}
        \eta_t\left[\left\langle c(s,\cdot)+\gamma\P_sV_t,\pi_{t+1}(\cdot|s)-\pi_t(\cdot|s)\right\rangle+h(\pi_{t+1}(\cdot|s))-h(\pi_t(\cdot|s))\right]+D^{\pi_{t+1}}_{\pi_t}(s)\leq 0.
        \label{eq1:lem_monotone_deter}
    \end{align}
    Since we have $\Gamma_{\pi_t}V_t=c_{\pi_t}+\gamma\P_{\pi_t}V_t+h(\pi_t)$ and $\Gamma_{\pi_{t+1}}V_t=c_{\pi_{t+1}}+\gamma\P_{\pi_{t+1}}V_t+h(\pi_{t+1})$ by definitions in~\eqref{def:Gamma_pi_V}, and $D_{\pi_t}^{\pi_{t+1}}\geq 0$, then~\eqref{eq1:lem_monotone_deter} simplifies to $(\Gamma_{\pi_{t+1}}V_t)(s)\leq(\Gamma_{\pi_t}V_t)(s)$, and hence
    \begin{align}
        \Gamma_{\pi_{t+1}}V_t\leq\Gamma_{\pi_t}V_t.
        \label{eq2:lem_monotone_deter}
    \end{align}
    Therefore, we have
    \begin{align}
        V_t\geq\Gamma_{\pi_t}V_t\geq\Gamma_{\pi_{t+1}}V_t=V_{t+1},
        \label{eq3:lem_monotone_deter}
    \end{align}
    where the first inequality uses the inductive hypothesis, the second inequality follows from~\eqref{eq2:lem_monotone_deter}, and the equality holds by~\eqref{eq2:algo1}. Moreover, we obtain
    \begin{align}
        V_{t+1}=\Gamma_{\pi_{t+1}}V_t\geq\Gamma_{\pi_{t+1}}V_{t+1},
        \label{eq4:lem_monotone_deter}
    \end{align}
    where the equality follows from~\eqref{eq2:algo1}, and the inequality uses $V_t\geq V_{t+1}$ by~\eqref{eq3:lem_monotone_deter} and the monotonicity of $\Gamma_{\pi_{t+1}}$. Combining~\eqref{eq3:lem_monotone_deter} and~\eqref{eq4:lem_monotone_deter} completes the induction and hence concludes the proof.
\end{proof}
\vgap

The convergence of Algorithm~\ref{algo:VMD1} follows from Lemma~\ref{lem:converge_deter} and Lemma~\ref{lem:monotone_deter}.

\begin{proposition}
    For any $T\in\{0,1,\ldots\}$, taking $\eta_t=\eta>0$ for $t=0,1,\ldots,T-1$ in Algorithm~\ref{algo:VMD1} yields
    \begin{align}
        \big\|(V_T-V^*)+\tfrac{1}{\eta T}(I-\gamma\P_{\pi^*})^{-1}D_{\pi_T}^{\pi^*}\big\|_\infty \leq \tfrac{\|V_0-V^*\|_\infty}{(1-\gamma)T}+\tfrac{D_0}{\eta(1-\gamma)T}.
        \label{res:prop_bvmd}
    \end{align}
    Specifically, when $\eta=D_0(1-\gamma)/(1+\Bar{h})$ and $T \geq 2(1+\Bar{h})(1-\gamma)^{-2}\epsilon^{-1}$, we have $\|V^{\pi_T}-V^*\|_\infty \leq \epsilon$ and $\|D_{\pi_T}^{\pi^*}\|_\infty \leq 2D_0/(1-\gamma)$.
    \label{prop:bvmd}
\end{proposition}
\begin{proof}
    First, by applying the monotone operator $(I-\gamma\P_{\pi^*})^{-1}$ to~\eqref{res1:lem3p1} in Lemma~\ref{lem:converge_deter}, and using $\eta_t=\eta$ for $t=0,1,...,T-1$, we obtain
    \begin{align}
        &(I-\gamma\P_{\pi^*})^{-1}(V_T-V^*)+\tsum_{t=1}^{T-1}(V_t-V^*)+\tfrac{1}{\eta}(I-\gamma\P_{\pi^*})^{-1}D_{\pi_T}^{\pi^*} \nn\\
        &\leq \tfrac{1}{\eta}(I-\gamma\P_{\pi^*})^{-1}D_{\pi_0}^{\pi^*}+\gamma(I-\gamma\P_{\pi^*})^{-1}\P_{\pi^*}(V_0-V^*).
        \label{eq1:prop_bvmd}
    \end{align}
    By Lemma~\ref{lem:monotone_deter}, $V_0-V^*\geq V_1-V^*\geq\cdots\geq V_T-V^*\geq \mathbf{0}_{\bS}$ holds, then we have $(I-\gamma\P_{\pi^*})^{-1}(V_T-V^*)\geq V_T-V^*$, and thus~\eqref{eq1:prop_bvmd} implies
    \begin{align}
        T(V_T-V^*)+\tfrac{1}{\eta}(I-\gamma\P_{\pi^*})^{-1}D_{\pi_T}^{\pi^*} \leq \tfrac{1}{\eta}(I-\gamma\P_{\pi^*})^{-1}D_{\pi_0}^{\pi^*}+\gamma(I-\gamma\P_{\pi^*})^{-1}\P_{\pi^*}(V_0-V^*).
        \label{eq2:prop_bvmd}
    \end{align}
    Utilizing $\|(I-\gamma\P_{\pi^*})^{-1}\|_{\infty} \leq (1-\gamma)^{-1}$, $\|\P_{\pi^*}\|_\infty \leq 1$ and $\|D_{\pi_0}^{\pi^*}\|_\infty \leq D_0$ in~\eqref{eq2:prop_bvmd} yields~\eqref{res:prop_bvmd}.
    
    Finally, combining~\eqref{res:prop_bvmd} with $V_T-V^*\geq V^{\pi_T}-V^*\geq \mathbf{0}_{\bS}$ by Lemma~\ref{lem:monotone_deter}(b), $\|V_0-V^*\|_\infty \leq (1+\hat{h})/(1-\gamma)$, and the parameter choices $\eta=D_0(1-\gamma)/(1+\bar{h})$ and $T \geq 2(1+\Bar{h})(1-\gamma)^{-2}\epsilon^{-1}$, we obtain $\|V^{\pi_T}-V^*\|_\infty \leq \epsilon$ and $\|D_{\pi_T}^{\pi^*}\|_\infty \leq 2D_0/(1-\gamma)$. This completes the proof.
\end{proof}
\vgap

Proposition~\ref{prop:bvmd} establishes a sublinear convergence rate for Algorithm~\ref{algo:VMD1} and shows that the Bregman divergence between the output $\epsilon$-optimal policy $\pi_T$ and the optimal policy $\pi^*$ is bounded by $\cO((1-\gamma)^{-1})$.
\vgap

Additionally, we can see from Proposition~\ref{prop:bvmd} that,~\eqref{res:prop_bvmd} with sufficiently large $T$ and $\eta$ yields $\|V_T-V^*\|_\infty \leq \|V_0-V^*\|_\infty/2$; in particular, it suffices to take $T \geq 1/[4(1-\gamma)]$ and $\eta \geq D_0/\|V_0-V^*\|_\infty$. This observation motivates us to achieve a linear convergence rate for Algorithm~\ref{algo:VMD1} by employing a restarting strategy. Specifically, we group every $T$ iterations of Algorithm~\ref{algo:VMD1} into an epoch and use epoch-dependent stepsizes. After each epoch, the bound on the policy value gap is reduced by a factor of two.

To establish the linear convergence of Algorithm~\ref{algo:VMD1}, we consider $K$ epochs, write $\hat{\pi}_0=\pi_0$, $\hat{V}_0=V_0$, and denote the last policy iterate and value iterate produced in each epoch $k\in\{0,1,...,K-1\}$ of Algorithm~\ref{algo:VMD1} as $\hat{\pi}_{k+1}$ and $\hat{V}_{k+1}$, respectively. Moreover, we define
\begin{align}
    u_k:=\tfrac{1+\Bar{h}}{2^k(1-\gamma)} \quad\text{and}\quad D_k:=2^kD_0 \qquad\text{for}\quad k=0,1,\ldots,K.
    \label{def:uk}
\end{align}

\begin{theorem}
    Suppose that we run Algorithm~\ref{algo:VMD1} for $K$ epochs, where each epoch $k\in\{0,1,\ldots,K-1\}$ consists of $T$ iterations with stepsizes $\{\eta_{k,t}\}_{t=0}^{T-1}$. When the algorithmic parameters are set to
    \begin{align}
        K=\left\lceil\log_2\left(\tfrac{1+\Bar{h}}{(1-\gamma)\epsilon}\right)\right\rceil,\quad T=\left\lceil\tfrac{4}{1-\gamma}\right\rceil \quad\text{and}\quad \eta_{k,t}=\tfrac{D_k}{u_k}, \ t=0,1,\ldots,T-1 \quad\text{for} \quad k=0,1,\ldots,K-1,
        \label{para1:thm1_deter}
    \end{align}
    we have $\|V^{\hat{\pi}_K}-V^*\|_\infty\leq\epsilon$ and $\|D_{\hat{\pi}_K}^{\pi^*}\|_\infty \leq 2(1+\bar{h})D_0(1-\gamma)^{-2}\epsilon^{-1}$.
    \label{thm_deter}
\end{theorem}
\begin{proof}
    We first claim that for every epoch $k\in\{0,1,\ldots,K\}$,
    \begin{align}
        (a) \ \big\|V^{\hat{\pi}_k}-V^*\big\|_\infty\leq\big\|\hat{V}_k-V^*\big\|_\infty\leq u_k \qquad\text{and}\qquad (b) \ \big\|(I-\gamma\P_{\pi^*})^{-1}D_{\hat{\pi}_k}^{\pi^*}\big\|_\infty\leq\tfrac{D_k}{1-\gamma}
        \label{eq1:thm_deter}
    \end{align}
    are satisfied. Below we show this claim by induction. First, notice that~\eqref{eq1:thm_deter} holds trivially for $k=0$. Next, for any $k'\in\{0,1,...,K-1\}$, suppose that~\eqref{eq1:thm_deter} holds for $k=k'$. We then show that it also holds for $k'+1$.

    To this end, we consider epoch $k'$. For notational simplicity, let us denote the policy iterates and value iterates in epoch $k'$ as $\{\pi_t\}_{t=0}^T$ and $\{V_t\}_{t=0}^T$, respectively. Then, Lemma~\ref{lem:converge_deter} together with Lemma~\ref{lem:monotone_deter} and the parameter choice of $\{\eta_{k',t}\}_{t=0}^{T-1}$ in~\eqref{para1:thm1_deter} implies
    \begin{align}
        T(V_T-V^*)+\tfrac{u_{k'}}{D_{k'}}(I-\gamma\P_{\pi^*})^{-1}D_{\pi_T}^{\pi^*}\leq \tfrac{u_{k'}}{D_{k'}}(I-\gamma\P_{\pi^*})^{-1}D_{\pi_0}^{\pi^*}+\gamma(I-\gamma\P_{\pi^*})^{-1}\P_{\pi^*}(V_0-V^*),
        \label{eq2:thm_deter}
    \end{align}
    by following a similar procedure to the proof of Proposition~\ref{prop:bvmd}. Next, taking $\|\cdot\|_\infty$ on the right-hand side of~\eqref{eq2:thm_deter}, plugging in $\hat{V}_{k'}=V_0$, $\hat{V}_{k'+1}=V_T$, $\hat{\pi}_{k'}=\pi_0$ and $\hat{\pi}_{k'+1}=\pi_T$, and utilizing $\|(I-\gamma\P_{\pi^*})^{-1}D_{\hat{\pi}_{k'}}^{\pi^*}\|_\infty\leq D_{k'}/(1-\gamma)$ and $\|\hat{V}_{k'}-V^*\|_\infty\leq u_{k'}$ by the inductive hypothesis, we have
    \begin{align}
        T\big(\hat{V}_{k'+1}-V^*\big)+\tfrac{u_{k'}}{D_{k'}}(I-\gamma\P_{\pi^*})^{-1}D_{\hat{\pi}_{k'+1}}^{\pi^*}\leq \tfrac{u_{k'}}{(1-\gamma)}\cdot\mathbf{1}_{\bS}+\tfrac{\gamma u_{k'}}{1-\gamma}\cdot\mathbf{1}_{\bS}\leq\tfrac{2u_{k'}}{1-\gamma}\cdot\mathbf{1}_{\bS},
        \label{eq3:thm_deter}
    \end{align}
    In view of the parameter choice of $T$ in~\eqref{para1:thm1_deter}, we obtain $\|\hat{V}_{k'+1}-V^*\|_\infty\leq u_{k'}/2=u_{k'+1}$ from~\eqref{eq3:thm_deter}, which, together with $\hat{V}_{k'+1}\geq V^{\hat{\pi}_{k'+1}}\geq V^*$ by Lemma~\ref{lem:monotone_deter}, implies~\eqref{eq1:thm_deter}$(a)$ for $k=k'+1$. Moreover, since $\hat{V}_{k'+1}-V^*\geq\mathbf{0}_{\bS}$, then~\eqref{eq3:thm_deter} also yields~\eqref{eq1:thm_deter}$(b)$ for $k=k'+1$. This concludes the induction and shows the claim in~\eqref{eq1:thm_deter}. As a consequence of~\eqref{eq1:thm_deter}, we have $\|V^{\hat{\pi}_K}-V^*\|_\infty\leq u_K=u_0/2^K\leq\epsilon$ and $\|D_{\hat{\pi}_K}^{\pi^*}\|_\infty \leq D_K/(1-\gamma)=2^KD_0/(1-\gamma) \leq 2(1+\bar{h})D_0(1-\gamma)^{-2}\epsilon^{-1}$. This completes the proof.
\end{proof}
\vgap

Theorem~\ref{thm_deter} implies that Algorithm~\ref{algo:VMD1} converges linearly and requires $\Tilde{\cO}((1-\gamma)^{-1})$ prox-mapping computations of the form~\eqref{eq1:algo1} to obtain an $\epsilon$-optimal policy for a DMDP. By comparison, the classical value iteration also requires $\Tilde{\cO}((1-\gamma)^{-1})$ iterations of the form~\eqref{valueiteration} to achieve the same objective.

Observe, however, that the produced $\epsilon$-optimal policy $\hat{\pi}_K$ can be highly unstable, since the bound on $\|D_{\hat{\pi}_K}^{\pi^*}\|_\infty$ grows arbitrarily large as $\epsilon$ approaches zero. This instability arises because the bound on $\|D_{\hat{\pi}_k}^{\pi^*}\|_\infty$ increases by a factor of two after each epoch $k\in\{0,1,\ldots,K-1\}$, as shown in~\eqref{eq1:thm_deter}. In view of Proposition~\ref{prop:bvmd} and Theorem~\ref{thm_deter}, the loss of policy stability is a tradeoff for the improvement on the convergence rate from sublinear to linear.

\subsection{Value Mirror Descent under Model Misspecification}

In this subsection, we present a new algorithmic scheme that can achieve linear convergence rate while maintaining policy stability. Specifically, we consider a more general setting where the model may be misspecified; that is, the accessible probability transition kernel $\P$ may differ from the true kernel $\Bar{\P}$. We make the following assumption on the discrepancy between $\P$ and $\Bar{\P}$.

\begin{assumption}
    For any $(s,a)\in\bS\times\bA$, we have $\|\P_{s,a}-\Bar{\P}_{s,a}\|_1 \leq \tau$.
    \label{assum:kernel_inexactness}
\end{assumption}
\vgap

If $\tau=0$, then $\P=\Bar{\P}$, and we have access to the exact kernel $\Bar{\P}$. Otherwise, the value of $\tau>0$ may either be estimated or remain unknown. Specifically, for model-based methods where $\P$ is computed using independent state transition observations generated under $\Bar{\P}$, we can establish high probability guarantees for the estimation of $\tau$ via concentration inequalities. In contrast, for model-free methods that rely on an inaccurate simulator of the true environment, $\tau$ can be generally unknown.
\vgap

For any policy $\pi\in\Pi$, we define $\Bar{V}^{\pi}:\bS\to\R$ as the state-value function associated with the MDP induced by the true kernel $\Bar{P}$, denote $\Bar{\pi}^*$ as an optimal policy under $\bar{\P}$, and let $\Bar{V}^*=\Bar{V}^{\Bar{\pi}^*}$. To address the issue of policy instability, we introduce several modifications to Algorithm~\ref{algo:VMD1}.

\begin{algorithm}[H]
\caption{Value Mirror Descent (VMD) for DMDPs}
\label{algo:VMD1_modified}
\begin{algorithmic}
\State{\textbf{Input:} $\hat{\pi}_0(\cdot|s)=\mathbf{1}_{\bA}/|\bA| \ \forall s\in\bS$, $\hat{V}_0=(1+\Bar{h})/(1-\gamma)\cdot\mathbf{1}_{\bS}\in\R^{|\bS|}$, $K>0$, $T>0$, $\{\eta_{k,t}\}_{t=0}^{T-1}$ for $k=0,1,\ldots,K-1$.}
\For{$k=0,1,\ldots,K-1$}
\State{Set $V_0=\hat{V}_k$, $\pi_0=\hat{\pi}_k$, $\Tilde{\pi}_0=\hat{\pi}_0$ and $\{\eta_t\}_{t=0}^{T-1}=\{\eta_{k,t}\}_{t=0}^{T-1}$.}
\For{$t=0,1,\ldots,T-1$}
\State{For all $s\in\bS$, update}
\State{\begin{align}
&\Tilde{\pi}_{t+1}(\cdot|s)=\argmin_{\pi(\cdot|s)\in\Delta_{\bA}}\big\{\eta_t(\Gamma_{\pi}V_t)(s)+D^{\pi}_{\Tilde{\pi}_t}(s)\big\}, \label{eq1:algo1_modified} \\
&V_{t+1}(s)=\min\big\{\Tilde{V}_{t+1}(s):=(\Gamma_{\Tilde{\pi}_{t+1}}V_t)(s),V_t(s)\big\}, \label{eq2:algo1_modified} \\
&\pi_{t+1}(\cdot|s) =\begin{cases}
\Tilde{\pi}_{t+1}(\cdot|s), & \text{if } V_{t+1}(s)=\Tilde{V}_{t+1}(s);\\
\pi_t(\cdot|s), & \text{otherwise}.
\end{cases} \label{eq3:algo1_modified}
\end{align}}
\EndFor
\State{Set $\hat{V}_{k+1}=V_T$ and $\hat{\pi}_{k+1}=\pi_T$.}
\EndFor
\State{\textbf{Output:} $\hat{\pi}_K$ and $\hat{V}_K$.}
\end{algorithmic}
\end{algorithm}
\vgap

In every epoch $k\in\{0,1,\ldots,K-1\}$ of Algorithm~\ref{algo:VMD1_modified}, we introduce an additional sequence of policies $\{\Tilde{\pi}_t\}_{t=0}^{T}$. The initial policy $\Tilde{\pi}_0$ is fixed as $\hat{\pi}_0$, i.e., the uniform policy over $\bA$ for all $s\in\bS$, while for later epochs $k\in\{1,2,...,K-1\}$, the initial value iterate $V_0$ carries over the last value iterate $V_T$ from the previous epoch. As a consequence, $\Tilde{\pi}_0$ does not depend on $V_0$. Each prox-mapping step~\eqref{eq1:algo1_modified} uses $\Tilde{\pi}_t$ as the prox-mapping center to produce $\Tilde{\pi}_{t+1}$, and then~\eqref{eq2:algo1_modified} applies $\Tilde{\pi}_{t+1}$ to compute $\Tilde{V}_{t+1}$. However, in this way, the monotonicity property $\Tilde{V}_{t+1} \geq V^{\Tilde{\pi}_{t+1}}$ is no longer automatically guaranteed for every epoch $k\in\{0,1,...,K-1\}$, since $\Tilde{\pi}_0$ is not adapted to $V_0$ at the beginning of each epoch. Therefore, in~\eqref{eq2:algo1_modified}, we set $V_{t+1}(s)$ to the smaller one of $\Tilde{V}_{t+1}(s)$ and $V_t(s)$ to enforce the monotonicity $V_0\geq V_1\geq\cdots\geq V_T$. Accordingly, in~\eqref{eq3:algo1_modified}, we update $\pi_{t+1}(\cdot|s)$ to either $\Tilde{\pi}_{t+1}(\cdot|s)$ or $\pi_t(\cdot|s)$. The steps~\eqref{eq2:algo1_modified} and~\eqref{eq3:algo1_modified} help preserve the monotonicity property $V_t \geq V^{\pi_t}$ for all $t=1,2,\ldots,T$.
\vgap

Below we show the convergence of Algorithm~\ref{algo:VMD1_modified} and the stability of its output policy. Recall that $\Bar{\pi}^*$ denotes an optimal policy under $\bar{\P}$, and $\Bar{V}^*=\Bar{V}^{\Bar{\pi}^*}$ denotes the optimal policy value induced by $\bar{\P}$.

\begin{theorem}
    Under Assumption~\ref{assum:kernel_inexactness}, suppose that the algorithmic parameters in Algorithm~\ref{algo:VMD1_modified} are set to
    \begin{align}
        K=\left\lceil\log_2\left(\tfrac{1+\Bar{h}}{(1-\gamma)\epsilon}\right)\right\rceil,\quad T=\left\lceil\tfrac{4}{1-\gamma}\right\rceil \quad\text{and}\quad \eta_{k,t}=\tfrac{D_0}{u_k}, \ t=0,1,\ldots,T-1 \quad\text{for} \quad k=0,1,\ldots,K-1.
        \label{para1:thm_deter_inexact}
    \end{align}
    Then we have
    \begin{alignat}{2}
        &(a) \ \big\|V^{\hat{\pi}_K}-V^*\big\|_\infty\leq\epsilon; \qquad &&(b) \ \big\|D_{\hat{\pi}_K}^{\pi^*}\big\|_\infty \leq \tfrac{2D_0}{1-\gamma}; \nn\\
        &(c) \ \big\|\Bar{V}^{\hat{\pi}_K}-\Bar{V}^*\big\|_\infty\leq\epsilon+\tfrac{2\gamma(1+\Bar{h})\tau}{(1-\gamma)^2}; \qquad &&(d) \ \big\|D_{\hat{\pi}_K}^{\Bar{\pi}^*}\big\|_\infty \leq \tfrac{2D_0}{1-\gamma}+\tfrac{22(1+\Bar{h})D_0\tau}{(1-\gamma)^3\epsilon}.
        \label{res:thm_deter_inexact}
    \end{alignat}
    Specifically, when $\epsilon \geq 2(1+\Bar{h})(1-\gamma)^{-2}\tau$, we obtain $\|\Bar{V}^{\hat{\pi}_K}-\Bar{V}^*\|_\infty \leq 2\epsilon$ and $\|D_{\hat{\pi}_K}^{\Bar{\pi}^*}\|_\infty \leq 13D_0/(1-\gamma)$.
    \label{thm:deter_inexact}
\end{theorem}
\vgap

It can be seen from Theorem~\ref{thm:deter_inexact} that Algorithm~\ref{algo:VMD1_modified} achieves a linear convergence rate while maintaining policy stability. In particular, when $\tau=0$, it finds an $\epsilon$-optimal policy $\hat{\pi}_K$ satisfying $\|D_{\hat{\pi}_K}^{\pi^*}\|_\infty \leq \cO((1-\gamma)^{-1})$. When $\tau>0$, it attains an $\epsilon$-optimal policy $\hat{\pi}_K$ with $\epsilon = \cO((1-\gamma)^{-2}\tau)$, while preserving $\|D_{\hat{\pi}_K}^{\Bar{\pi}^*}\|_\infty \leq \cO((1-\gamma)^{-1})$.

\subsubsection{Convergence Analysis}

Theorem~\ref{thm:deter_inexact} follows a proof strategy similar to that used for Theorem~\ref{thm_deter}. Specifically, we first establish the following lemma, which is analogous to Lemma~\ref{lem:threepoint1}, Lemma~\ref{lem:converge_deter} and Lemma~\ref{lem:monotone_deter}.

\begin{lemma}
    For any epoch $k\in\{0,1,...,K-1\}$ of Algorithm~\ref{algo:VMD1_modified}, the following relations hold:
    \begin{itemize}
        \item [(a)] For any iteration $t\in\{0,1,...,T-1\}$, state $s\in\bS$ and policy $\pi(\cdot|s)\in\Delta_{\bA}$, we have
        \begin{align}
            \eta_t[\langle c(s,\cdot)+\gamma\P_sV_t,\Tilde{\pi}_{t+1}(\cdot|s)-\pi(\cdot|s)\rangle+h(\Tilde{\pi}_{t+1}(\cdot|s))-h(\pi(\cdot|s))]+D^{\Tilde{\pi}_{t+1}}_{\Tilde{\pi}_t}(s)\leq D^{\pi}_{\Tilde{\pi}_t}(s)-D^{\pi}_{\Tilde{\pi}_{t+1}}(s);
            \label{res1:lem_deter_inexact}
        \end{align}
        \item [(b)] For any iteration $t'\in\{1,2,...,T\}$, taking $\eta_t=\eta>0$ for $t=0,1,...,t'-1$ yields
        \begin{align}
            (V_{t'}-V^*)+\tsum_{t=1}^{t'-1}(I-\gamma\P_{\pi^*})(V_t-V^*)+\tfrac{1}{\eta}D_{\Tilde{\pi}_{T}}^{\pi^*}\leq \tfrac{1}{\eta}D_{\Tilde{\pi}_0}^{\pi^*}+\gamma\P_{\pi^*}(V_0-V^*);
            \label{res2:lem_deter_inexact}
        \end{align}
        \item [(c)] $V_0 \geq V_1 \geq V_2 \geq \cdots \geq V_T$ is satisfied, and $V_t\geq\Gamma_{\pi_t}V_t\geq V^{\pi_t}\geq V^*$ holds for all $t=0,1,\ldots,T$.
    \end{itemize}
    \label{lem:deter_inexact}
\end{lemma}
\begin{proof}
    First, for any $t\in\{0,1,...,T-1\}$, $s\in\bS$ and $\pi(\cdot|s)\in\Delta_{\bA}$,~\eqref{res1:lem_deter_inexact} directly follows from the prox-mapping step~\eqref{eq1:algo1_modified} and~\cite[Lemma 3.4]{lan2020first}, analogous to Lemma~\ref{lem:threepoint1}. This implies Lemma~\ref{lem:deter_inexact}(a).

    Next, we show Lemma~\ref{lem:deter_inexact}(b). For any $t\in\{0,1,...,t'-1\}$ and $s\in\bS$, by taking $\pi=\pi^*$ in~\eqref{res1:lem_deter_inexact}, we have
    \begin{align}
        \eta V_{t+1}(s) \leq \eta[\langle c(s,\cdot)+\gamma\P_sV_t,\pi^*(\cdot|s)\rangle+h(\pi^*(\cdot|s))]+D^{\pi^*}_{\Tilde{\pi}_t}(s)-D^{\pi^*}_{\Tilde{\pi}_{t+1}}(s),
        \label{eq1:lem_deter_inexact}
    \end{align}
    where~\eqref{eq2:algo1_modified}, $\eta_t=\eta$ and $D^{\Tilde{\pi}_{t+1}}_{\Tilde{\pi}_t}(s) \geq 0$ are applied. Then, by incorporating $\langle c(s,\cdot)+\gamma\P_sV_t,\pi^*(\cdot|s)\rangle+h(\pi^*(\cdot|s))=V^*(s)+\gamma (\P_{\pi^*})_s(V_t-V^*)$ (see the derivation in~\eqref{eq2:lem_converge_deter}) into~\eqref{eq1:lem_deter_inexact}, we obtain
    \begin{align}
        \eta (V_{t+1}(s)-V^*(s))\leq D^{\pi^*}_{\Tilde{\pi}_t}(s)-D^{\pi^*}_{\Tilde{\pi}_{t+1}}(s)+\eta\gamma(\P_{\pi^*})_s(V_t-V^*).
        \label{eq2:lem_deter_inexact}
    \end{align}
    Considering~\eqref{eq2:lem_deter_inexact} for all states $s\in\bS$ and summing them up for $t=0,1,...,t'-1$ yields~\eqref{res2:lem_deter_inexact}.

    Now we show Lemma~\ref{lem:deter_inexact}(c). First, observe that $V_0\geq V_1\geq\cdots\geq V_T$ by~\eqref{eq2:algo1_modified}. To further establish $V_t \geq \Gamma_{\pi_t}V_t$ for all $t=0,1,...,T$, we use mathematical induction. First, notice that $V_0 \geq \Gamma_{\pi_0}V_0$ holds for epoch $k=0$. Now, suppose that $V_t \geq \Gamma_{\pi_t}V_t$ holds for some $t\in\{0,1,\ldots,T-1\}$, and we show $V_{t+1} \geq \Gamma_{\pi_{t+1}}V_{t+1}$ by considering two cases for different states $s\in\bS$. In the first case $V_{t+1}(s)=\Tilde{V}_{t+1}(s)$, we have $V_{t+1}(s)=(\Gamma_{\pi_{t+1}}V_t)(s)\geq (\Gamma_{\pi_{t+1}}V_{t+1})(s)$ which uses $V_t\geq V_{t+1}$; otherwise, $V_{t+1}(s)=V_t(s) \geq (\Gamma_{\pi_t}V_t)(s)=(\Gamma_{\pi_{t+1}}V_t)(s) \geq (\Gamma_{\pi_{t+1}}V_{t+1})(s)$ holds by the inductive hypothesis, $\pi_{t+1}(\cdot|s)=\pi_t(\cdot|s)$ by~\eqref{eq3:algo1_modified}, and $V_t\geq V_{t+1}$. Combining the two cases completes the induction. Since $V_0$ in each epoch $k\in\{1,2,...,K-1\}$ carries over $V_T$ from the previous epoch $k-1$, then $V_t \geq \Gamma_{\pi_t}V_t$ holds for all $t=0,1,...,T$ in every epoch $k\in\{0,1,...,K-1\}$. This implies Lemma~\ref{lem:deter_inexact}(c) and completes the proof.
\end{proof}
\vgap

With the help of Lemma~\ref{lem:deter_inexact}, we are able to establish the single-epoch convergence of Algorithm~\ref{algo:VMD1_modified}.

\begin{proposition}
    For any epoch $k\in\{0,1,\ldots,K-1\}$ of Algorithm~\ref{algo:VMD1}, if $\|\hat{V}_k-V^*\|_\infty\leq u_k$, then taking $T=\lceil 4/(1-\gamma)\rceil$ and $\eta_t=D_0/u_k$ for $t=0,1,\ldots,T-1$ yields
    \begin{align}
        (a) \ \big\|V^{\hat{\pi}_{k+1}}-V^*\big\|_\infty\leq\big\|\hat{V}_{k+1}-V^*\big\|_\infty\leq u_{k+1}=\tfrac{u_k}{2} \qquad \text{and} \qquad (b) \ \big\|D_{\Tilde{\pi}_t}^{\pi^*}\big\|_\infty\leq\tfrac{2D_0}{1-\gamma} \quad \text{for all} \ t=1,2,...,T.
        \label{res:prop_deter_inexact}
    \end{align}
    \label{prop:deter_inexact}
\end{proposition}
\begin{proof}
    First, for any $t'\in\{1,2,...,T\}$, combining Lemma~\ref{lem:deter_inexact}(b), Lemma~\ref{lem:deter_inexact}(c) and the parameter choice $\eta_t=D_0/u_k$ for $t=0,1,...,t'-1$, we obtain
    \begin{align}
        t'(V_{t'}-V^*)+\tfrac{u_k}{D_0}(I-\gamma\P_{\pi^*})^{-1}D_{\Tilde{\pi}_{t'}}^{\pi^*}\leq \tfrac{u_k}{D_0}(I-\gamma\P_{\pi^*})^{-1}D_{\Tilde{\pi}_0}^{\pi^*}+\gamma(I-\gamma\P_{\pi^*})^{-1}\P_{\pi^*}(V_0-V^*).
        \label{eq1:prop_deter_inexact}
    \end{align}
    Then,~\eqref{eq1:prop_deter_inexact} together with $\|D_{\Tilde{\pi}_0}^{\pi^*}\|_\infty=\|D_{\hat{\pi}_0}^{\pi^*}\|_\infty \leq D_0$ and $\|V_0-V^*\|_\infty=\|\hat{V}_k-V^*\|_\infty \leq u_k$ implies
    \begin{align}
        \big\|t'(V_{t'}-V^*)+\tfrac{u_k}{D_0}(I-\gamma\P_{\pi^*})^{-1}D_{\Tilde{\pi}_{t'}}^{\pi^*}\big\|_\infty \leq \tfrac{u_k}{1-\gamma}+\tfrac{u_k}{1-\gamma}=\tfrac{2u_k}{1-\gamma}.
        \label{eq2:prop_deter_inexact}
    \end{align}
    Taking $t'=T=\lceil 4/(1-\gamma) \rceil$ in~\eqref{eq2:prop_deter_inexact}, and utilizing $V_T=\hat{V}_{k+1}$ and $D_{\Tilde{\pi}_{t'}}^{\pi^*} \geq \mathbf{0}_{\bS}$, we obtain $\|\hat{V}_{k+1}-V^*\|_\infty \leq u_k/2=u_{k+1}$. Combining this bound with Lemma~\ref{lem:deter_inexact}(c) and $\hat{\pi}_{k+1}=\pi_T$ yields~\eqref{res:prop_deter_inexact}(a). Moreover, since $V_{t'}-V^* \geq \mathbf{0}_{\bS}$ by Lemma~\ref{lem:deter_inexact}(c),~\eqref{eq2:prop_deter_inexact} also implies $\|D_{\Tilde{\pi}_{t'}}^{\pi^*}\|_\infty \leq \|(I-\gamma\P_{\pi^*})^{-1}D_{\Tilde{\pi}_{t'}}^{\pi^*}\|_\infty \leq 2D_0/(1-\gamma)$, and hence~\eqref{res:prop_deter_inexact}(b) holds. This completes the proof.
\end{proof}
\vgap

Proposition~\ref{prop:deter_inexact} yields a result analogous to~\eqref{eq1:thm_deter} in the proof of Theorem~\ref{thm_deter}, in that the bound on the policy value gap is reduced by half after each epoch $k\in\{0,1,...,K-1\}$ of Algorithm~\ref{algo:VMD1_modified}. However,~\eqref{eq1:thm_deter} shows that the bound on $\|D_{\hat{\pi}_k}^{\pi^*}\|_\infty$ increases by a factor of two after each epoch $k$ of Algorithm~\ref{algo:VMD1}, whereas Proposition~\ref{prop:deter_inexact} implies that $\|D_{\Tilde{\pi}_t}^{\pi^*}\|_\infty$ remains bounded by the constant $2D_0/(1-\gamma)$ for every epoch $k\in\{0,1,...,K-1\}$ and iteration $t\in\{1,2,...,T\}$ of Algorithm~\ref{algo:VMD1_modified}.

Built upon these properties, we show the proof of Theorem~\ref{thm:deter_inexact}.

\paragraph{Proof of Theorem~\ref{thm:deter_inexact}}

We first show~\eqref{res:thm_deter_inexact}(a) and~\eqref{res:thm_deter_inexact}(b), which directly follow from Proposition~\ref{prop:deter_inexact}. For~\eqref{res:thm_deter_inexact}(a), first notice that $\|\hat{V}_0-V^*\|_\infty \leq u_0$. Then Proposition~\ref{prop:deter_inexact} inductively implies $\|\hat{V}_k-V^*\|_\infty \leq u_k$ for all $k=0,1,...,K$ and yields~\eqref{res:thm_deter_inexact}(a), i.e., $\|\hat{V}_K-V^*\|_\infty \leq u_K=u_0/2^K \leq \epsilon$, by applying the parameter choice of $K$ in~\eqref{para1:thm_deter_inexact}. Furthermore, in view of~\eqref{eq3:algo1_modified}, for every $s\in\bS$, either $\hat{\pi}_K(\cdot|s)=\hat{\pi}_0(\cdot|s)$ or $\hat{\pi}_K(\cdot|s)=\Tilde{\pi}_{t'}(\cdot|s)$ for some epoch $k\in\{0,1,\ldots,K-1\}$ and iteration $t'\in\{1,2,\ldots,T\}$. This property together with Proposition~\ref{prop:deter_inexact} yields~\eqref{res:thm_deter_inexact}(b).

Next, we show~\eqref{res:thm_deter_inexact}(c), by bounding $\|V^{\hat{\pi}_K}-\Bar{V}^{\hat{\pi}_K}\|_\infty$ and $\|V^*-\Bar{V}^*\|_\infty$, respectively. For any $\pi\in\Pi$, using $V^{\pi}=c_{\pi}+\gamma\P_{\pi}V^{\pi}+h(\pi)$ and $\Bar{V}^{\pi}=c_{\pi}+\gamma\Bar{\P}_{\pi}\Bar{V}^{\pi}+h(\pi)$, we obtain
\begin{align}
    V^{\pi}-\Bar{V}^{\pi} = \gamma(\P_{\pi}V^{\pi}-\Bar{\P}_{\pi}\Bar{V}^{\pi}) = \gamma\P_{\pi}(V^{\pi}-\Bar{V}^{\pi})+\gamma(\P_{\pi}-\Bar{\P}_{\pi})\Bar{V}^{\pi},
    \label{eq1:thm_deter_inexact}
\end{align}
which, combined with Assumption~\ref{assum:kernel_inexactness} and $\|\Bar{V}^{\pi}\|_\infty \leq (1+\Bar{h})/(1-\gamma)$, implies
\begin{align}
    \big\|V^{\pi}-\Bar{V}^{\pi}\big\|_\infty = \big\|\gamma(I-\gamma\P_{\pi})^{-1}(\P_{\pi}-\Bar{\P}_{\pi})\Bar{V}^{\pi}\big\|_\infty \leq \tfrac{\gamma(1+\Bar{h})\tau}{(1-\gamma)^2}.
    \label{eq2:thm_deter_inexact}
\end{align}
Moreover,~\eqref{eq2:thm_deter_inexact} together with $V^{\pi^*} \leq V^{\Bar{\pi}^*}$ and $\Bar{V}^{\Bar{\pi}^*} \leq \Bar{V}^{\pi^*}$ yields
\begin{align}
    &\big\|V^*-\Bar{V}^*\big\|_\infty = \big\|V^{\pi^*}-\Bar{V}^{\Bar{\pi}^*}\big\|_\infty = \max_{s\in\bS}\Big\{\max\big\{V^{\pi^*}(s)-\Bar{V}^{\Bar{\pi}^*}(s),\Bar{V}^{\Bar{\pi}^*}(s)-V^{\pi^*}(s)\big\}\Big\} \nn\\
    &\leq \max_{s\in\bS}\Big\{\max\big\{V^{\bar{\pi}^*}(s)-\Bar{V}^{\bar{\pi}^*}(s),\Bar{V}^{\pi^*}(s)-V^{\pi^*}(s)\big\}\Big\} \leq \max\big\{\big\|V^{\Bar{\pi}^*}-\Bar{V}^{\Bar{\pi}^*}\big\|_\infty, \big\|V^{\pi^*}-\Bar{V}^{\pi^*}\big\|_\infty\big\} \leq \tfrac{\gamma(1+\Bar{h})\tau}{(1-\gamma)^2}.
    \label{eq3:thm_deter_inexact}
\end{align}
Combining~\eqref{res:thm_deter_inexact}(a),~\eqref{eq2:thm_deter_inexact} with $\pi=\hat{\pi}_K$ and~\eqref{eq3:thm_deter_inexact}, we obtain~\eqref{res:thm_deter_inexact}(c).

Finally, we show~\eqref{res:thm_deter_inexact}(d). In any epoch $k\in\{0,1,\ldots,K-1\}$ of Algorithm~\ref{algo:VMD1_modified}, for any iteration $t\in\{0,1,\ldots,T-1\}$ and state $s\in\bS$, according to Lemma~\ref{lem:deter_inexact}(a) with $\pi=\bar{\pi}^*$, we have
\begin{align}
    \eta_t\left[\left\langle c(s,\cdot)+\gamma\P_sV_t,\Tilde{\pi}_{t+1}(\cdot|s)-\Bar{\pi}^*(\cdot|s)\right\rangle+h(\Tilde{\pi}_{t+1}(\cdot|s))-h(\Bar{\pi}^*(\cdot|s)\right]+D^{\Tilde{\pi}_{t+1}}_{\Tilde{\pi}_t}(s)\leq D^{\Bar{\pi}^*}_{\Tilde{\pi}_t}(s)-D^{\Bar{\pi}^*}_{\Tilde{\pi}_{t+1}}(s).
    \label{eq4:thm_deter_inexact}
\end{align}
Then, utilizing $\Bar{V}^*=\langle c(s,\cdot)+\gamma\Bar{\P}_s\Bar{V}^*,\Bar{\pi}^*(\cdot|s)\rangle+h(\Bar{\pi}^*(\cdot|s))$, we obtain
\begin{align}
    &\big\langle c(s,\cdot)+\gamma\P_sV_t,\Bar{\pi}^*(\cdot|s)\big\rangle+h(\Bar{\pi}^*(\cdot|s)) = \Bar{V}^*(s)+\gamma (\P_{\Bar{\pi}^*}V_t-\Bar{\P}_{\Bar{\pi}^*}\Bar{V}^*)(s) \nn\\
    &= \Bar{V}^*(s)+\gamma(\P_{\Bar{\pi}^*})_s(V_t-\Bar{V}^*)+\gamma(\P_{\Bar{\pi}^*}-\Bar{\P}_{\Bar{\pi}^*})_s\Bar{V}^*.
    \label{eq5:thm_deter_inexact}
\end{align}
Combining~\eqref{eq1:algo1_modified},~\eqref{eq4:thm_deter_inexact},~\eqref{eq5:thm_deter_inexact} and $\eta_t=D_0/u_k$ yields
\begin{align}
    \tfrac{D_0}{u_k}(V_{t+1}(s)-\Bar{V}^*(s))\leq D^{\Bar{\pi}^*}_{\Tilde{\pi}_t}(s)-D^{\Bar{\pi}^*}_{\Tilde{\pi}_{t+1}}(s)+\tfrac{D_0}{u_k}\gamma(\P_{\Bar{\pi}^*})_s(V_t-\Bar{V}^*)+\tfrac{D_0}{u_k}\gamma(\P_{\Bar{\pi}^*}-\Bar{\P}_{\Bar{\pi}^*})_s\Bar{V}^*.
    \label{eq6:thm_deter_inexact}
\end{align}
For any $t'\in\{1,\ldots,T\}$, we consider~\eqref{eq6:thm_deter_inexact} for all $s\in\bS$, sum them up for $t=0,1,\ldots,t'-1$ and have
\begin{align}
    D^{\Bar{\pi}^*}_{\Tilde{\pi}_{t'}}+\tfrac{D_0}{u_k}(V_{t'}-\Bar{V}^*)+\tfrac{D_0}{u_k}\tsum_{t=1}^{t'-1}(I-\gamma\P_{\Bar{\pi}^*})(V_t-\Bar{V}^*) \leq D^{\Bar{\pi}^*}_{\Tilde{\pi}_0}+\tfrac{D_0}{u_k}(V_0-\Bar{V}^*)+\tfrac{t'D_0}{u_k}\gamma(\P_{\Bar{\pi}^*}-\Bar{\P}_{\Bar{\pi}^*})\Bar{V}^*.
    \label{eq7:thm_deter_inexact}
\end{align}
In view of Lemma~\ref{lem:deter_inexact}(c) and~\eqref{eq3:thm_deter_inexact}, $V_t-\Bar{V}^*=(V_t-V^*)+(V^*-\Bar{V}^*)\geq -\tau'\cdot\mathbf{1}_{\bS}$ holds for all $t=1,2,\ldots,t'$, where we denote $\tau'=\gamma(1+\Bar{h})\tau/(1-\gamma)^2$. Moreover, $V_0-\Bar{V}^* \leq (V_0-V^*)+(V^*-\Bar{V}^*) \leq (u_k+\tau')\cdot\mathbf{1}_{\bS}$ follows from~\eqref{eq3:thm_deter_inexact}. Using these results and Assumption~\ref{assum:kernel_inexactness}, and applying the monotone operator $(I-\gamma\P_{\Bar{\pi}^*})^{-1}$ to both sides of~\eqref{eq7:thm_deter_inexact}, we obtain
\begin{align}
    &\big\|D^{\Bar{\pi}^*}_{\Tilde{\pi}_{t'}}\big\|_\infty \leq \big\|(I-\gamma\P_{\Bar{\pi}^*})^{-1}D^{\Bar{\pi}^*}_{\Tilde{\pi}_{t'}}\big\|_\infty \leq \big\|(I-\gamma\P_{\Bar{\pi}^*})^{-1}D^{\Bar{\pi}^*}_{\Tilde{\pi}_0}\big\|_\infty+\tfrac{D_0}{u_k}\Big[\tfrac{u_k}{1-\gamma}+\big(\tfrac{2}{1-\gamma}+t'-1\big)\tau'+t'\gamma\tfrac{(1+\Bar{h})\tau}{1-\gamma}\Big] \nn\\
    &\leq \tfrac{2D_0}{1-\gamma}+\tfrac{22(1+\Bar{h})D_0\tau}{(1-\gamma)^3\epsilon}
    \label{eq9:thm_deter_inexact}
\end{align}
where the third inequality utilizes $\|D^{\Bar{\pi}^*}_{\Tilde{\pi}_0}\|_\infty \leq D_0$, the parameter choice of $T$ in~\eqref{para1:thm_deter_inexact}, $t' \leq T$, and $u_k \geq u_K \geq \epsilon/2$ by the definition of $u_k$ in~\eqref{def:uk} and the choice of $K$ in~\eqref{para1:thm_deter_inexact}. This implies~\eqref{res:thm_deter_inexact}(d) and completes the proof.
\vgap
\section{Stochastic Value Mirror Descent under General Convex Regularizers}
\label{section:stoc1}

In this section, we extend value mirror descent to the stochastic setting, where direct access to the probability transition kernel $\P$ is unavailable. Instead, we assume the access to a generative sampling model, which allows us to generate multiple independent one-step state transitions $(s_1,s_2,...)$ from any given state-action pair $(s,a)\in\bS\times\bA$, and thereby construct approximations of $\P_{s,a}$. Each query to the generative model, i.e., each independent one-step state transition generated from a state-action pair $(s,a)\in\bS\times\bA$, is counted as one observation. Built on Algorithm~\ref{algo:VMD1}, we develop a stochastic value mirror descent method, analyze its convergence properties, and provide its sample complexity guarantees, i.e., the total number of observations of state transitions.

\subsection{Algorithm and Convergence Result}

We begin by highlighting the main differences between the stochastic value mirror descent (SVMD) method and its deterministic counterpart, VMD, given in Algorithm~\ref{algo:VMD1_modified}. First, in the stochastic setting, we construct an estimator $\hat{\P}^{(m)}$ of the transition kernel $\P$ using $m$ independent observations drawn from the generative model, and use $\hat{\P}^{(m)}$ in place of $\P$ in each prox-mapping step of the algorithm. Specifically, for each state-action pair $(s,a)\in\bS\times\bA$, suppose that $m$ independent observations $s_1,s_2,...,s_m$ of one-step state transition starting from $(s,a)$ are generated, then we compute
\begin{align}
    \hat{\P}^{(m)}_{s,a}(s')=\tfrac{1}{m}\tsum_{i=1}^m \mathbb{I}_{\left\{s_i=s'\right\}} \quad \forall s'\in\bS,
    \label{approx_trans_kernel}
\end{align}
where $\mathbb{I}_A$ returns $1$ if the event $A$ is true, and returns $0$ otherwise.

As shown in Algorithm~\ref{algo:VMD2}, we still group the iterations into $K$ epochs, with each epoch consisting of $T$ iterations. For any epoch $k\in\{0,1,...,K-1\}$, we compute $\Tilde{\P}^tV_t$ as an unbiased estimation of $\P V_t$ in each iteration $t\in\{0,1,...,T-1\}$. In particular, we employ a variance reduction technique by estimating $\P V_0$ in the first iteration $t=0$, and then approximating $\P (V_t-V_0)$ in the subsequent iterations $t=1,2,...,T-1$, since the latter quantity may have smaller variance than $\P V_t$. For each state-action pair $(s,a)\in\bS\times\bA$, computing $\Tilde{\P}^0V_0$ requires $m_{k,1}$ observations of state transition, while approximating $\Tilde{\P}^t(V_t-V_0)$ requires $m_{k,2}$ additional observations in each iteration $t\in\{1,2,...,T-1\}$. This variance reduction technique is also used in~\cite{sidford2018near} and plays an important role in improving the sample complexity. The estimator $\Tilde{\P}^tV_t$ is then utilized in the prox-mapping step~\eqref{eq1:algoVMD2} and step~\eqref{eq2:algoVMD2} to compute $\Tilde{\pi}_{t+1}$ and $\Tilde{V}_{t+1}$, respectively. Subsequently, in~\eqref{eq2:algoVMD2}, we update $V_{k+1}(s)$ for every state $s\in\bS$ by taking the smaller one of $\Tilde{V}_{k+1}(s)$ and $V_t(s)$. These component-wise comparisons preserve the monotonicity $V_0\geq V_1\geq\cdots\geq V_T$. The policy $\pi_{t+1}$ is updated accordingly, as shown in step~\eqref{eq3:algoVMD2}.

Due to the inexact evaluations of $\P V_t$, the monotonicity property can only be satisfied approximately, i.e., $V_t\geq\Gamma_{\pi_t}V_t-\beta$ holds with some noise $\beta\in\R^{|\bS|}$. As a result, we further have
\begin{align}
    V_t\geq\Gamma_{\pi_t}V_t-\beta\geq\Gamma_{\pi_t}^2V_t-(I+\gamma\P_{\pi_t})\beta\geq\cdots\geq V^{\pi_t}-\tsum_{i=0}^{\infty}(\gamma\P_{\pi_t})^i\beta=V^{\pi_t}-(I-\gamma\P_{\pi_t})^{-1}\beta,
    \label{eq:approxmonotonebound}
\end{align}
indicating that $V_t$ is only an approximate upper bound on $V^{\pi_t}$, rather than an exact one. We will show that the bounds on $\beta$ must decrease exponentially across epochs. Therefore, in each of the later epochs $k=1,2,...,K-1$, we compute the initial value estimator $V_0$ as a more accurate approximation of $V^{\hat{\pi}_k}$ by solving $V_0=\langle c+\gamma\Tilde{\P}^0 V_0,\pi_0\rangle+h(\pi_0)$ using a  more accurate kernel estimator $\Tilde{\P}^0$, rather than directly inheriting $V_{T_{k-1}}$ from the previous epoch $k-1$.
This linear system can be solved efficiently, for example, by using the optimal operator extrapolation method in \cite{KotsalisLanLi2022-1,LiLanPan2023}, and does not require any additional sampling once $\Tilde{\P}^0$ is defined.

It is worth noting that, in order to further guarantee the exact monotonicity property $V_t\geq\Gamma_{\pi_t}V_t\geq V^{\pi_t}$ rather than the approximate one in~\eqref{eq:approxmonotonebound}, we can alternatively compute and add a small bias term to each component of $\Tilde{\P}^tV_t$ before performing~\eqref{eq1:algoVMD2}. A more detailed analysis of this approach is provided in~\cite{sidford2018near}. In this way, at the beginning of each epoch $k\in\{1,2,...,K-1\}$, we can carry over $V_T$ from the previous epoch as $V_0$, thereby avoiding the need to approximate $V^{\hat{\pi}_k}$. Nevertheless, the bias terms added to $\Tilde{\P}^tV_t$ may be conservative in practice, potentially slowing down the convergence of the algorithm.

\begin{algorithm}[H]
\caption{Stochastic Value Mirror Descent (SVMD)}
\label{algo:VMD2}
\begin{algorithmic}
\State{\textbf{Input:} $\hat{\pi}_0(\cdot|s)=\mathbf{1}_{\bA}/|\bA| \ \forall s\in\bS$, $V_0=(1+\Bar{h})/(1-\gamma)\cdot\mathbf{1}_{\bS}\in\R^{|\bS|}$, $K>0$, $T>0$, $\{\eta_{k,t}\}_{t=0}^{T-1}$ for $k=0,1,...,K-1$, $\{m_{k,1},m_{k,2}\}_{k=0}^{K-1}>0$.}
\For{$k=0,1,...,K-1$}
\State{Set $\pi_0=\hat{\pi}_k,\Tilde{\pi}_0=\hat{\pi}_0$, $\{\eta_t\}_{t=0}^{T-1}=\{\eta_{k,t}\}_{t=0}^{T-1}$, and compute $\Tilde{\P}^0=\hat{\P}^{(m_{k,1})}$ for each $(s,a)\in\bS\times\bA$.}
\State{If $k>0$, then compute $V_0=(I-\gamma\Tilde{\P}^0_{\pi_0})^{-1}(c_{\pi_0}+h(\pi_0))$.}
\For{$t=0,1,...,T-1$}
\State{If $t>0$, then compute $\Tilde{\P}^tV_t\in\R^{|\bS|\times|\bA|}$ by $(\Tilde{\P}^tV_t)_{s,a}=(\Tilde{\P}^0V_0)_{s,a}+\hat{\P}_{s,a}^{(m_{k,2})\tr}(V_t-V_0)$.}
\State{For all $s\in\bS$, update}
\State{\begin{align}
    &\Tilde{\pi}_{t+1}(\cdot|s)=\argmin_{\pi(\cdot|s)\in\Delta_{\bA}}\Big\{\eta_t\big[\big\langle c(s,\cdot)+\gamma\Tilde{\P}^t_s V_t,\pi(\cdot|s)\big\rangle+h(\pi(\cdot|s))\big]+D^{\pi}_{\Tilde{\pi}_t}(s)\Big\}. \label{eq1:algoVMD2}\\
    &V_{t+1}(s)=\min\Big\{\Tilde{V}_{t+1}(s):=\big\langle c(s,\cdot)+\gamma\Tilde{\P}^t_sV_t,\Tilde{\pi}_{t+1}(\cdot|s)\big\rangle+h(\Tilde{\pi}_{t+1}(\cdot|s)),V_t(s)\Big\}. \label{eq2:algoVMD2} \\
    & \pi_{t+1}(\cdot|s) =\begin{cases}
    \Tilde{\pi}_{t+1}(\cdot|s), & \text{if } V_{t+1}(s)=\Tilde{V}_{t+1}(s);\\
    \pi_t(\cdot|s), & \text{otherwise}.
    \end{cases}  
    \label{eq3:algoVMD2}
\end{align}}
\EndFor
\State{Set $\hat{\pi}_{k+1}=\pi_T$.}
\EndFor
\State{\textbf{Output:} $\hat{\pi}_K$.}
\end{algorithmic}
\end{algorithm}
\vgap

Below we establish the convergence of Algorithm~\ref{algo:VMD2}. As in Section~\ref{section:deter}, we assume that $\|D_{\hat{\pi}_0}^{\pi}\|_\infty\leq D_0$ for all $\pi\in\Pi$, and we continue to use the notation $u_k$ defined in~\eqref{def:uk}.

\begin{theorem}
    For any $\delta\in(0,1)$, suppose that the algorithmic parameters in Algorithm~\ref{algo:VMD2} are set to
    \begin{align}
        &K=\left\lceil\log_2\left(\tfrac{1+\Bar{h}}{(1-\gamma)\epsilon}\right)\right\rceil,\quad T=\left\lceil\tfrac{9}{1-\gamma}\right\rceil, \quad \eta_{k,t}=\tfrac{D_0}{u_k} \quad \text{for} \ t=0,1,...,T-1, \nn\\
        &m_{k,1}\geq\tfrac{[200\sqrt{2}(1+\Bar{h})]^2\ln(12K|\bS||\bA|\delta^{-1})}{(1-\gamma)^{3}\min\{u_k,1\}^{2}}\quad\text{and}\quad m_{k,2}\geq\tfrac{(80\sqrt{2})^2\ln[4K(T-1)|\bS||\bA|\delta^{-1}]}{(1-\gamma)^{2}} \quad\text{for} \ k=0,1,...,K-1.
        \label{param:thm_stoc1}
    \end{align}
    Then with probability at least $1-\delta$, $\hat{\pi}_K$ is an $\epsilon$-optimal policy, and $\|D_{\hat{\pi}_K}^{\pi^*}\|_\infty\leq 5D_0/(1-\gamma)$.
    \label{thm:stoc1}
\end{theorem}
\vgap

In view of Theorem~\ref{thm:stoc1}, the sample complexity for Algorithm~\ref{algo:VMD2} applied to DMDP with a general convex regularizer $h$ is bounded by
\begin{align}
    \cO\Big(|\bS||\bA|\tsum_{k=0}^{K-1}\big[m_{k,1}+(T-1)m_{k,2}\big]\Big)=\Tilde{\cO}\big(|\bS||\bA|(1-\gamma)^{-3}\epsilon^{-2}\big),
    \label{complexity:stoc1}
\end{align}
which matches the existing lower bound up to logarithmic factors.

Moreover, compared with the variance-reduced Q-value iteration in~\cite{sidford2018near}, Algorithm~\ref{algo:VMD2} guarantees $\|D_{\hat{\pi}_K}^{\pi^*}\|_\infty\leq\cO((1-\gamma)^{-1})$, which, under specific choices of Bregman divergences such as KL or Tsallis divergence, implies that the output policy $\hat{\pi}_K(\cdot|s)$ assigns positive probability to the optimal action at every state $s\in\bS$. As a consequence, the policies generated by Algorithm~\ref{algo:VMD2} are stable and may serve as a good starting point for continual online and on-policy learning using various policy optimization methods.
\vgap

Below we also establish the convergence of Algorithm~\ref{algo:VMD2} in the model misspecification setting, where $\P$ is an inexact approximation of the true probability transition kernel $\Bar{\P}$.

\begin{corollary}
    Under Assumption~\ref{assum:kernel_inexactness}, for any $\delta\in(0,1)$, suppose that the algorithmic parameters in Algorithm~\ref{algo:VMD2} are set to~\eqref{param:thm_stoc1}, then with probability at least $1-\delta$, we have
    \begin{alignat}{2}
        &(a) \ \big\|V^{\hat{\pi}_K}-V^*\big\|_\infty\leq\epsilon; \qquad &&(b) \ \big\|D_{\hat{\pi}_K}^{\pi^*}\big\|_\infty \leq \tfrac{5D_0}{1-\gamma}; \nn\\
        &(c) \ \big\|\Bar{V}^{\hat{\pi}_K}-\Bar{V}^*\big\|_\infty\leq\epsilon+\tfrac{2\gamma(1+\Bar{h})\tau}{(1-\gamma)^2}; \qquad &&(d) \ \big\|D_{\hat{\pi}_K}^{\Bar{\pi}^*}\big\|_\infty \leq \tfrac{5D_0}{1-\gamma}+\tfrac{44(1+\Bar{h})D_0\tau}{(1-\gamma)^3\epsilon}.
        \label{res:cor_stoc1}
    \end{alignat}
    Specifically, when $\epsilon \geq 2(1+\Bar{h})(1-\gamma)^{-2}\tau$, we obtain $\|\Bar{V}^{\hat{\pi}_K}-\Bar{V}^*\|_\infty \leq 2\epsilon$ and $\|D_{\hat{\pi}_K}^{\Bar{\pi}^*}\|_\infty \leq 27D_0/(1-\gamma)$.
    \label{cor:stoc1}
\end{corollary}
\vgap

Corollary~\ref{cor:stoc1} implies that when $\tau=0$, with high probability, Algorithm~\ref{algo:VMD2} finds an $\epsilon$-optimal policy with $\|D_{\hat{\pi}_K}^{\pi^*}\|_\infty \leq \cO((1-\gamma)^{-1})$. When $\tau>0$, with high probability, it attains an $\epsilon$-optimal policy with $\epsilon=\cO(\tau(1-\gamma)^{-2})$, while maintaining $\|D_{\hat{\pi}_K}^{\Bar{\pi}^*}\|_\infty \leq \cO((1-\gamma)^{-1})$.

\subsection{Convergence Analysis}

We now establish Theorem~\ref{thm:stoc1}. Our analysis proceeds by first deriving a single-epoch convergence result, and then we show that the approximate monotonicity error remains sufficiently controlled with high probability. Finally, by combining these bounds across epochs, we obtain the overall convergence guarantee.

Let us begin a stochastic counterpart of Lemma~\ref{lem:threepoint1} and Lemma~\ref{lem:deter_inexact}(a).

\begin{lemma}
    In any epoch $k\in\{0,1,...,K-1\}$ of Algorithm~\ref{algo:VMD2}, for any iteration $t\in\{0,1,...,T-1\}$, state $s\in\bS$ and policy $\pi(\cdot|s)\in\Delta_{\bA}$, we have
    \begin{align*}
        \eta_t\big[\big\langle c(s,\cdot)+\gamma\Tilde{\P}^t_sV_t,\Tilde{\pi}_{t+1}(\cdot|s)-\pi(\cdot|s)\big\rangle+h(\Tilde{\pi}_{t+1}(\cdot|s))-h(\pi(\cdot|s))\big]+D^{\Tilde{\pi}_{t+1}}_{\Tilde{\pi}_t}(s)\leq D^{\pi}_{\Tilde{\pi}_t}(s)-D^{\pi}_{\Tilde{\pi}_{t+1}}(s).
    \end{align*}
    \label{lem:threepoint2}
\end{lemma}
\begin{proof}
    Similar to the proof of Lemma~\ref{lem:threepoint1}, this proof also directly utilizes~\cite[Lemma 3.4]{lan2020first}.
\end{proof}
\vgap

Built upon Lemma~\ref{lem:threepoint2}, we derive a single-epoch result below. It plays the same role as Lemma~\ref{lem:converge_deter} and Lemma~\ref{lem:deter_inexact}(b) in the deterministic setting, but now includes additional error terms arising from the stochastic setting.

To facilitate our analysis, we introduce several notations. For any epoch $k\in\{0,1,...,K-1\}$ of Algorithm~\ref{algo:VMD2}, we define the matrices $Q^t,\Tilde{Q}^t\in\R^{|\bS|\times|\bA|},t=0,1,...,T-1$, with their row vectors $Q^t(s,\cdot):=c(s,\cdot)+\gamma\P_sV_t\in\R^{|\bA|}$ and $\Tilde{Q}^t(s,\cdot):=c(s,\cdot)+\gamma\Tilde{\P}_s^tV_t\in\R^{|\bA|}$ for all $s\in\bS$. For any matrix $Q\in\R^{|\bS|\times|\bA|}$ with row vectors $Q(s,\cdot)\in\R^{|\bA|}$ and any policy $\pi$, we define $\langle Q,\pi\rangle\in\R^{|\bS|}$ such that each of its entries $(\langle Q,\pi\rangle)(s):=\langle Q(s,\cdot),\pi(\cdot|s)\rangle$.

\begin{lemma}
    In any epoch $k\in\{0,1,...,K-1\}$ of Algorithm~\ref{algo:VMD2}, for any $\Bar{t}\in\{0,1,...,T-1\}$, taking $\eta_t=\eta>0$ for $t=0,1,...,\Bar{t}$, we have
    \begin{align}
        (V_{\Bar{t}+1}-V^*)+\tsum_{t=1}^{\Bar{t}}(I-\gamma\P_{\pi^*})(V_t-V^*)+\tfrac{1}{\eta}D_{\Tilde{\pi}_{\Bar{t}+1}}^{\pi^*}\leq \tfrac{1}{\eta}D_{\Tilde{\pi}_0}^{\pi^*}+\gamma\P_{\pi^*}(V_0-V^*)+\tsum_{t=0}^{\Bar{t}}\big\langle\Tilde{Q}^t-Q^t,\pi^*\big\rangle.
        \label{res:lem3p2}
    \end{align}
    \label{lem:converge_stoc1}
\end{lemma}
\begin{proof}
    For any $t\in\{0,1,...,T-1\}$ and $s\in\bS$, setting $\pi=\pi^*$ and $\eta_t=\eta$ in Lemma~\ref{lem:threepoint2}, we have
    \begin{align}
        \eta\big[\big\langle c(s,\cdot)+\gamma\Tilde{\P}_sV_t,\Tilde{\pi}_{t+1}(\cdot|s)-\pi^*(\cdot|s)\big\rangle+h(\Tilde{\pi}_{t+1}(\cdot|s))-h(\pi^*(\cdot|s))\big]+D^{\Tilde{\pi}_{t+1}}_{\Tilde{\pi}_t}(s)\leq D^{\pi^*}_{\Tilde{\pi}_t}(s)-D^{\pi^*}_{\Tilde{\pi}_{t+1}}(s).
        \label{eq1:lem_converge_stoc1}
    \end{align}
    Since $V_{t+1}(s)\leq\langle c(s,\cdot)+\gamma\Tilde{\P}_s^tV_t,\Tilde{\pi}_{t+1}(\cdot|s)\rangle+h(\Tilde{\pi}_{t+1}(\cdot|s))$ is satisfied by~\eqref{eq1:algoVMD2}, and $\langle c(s,\cdot)+\gamma\Tilde{\P}_s^tV_t,\pi^*(\cdot|s)\rangle=\langle \Tilde{Q}^t(s,\cdot),\pi^*(\cdot|s)\rangle$ holds by definition, then~\eqref{eq1:lem_converge_stoc1} becomes
    \begin{align}
        \eta V_{t+1}(s)+D^{\Tilde \pi_{t+1}}_{\Tilde{\pi}_t}(s)\leq D^{\pi^*}_{\Tilde{\pi}_t}(s)-D^{\pi^*}_{\Tilde{\pi}_{t+1}}(s)+\eta\big[\big\langle \Tilde{Q}^t(s,\cdot),\pi^*(\cdot|s)\big\rangle+h(\pi^*(\cdot|s))\big].
        \label{eq2:lem_converge_stoc1}
    \end{align}
    Notice that we have
    \begin{align}
        &\big\langle Q^t(s,\cdot),\pi^*(\cdot|s)\big\rangle+h(\pi^*(\cdot|s))=\big\langle c(s,\cdot)+\gamma\P_sV_t,\pi^*(\cdot|s)\big\rangle+h(\pi^*(\cdot|s)) \nn\\
        &=V^*(s)+\gamma(\P_{\pi^*})_s(V_t-V^*),
        \label{eq3:lem_converge_stoc1}
    \end{align}
    where 
     the second equality follows from \eqnok{eq2:lem_converge_deter}. Hence, it follows from~\eqref{eq2:lem_converge_stoc1} and~\eqref{eq3:lem_converge_stoc1} that
    \begin{align}
        V_{t+1}(s)-V^*(s)+\tfrac{1}{\eta}D^{\Tilde{\pi}_{t+1}}_{\Tilde{\pi}_t}(s)\leq \tfrac{1}{\eta}D^{\pi^*}_{\Tilde{\pi}_t}(s)-\tfrac{1}{\eta}D^{\pi^*}_{\Tilde{\pi}_{t+1}}(s)+\big\langle \Tilde{Q}^t(s,\cdot)-Q^t(s,\cdot),\pi^*(\cdot|s)\big\rangle+\gamma(\P_{\pi^*})_s(V_t-V^*).
        \label{eq4:lem_converge_stoc1}
    \end{align}
    Since~\eqref{eq4:lem_converge_stoc1} holds for every $s\in\bS$, we have
    \begin{align}
        V_{t+1}-V^*+\tfrac{1}{\eta}D^{\Tilde{\pi}_{t+1}}_{\Tilde{\pi}_t}\leq \tfrac{1}{\eta}D^{\pi^*}_{\Tilde{\pi}_t}-\tfrac{1}{\eta}D^{\pi^*}_{\Tilde{\pi}_{t+1}}+\big\langle \Tilde{Q}^t-Q^t,\pi^*\big\rangle+\gamma\P_{\pi^*}(V_t-V^*),
        \label{eq5:lem_converge_stoc1}
    \end{align}
    and summing up~\eqref{eq5:lem_converge_stoc1} for $t=0,1,...,\Bar{t}$ yields~\eqref{res:lem3p2}.
\end{proof}
\vgap

Similar to Lemma~\ref{lem:converge_deter} and Lemma~\ref{lem:deter_inexact}(b), for each epoch $k\in\{0,1,...,K-1\}$ of Algorithm~\ref{algo:VMD2}, by applying $(I-\gamma\P_{\pi^*})^{-1}$ to both sides of~\eqref{res:lem3p2} in Lemma~\ref{lem:converge_stoc1}, we obtain an upper bound on the sum of the value gaps $V_t-V^*,t=1,2,...,\Bar{t}$, the transformed terminal value gap $(I-\gamma\P_{\pi^*})^{-1}(V_{\Bar{t}+1}-V^*)$, and the transformed terminal Bregman divergence $(I-\gamma\P_{\pi^*})^{-1}D_{\Tilde{\pi}_{\Bar{t}+1}}^{\pi^*}$. However, unlike in Section~\ref{section:deter}, this bound alone is not sufficient to control the policy value gap $V^{\pi_t}-V^*$, because the exact monotonicity relation $V_t\geq\Gamma_{\pi_t}V_t$ is no longer available. Therefore, our next goal is to show that $V_t$ remains an approximate upper bound on $V^{\pi_t}$ with high probability throughout the epoch.

Moreover, compared with Lemma~\ref{lem:deter_inexact}(b), we observe that the right-hand side of~\eqref{res:lem3p2} in Lemma~\ref{lem:converge_stoc1} contains an additional term $\tsum_{t=0}^{\Bar{t}}\langle \Tilde{Q}^t-Q^t,\pi^*\rangle$, which stems from the cumulative estimation errors of $\Tilde{\P}^tV_t$. More specifically, for any $t\in\{0,1,...,T-1\}$ and $s\in\bS$, in view of Algorithm~\ref{algo:VMD2}, the error $\langle\Tilde{Q}^t(s,\cdot)-Q^t(s,\cdot),\pi^*(\cdot|s)\rangle$ can be written as
\begin{align}
    &\big\langle \Tilde{Q}^t(s,\cdot)-Q^t(s,\cdot),\pi^*(\cdot|s)\big\rangle=\gamma\big\langle\Tilde{\P}_s^tV_t-\P_sV_t,\pi^*(\cdot|s)\big\rangle \nonumber\\
    &=\gamma\big\langle\big(\underbrace{\Tilde{\P}_s^0V_0-\P_sV_0}_A\big)+\big[\underbrace{\hat{\P}_s^{(m_{k,2})}(V_t-V_0)-\P_s(V_t-V_0)}_B\big],\pi^*(\cdot|s)\big\rangle.
    \label{error_decompose_stoc1}
\end{align}
Note that the terms $A$ and $B$ in~\eqref{error_decompose_stoc1} occur from independent estimation processes in Algorithm~\ref{algo:VMD2}. In order to bound $A$ and $B$ respectively with high probability, we utilize the standard concentration inequalities below, which have also been shown in~\cite{hoeffding1963probability,massart2007concentration} and used in~\cite[Theorem E.1, E.2]{sidford2018near}.

\begin{lemma}
    For any $(s,a)\in\bS\times \bA$, transition probability vector $\P_{s,a}\in\Delta_{\bS}$, and vector $V\in\R^{|\bS|}$, let $\hat{\P}^{(m)}_{s,a}\in\Delta_{\bS}$ be the estimation of $\P_{s,a}$ through $m$ independent observations of one-step state transition  $(s_1,s_2,...,s_m)$ starting from $(s,a)$ (see~\eqref{approx_trans_kernel}), then we have
    \begin{itemize}
        \item [(a)] For any $\delta\in(0,1)$,
        \begin{align}
            \big|\hat{\P}_{s,a}^{(m)\tr} V-\P_{s,a}^{\tr} V\big|\leq\|V\|_\infty\sqrt{2m^{-1}\ln(2\delta^{-1})}
            \label{hoeffding1}
        \end{align}
        holds with probability at least $1-\delta$.
        \item [(b)] For any $\delta\in(0,1)$, let $\mathrm{Var}_{s'\sim \P_{s,a}}(V(s'))=\P_{s,a}^\tr (V)^2-(\P_{s,a}^\tr V)^2$, then
        \begin{align}
            \big|\hat{\P}_{s,a}^{(m)\tr} V-\P_{s,a}^{\tr} V\big|\leq\sqrt{2m^{-1}\mathrm{Var}_{s'\sim \P_{s,a}}(V(s'))\ln(2\delta^{-1})}+(2/3)\|V\|_\infty m^{-1}\ln(2\delta^{-1})
            \label{bernstein1}
        \end{align}
        holds with probability at least $1-\delta$.
    \end{itemize}
    \label{lem:concentration_bounds}
\end{lemma}
\vgap

Note that~\eqref{hoeffding1} and~\eqref{bernstein1} in Lemma~\ref{lem:concentration_bounds} follow directly from Hoeffding's inequality and Bernstein's inequality, respectively. For notational simplicity, we define $\sigma_V\in\R^{|\bS|\times|\bA|}$ as the one-step variance matrix associated with a value vector $V\in\R^{|\bS|}$, with entries $\sigma_V(s,a):=\P_{s,a}^\tr (V)^2-(\P_{s,a}^\tr V)^2$. Moreover, for any policy $\pi$, we define $\sigma_V^{\pi}:=\P_{\pi}(V)^2-(\P_{\pi}V)^2\in\R^{|\bS|}$ as the one-step variance vector associated with $V\in\R^{|\bS|}$ and $\pi$. By the law of total variance, the following relation between $\sigma_V(s,\cdot)$ and $\sigma_V^{\pi}(s)$ holds for any $s\in\bS$:
\begin{align}
    \sigma_V^{\pi}(s)=\mathrm{Var}_{s'\sim\P_{s,a},a\sim\pi(\cdot|s)}(V(s'))\geq\E_{a\sim\pi(\cdot|s)}\left[\mathrm{Var}_{s'\sim\P_{s,a}}(V(s')\mid a)\right]=\left\langle\sigma_V(s,\cdot),\pi(\cdot|s)\right\rangle.
    \label{law_of_total_var_entry}
\end{align}
Next, we define $\langle\sigma_V,\pi\rangle\in\R^{|\bS|}$ and $\langle\sqrt{\sigma_V},\pi\rangle\in\R^{|\bS|}$ by $(\langle\sigma_V,\pi\rangle)(s):=\langle\sigma_V(s,\cdot),\pi(\cdot|s)\rangle$ and $(\langle\sqrt{\sigma_V},\pi\rangle)(s):=\langle\sqrt{\sigma_V}(s,\cdot),\pi(\cdot|s)\rangle$ for all $s\in\bS$, respectively, and let $\sqrt{\sigma_V^{\pi}}$ denote the element-wise square root of $\sigma_V^{\pi}$. Then, it follows from~\eqref{law_of_total_var_entry} that
\begin{align}
    \sigma_V^{\pi}\geq\left\langle\sigma_V,\pi\right\rangle \qquad \text{and} \qquad \sqrt{\sigma_V^{\pi}}\geq\left\langle\sqrt{\sigma_V},\pi\right\rangle,
    \label{law_of_total_var}
\end{align}
where the second relation follows from the first one by Jensen's inequality.

We also develop the following total variance bound, which is similar to~\cite[Lemma 8]{gheshlaghi2013minimax} and~\cite[Lemma C.1]{sidford2018near} and will be useful in our later analysis. Unlike those works, where $\pi$ is restricted to deterministic stationary policies, our $\pi$ may represent an arbitrary randomized stationary policy. Moreover, our result applies to the regularized setting. i.e., $h\neq 0$. The proof is deferred to the appendix.

\begin{lemma}
    For any policy $\pi$, we have
    \begin{align*}
        \big\|(I-\gamma\P_{\pi})^{-1}\sqrt{\sigma_{V^{\pi}}^\pi}\big\|_\infty^2\leq\tfrac{(1+\Bar{h})^2(1+\gamma)}{\gamma^2(1-\gamma)^3},
    \end{align*}
    \label{lem:total_variance}
    where $V^{\pi}$ is the state-value function defined in~\eqref{def:Vpi}.
\end{lemma}
\vgap

Lemma~\ref{lem:total_variance} provides an upper bound for $\|(I-\gamma\P_{\pi})^{-1}\sqrt{\sigma_{V^{\pi}}^\pi}\|_\infty^2$ in $\cO((1-\gamma)^{-3})$, which improves upon the $\cO((1-\gamma)^{-4})$ bound obtained by directly using $\|(I-\gamma\P_\pi)^{-1}\|_\infty\leq(1-\gamma)^{-1}$ and $\|\sigma_{V^{\pi}}^\pi\|_\infty\leq\|V^{\pi}\|_\infty^2\leq (1+\Bar{h})^2(1-\gamma)^{-2}$. To clarify why Lemma~\ref{lem:total_variance} is useful in the convergence analysis, we first compare the two bounds~\eqref{hoeffding1} and~\eqref{bernstein1} in Lemma~\ref{lem:concentration_bounds}. When $m$ is large, the right-hand side of~\eqref{bernstein1} is dominated by its first term. Moreover, since $\sqrt{\mathrm{Var}_{s'\sim \P_{s,a}}(V(s'))}\leq\|V\|_\infty$, then both~\eqref{hoeffding1} and~\eqref{bernstein1} yields bounds of order $\cO(\|V\|_\infty m^{-1/2})$. However, by combining Lemma~\ref{lem:total_variance} with Bernstein's inequality~\eqref{bernstein1}, we are able to derive sharper bounds than those obtained by directly applying Hoeffding's inequality~\eqref{hoeffding1} for several error terms appearing later in the analysis.

Now we bound the error terms $A$ and $B$ in~\eqref{error_decompose_stoc1}, respectively, in the following lemma.

\begin{lemma}
    \begin{itemize}
        \item [(a)] For epoch $k=0$ of Algorithm~\ref{algo:VMD2} and any $\delta\in(0,1)$, the following relation
        \begin{align}
            \big|\Tilde{\P}_{s,a}^{0\tr}V_0-\P_{s,a}^\tr V_0\big|\leq \sqrt{2}m_{0,1}^{-\tfrac{1}{2}}\sqrt{\sigma_{V_0}(s,a)}\sqrt{\ln(2|\bS||\bA|\delta^{-1})}+\tfrac{2(1+\Bar{h})}{3}m_{0,1}^{-1}(1-\gamma)^{-1}\ln(2|\bS||\bA|\delta^{-1})
            \label{error_pv0_stoc1}
        \end{align}
        is satisfied for all $(s,a)\in\bS\times\bA$ with probability at least $1-\delta$.
        \item [(b)] For any epoch $k\in\{1,2,...,K-1\}$ of Algorithm~\ref{algo:VMD2}, $\delta\in(0,1)$ and $u\in(0,(1+\Bar{h})/(1-\gamma)]$, when
        \begin{align}
            m_{k,1}\geq \max\big\{200(1+\Bar{h})^2u^{-2},32(1-\gamma)\big\}(1-\gamma)^{-3}\ln(6|\bS||\bA|\delta^{-1})
            \label{m_error_pvki_stoc1}
        \end{align}
        holds, we have
        \begin{align}
            &\big|\Tilde{\P}_{s,a}^{0\tr}V_0-\P_{s,a}^\tr V_0\big|\leq \tfrac{1}{2}\Big[\tfrac{(1-\gamma)^{1/2}}{1+\Bar{h}}\sqrt{\sigma_{V^{\pi_0}}(s,a)}+1\Big](1-\gamma)u \quad \forall (s,a)\in\bS\times\bA \quad\text{and}\quad \big\|V_0-V^{\pi_0}\big\|_\infty\leq u
            \label{error_pvki_stoc1}
        \end{align}
        with probability at least $1-\delta$.
        \item [(c)] For any $\delta\in(0,1)$ and any epoch $k\in\{0,1,...,K-1\}$ of Algorithm~\ref{algo:VMD2}, in any iteration $t\in\{1,2,...,T-1\}$,
        \begin{align}
            \big|\hat{\P}_{s,a}^{(m_{k,2})\tr}(V_t-V_0)-\P_{s,a}^\tr (V_t-V_0)\big|\leq\sqrt{2}m_{k,2}^{-\tfrac{1}{2}}\|V_t-V_0\|_\infty\sqrt{\ln(2|\bS||\bA|\delta^{-1})}
            \label{error_pvkv0_stoc1}
        \end{align}
        is satisfied for all $(s,a)\in\bS\times\bA$ with probability at least $1-\delta$.
    \end{itemize}
    \label{lem:error_pv_stoc1}
\end{lemma}
\begin{proof}
    We first show Lemma~\ref{lem:error_pv_stoc1}$(a)$. In view of Lemma~\ref{lem:concentration_bounds}$(b)$ and $\|V_0\|_\infty\leq u_0$, for any $(s,a)\in\bS\times\bA$, we know that~\eqref{error_pv0_stoc1} holds with probability at least $1-\delta/(|\bS||\bA|)$. Since the processes of approximating $\Tilde{\P}_{s,a}^0$ are independent for different $(s,a)$ pairs, then Lemma~\ref{lem:error_pv_stoc1}$(a)$ holds by taking a union bound over all $(s,a)\in\bS\times\bA$.

    Next, since the proof of Lemma~\ref{lem:error_pv_stoc1}$(b)$ is more complicated, we defer it to the appendix.

    Let us now consider Lemma~\ref{lem:error_pv_stoc1}$(c)$. In view of Lemma~\ref{lem:concentration_bounds}$(b)$, for any $(s,a)\in\bS\times\bA$,~\eqref{error_pvkv0_stoc1} holds with probability at least $1-\delta/(|\bS||\bA|)$. Since we independently estimate $\hat{\P}_{s,a}^{(m_{k,2})}$ for each $(s,a)$ pair in Algorithm~\ref{algo:VMD2}, Lemma~\ref{lem:error_pv_stoc1}$(c)$ is satisfied by taking a union bound over all $(s,a)\in\bS\times\bA$.
\end{proof}
\vgap

Lemma~\ref{lem:error_pv_stoc1} quantifies the stochastic errors introduced by the estimators $\Tilde{\P}^tV_t,t=0,1,...,T-1$. In the next several lemmas, we show that for each epoch $k$, if $\|V_0-V^*\|_\infty\leq u_k$, then with high probability, $V_t$ serves as an approximate upper bound on $V^{\pi_t}$ for all $t\in\{1,2,...,T\}$, with an error of order $\cO(u_k)$. To establish this result, we first utilize the error bounds in Lemma~\ref{lem:error_pv_stoc1} to derive the following approximate monotonicity property
\begin{align}
    V_t\geq\Gamma_{\pi_t}V_t-\beta_t \quad\text{and}\quad V_t\geq V^{\pi_t}-(I-\gamma\P_{\pi_t})^{-1}\beta_t \quad\text{for}\quad t=1,2,...,T
    \label{induction1_stoc1}
\end{align}
for some $\beta_t\in\R^{|\bS|}$, and then show $\|(I-\gamma\P_{\pi_t})^{-1}\beta_t\|_\infty\in\cO(u_k)$ with high probability.

Let us first specify $\beta_t,t=1,...,T$ in~\eqref{induction1_stoc1}. For notational simplicity, we define $\Tilde{\Gamma}^t_{\pi}V:=\langle c+\gamma\Tilde{\P}^t V,\pi\rangle+h(\pi)\in\R^{|\bS|}$ for any $V\in\R^{|\bS|}$ and policy $\pi$.

\begin{lemma}
    In any epoch $k\in\{0,1,...,K-1\}$ of Algorithm~\ref{algo:VMD2}, for any iteration $t\in\{1,2,...,T\}$, we have
    \begin{align}
        V_t\geq\Gamma_{\pi_t}V_t-\beta_t,
        \label{res:lem_induc1_stoc1}
    \end{align}
    in which each entry of $\beta_t\in\R^{|\bS|}$ is defined as $\beta_t(s):=-\langle \Tilde{Q}^{\tau(s,t)}(s,\cdot)-Q^{\tau(s,t)}(s,\cdot),\pi_t(\cdot|s)\rangle$, where the function $\tau(s,t)\in\{0,1,...,t-1\}$ depends on the state $s$ and iteration $t$, and returns some iteration before $t$.
    \label{lem:induc1_stoc1}
\end{lemma}
\begin{proof}
    We show this result through an inductive argument. First, for $t=1$ and any $s\in\bS$, we consider two cases from~\eqref{eq2:algoVMD2}. In the first case $V_1(s)=\Tilde{V}_1(s)$, we have
    \begin{align}
        V_1(s)=\big(\Tilde{\Gamma}^0_{\pi_1}V_0\big)(s)=(\Gamma_{\pi_1}V_0)(s)+\big\langle \Tilde{Q}^0(s,\cdot)-Q^0(s,\cdot),\pi_1(\cdot|s)\big\rangle\geq(\Gamma_{\pi_1}V_1)(s)-\beta_1(s),
        \label{eq1:lem_induc1_stoc1}
    \end{align}
    where the inequality uses $V_0\geq V_1$ and $\beta_1(s)=-\langle \Tilde{Q}^0(s,\cdot)-Q^0(s,\cdot),\pi_1(\cdot|s)\rangle$, with $\tau(s,1)=0$. Now we discuss the second case $V_1(s)=V_0(s)$. In view of~\eqref{eq3:algoVMD2} in Algorithm~\ref{algo:VMD2}, we know that $\pi_1(\cdot|s)=\pi_0(\cdot|s)$, and hence
    \begin{align}
        V_1(s)=V_0(s)\geq\big(\Tilde{\Gamma}^0_{\pi_0}V_0\big)(s)\geq\big(\Tilde{\Gamma}^0_{\pi_1}V_1\big)(s)=(\Gamma_{\pi_1}V_1)(s)-\beta_1(s),
        \label{eq2:lem_induc1_stoc1}
    \end{align}
    where the first inequality holds because we have $V_0=(1+\Bar{h})/(1-\gamma)\cdot\mathbf{1}_{\bS}$ when $k=0$, and $V_0=\Tilde{\Gamma}^0_{\pi_0}V_0$ when $k>0$ by the algorithmic procedure in Algorithm~\ref{algo:VMD2}, the second inequality uses $V_0\geq V_1$, and the second equality applies $\beta_1(s)=-\langle \Tilde{Q}^0(s,\cdot)-Q^0(s,\cdot),\pi_1(\cdot|s)\rangle$, with $\tau(s,1)=0$.

    Next, suppose that for any $t'\in\{1,2,...,T-1\}$,~\eqref{res:lem_induc1_stoc1} holds for $t=1,2,...,t'$, then we show that~\eqref{res:lem_induc1_stoc1} holds for $t=t'+1$. For any $s\in\bS$, we still consider two cases from~\eqref{eq2:algoVMD2}. In the first case $V_{t'+1}(s)=\Tilde{V}_{t'+1}(s)$, we have
    \begin{align}
        V_{t'+1}(s)=\big(\Tilde{\Gamma}^{t'}_{\pi_{t'+1}}V_{t'}\big)(s)=\big(\Gamma_{\pi_{t'+1}}V_{t'}\big)(s)+\big\langle \Tilde{Q}^{t'}(s,\cdot)-Q^{t'}(s,\cdot),\pi_{t'+1}(\cdot|s)\big\rangle\geq\big(\Gamma_{\pi_{t'+1}}V_{t'+1}\big)(s)-\beta_{t'+1}(s),
        \label{eq3:lem_induc1_stoc1}
    \end{align}
    where the inequality utilizes $V_{t'}\geq V_{t'+1}$ and $\beta_{t'+1}(s)=-\langle \Tilde{Q}^{t'}(s,\cdot)-Q^{t'}(s,\cdot),\pi_{t'+1}(\cdot|s)\rangle$, with $\tau(s,t'+1)=t'$. For the second case $V_{t'+1}(s)=V_{t'}(s)$, $\pi_{t'+1}(\cdot|s)=\pi_{t'}(\cdot|s)$ holds by~\eqref{eq3:algoVMD2} in Algorithm~\ref{algo:VMD2}, then
    \begin{align}
        V_{t'+1}(s)=V_{t'}(s)\geq\big(\Gamma_{\pi_{t'}}V_{t'}\big)(s)-\beta_{t'}(s)\geq\big(\Gamma_{\pi_{t'+1}}V_{t'+1}\big)(s)-\beta_{t'+1}(s),
        \label{eq4:lem_induc1_stoc1}
    \end{align}
    where the first inequality follows from the inductive hypothesis, and the second inequality uses $V_{t'}\geq V_{t'+1}$ and $\beta_{t'+1}(s)=\beta_{t'}(s)$, with $\tau(s,t'+1)=\tau(s,t')$. This completes the inductive argument and concludes the proof.
\end{proof}
\vgap

Lemma~\ref{lem:induc1_stoc1} characterizes the sequence $\beta_t$ defined in~\eqref{induction1_stoc1} for $t=1,2,...,T$. To further show $\|(I-\gamma\P_{\pi_t})^{-1}\beta_t\|_\infty\in\cO(u_k)$ for $t=1,2,...,T$ with high probability, we proceed by mathematical induction and divide the argument into two parts. The first part addresses the initial iterate in each epoch and the choice of $m_{k,1}$, while the second part handles the remaining iterates within each epoch and the choice of $m_{k,2}$.

For any epoch $k\in\{0,1,...,K-1\}$ of Algorithm~\ref{algo:VMD2} and $V\in\R^{|\bS|}$, we first define the error
\begin{align}
    \varepsilon_{k}(v):=\Big[\tfrac{1}{1+\Bar{h}}(1-\gamma)^{\tfrac{3}{2}}\sqrt{\sigma_{V}(s,a)}+(1-\gamma)\Big]\min\{u_k,1\}.
    \label{def:varepsilon1_stoc1}
\end{align}
Then, in the first part of our induction, we show that for any $c\geq 5$, under specific sample sizes, the following relations
\begin{align}
    &(a)\begin{cases}
        \big|\Tilde{\P}_{s,a}^{0\tr}V_0-\P_{s,a}^\tr V_0\big|\leq\tfrac{\varepsilon_{k}(V_0)}{2c} \quad \forall (s,a)\in\bS\times\bA \quad\text{if } k=0; \\
        \big|\Tilde{\P}_{s,a}^{0\tr}V_0-\P_{s,a}^\tr V_0\big|\leq\tfrac{\varepsilon_{k}(V^{\pi_0})}{2c} \quad \forall (s,a)\in\bS\times\bA \quad\text{and}\quad \|V_0-V^{\pi_0}\|_\infty\leq\tfrac{u_k}{c} \quad\text{if } k>0;
    \end{cases} \nn\\
    &(b) \ V_1\geq V^{\pi_1}-\tfrac{5u_k}{2c}\cdot\mathbf{1}_{\bS}
    \label{induction3_stoc1}
\end{align}
are satisfied with high probability as shown in the lemma below.

\begin{lemma}
    For any epoch $k\in\{0,1,...,K-1\}$ of Algorithm~\ref{algo:VMD2} and $\delta\in(0,1)$, suppose that $\|V^{\pi_0}-V^*\|_\infty\leq u_k$, and the algorithmic parameter
    \begin{align}
        m_{k,1}\geq \big[10\sqrt{2}c(1+\Bar{h})\big]^2(1-\gamma)^{-3}\min\{u_k,1\}^{-2}\ln(12|\bS||\bA|\delta^{-1}),
        \label{param:lem_induc2_stoc1}
    \end{align}
    then~\eqref{induction3_stoc1} holds with probability at least $1-\delta/2$.
    \label{lem:induc2_stoc1}
\end{lemma}
\begin{proof}
    First, in view of Lemma~\ref{lem:error_pv_stoc1}$(a)$, Lemma~\ref{lem:error_pv_stoc1}$(b)$ and the choice of $m_{k,1}$ in~\eqref{param:lem_induc2_stoc1}, we know that ~\eqref{induction3_stoc1}$(a)$ is satisfied with probability at least $1-\delta/2$.

    Next, to show~\eqref{induction3_stoc1}$(b)$, in view of~\eqref{eq:approxmonotonebound} and Lemma~\ref{lem:induc1_stoc1}, it suffices to show
    \begin{align}
        (I-\gamma\P_{\pi_1})^{-1}\big(-\big\langle\Tilde{Q}^0-Q^0,\pi_1\big\rangle\big)\leq\tfrac{5u_k}{2c}\cdot\mathbf{1}_{\bS}.
        \label{eq7:lem_monotone_stoc1}
    \end{align}
    For any $(s,a)\in\bS\times\bA$, since we have $\sigma_{V_0}(s,a)\leq\|V_0\|_\infty^2\leq (1+\Bar{h})^2/(1-\gamma)^2$ when the epoch $k=0$, and $\sigma_{V^{\pi_0}}(s,a)\leq\|V^{\pi_0}\|_\infty^2\leq (1+\Bar{h})^2/(1-\gamma)^2$ when $k>0$, then~\eqref{induction3_stoc1}$(a)$ implies that
    \begin{align}
        \big|\Tilde{\P}_{s,a}^{0\tr}V_0-\P_{s,a}^\tr V_0\big|\leq \tfrac{1}{2c}\big[(1-\gamma)^{\tfrac{1}{2}}+(1-\gamma)\big]=:w_0 \qquad \forall (s,a)\in\bS\times\bA,
        \label{eq10:lem_monotone_stoc1}
    \end{align}
    and thus
    \begin{align}
        \big\|(I-\gamma\P_{\pi_1})^{-1}\big\langle\Tilde{Q}^0-Q^0,\pi_1\big\rangle\big\|_\infty\leq(1-\gamma)^{-1}\gamma w_0\leq\tfrac{1}{2c}\big[(1-\gamma)^{-\tfrac{1}{2}}+1\big]=:w,
        \label{eq11:lem_monotone_stoc1}
    \end{align}
    which, in view of~\eqref{eq:approxmonotonebound} and Lemma~\ref{lem:induc1_stoc1}, yields
    \begin{align}
        V_1\geq V^{\pi_1}+(I-\gamma\P_{\pi_1})^{-1}\big\langle\Tilde{Q}^0-Q^0,\pi_1\big\rangle\geq V^{\pi_1}-w\cdot\mathbf{1}_{\bS}.
        \label{eq12:lem_monotone_stoc1}
    \end{align}

    Notice that the bound in~\eqref{eq11:lem_monotone_stoc1} does not imply~\eqref{eq7:lem_monotone_stoc1}. To establish~\eqref{eq7:lem_monotone_stoc1}, we further utilize Lemma~\ref{lem:total_variance}. Since $\Tilde{Q}^0-Q^0$ depends on $\sigma_{V_0}$ or $\sigma_{V^{\pi_0}}$ (see~\eqref{induction3_stoc1}), then we relates $\sigma_{V_0}$ or $\sigma_{V^{\pi_0}}$ to $\sigma_{V^{\pi_1}}$ in order to invoke Lemma~\ref{lem:total_variance}. To this end, we first bound $\|V^{\pi_1}-V_0\|_\infty$ if epoch $k=0$, and $\|V^{\pi_1}-V^{\pi_0}\|_\infty$ if $k>0$, i.e.,
    \begin{align}
    \begin{cases}
        &V_0+w\cdot\mathbf{1}_{\bS}\geq V_1+w\cdot\mathbf{1}_{\bS}\overset{(a)}{\geq} V^{\pi_1}\geq V^*\overset{(b)}{\geq}V_0-u_0\cdot\mathbf{1}_{\bS} \quad\text{if } k=0; \\
        &V^{\pi_0}+(u_k+w)\cdot\mathbf{1}_{\bS}\overset{(c)}{\geq} V_0+w\cdot\mathbf{1}_{\bS}\geq V_1+w\cdot\mathbf{1}_{\bS}\overset{(d)}{\geq} V^{\pi_1}\geq V^*\overset{(e)}{\geq}V^{\pi_0}-u_k\cdot\mathbf{1}_{\bS} \quad\text{if } k>0,
    \end{cases}
        \label{eq13:lem_monotone_stoc1}
    \end{align}
    where $(a),(d)$ follow from~\eqref{eq12:lem_monotone_stoc1}, $(b)$ uses $\|\hat{V}_0-V^*\|_\infty\leq u_0$, $(c)$ applies $\|V_0-V^{\pi_0}\|_\infty\leq u_k/c\leq u_k$ by~\eqref{induction3_stoc1}$(a)$, and $(e)$ uses $\|V^{\pi_0}-V^*\|_\infty\leq u_k$. It directly follows from~\eqref{eq13:lem_monotone_stoc1} that
    \begin{align}
        \|V^{\pi_1}-V_0\|_\infty\leq u_0+w \quad\text{if } k=0; \qquad \|V^{\pi_1}-V^{\pi_0}\|_\infty\leq u_k+w \quad\text{if } k>0.
        \label{eq14:lem_monotone_stoc1}
    \end{align}
    Utilizing~\eqref{eq14:lem_monotone_stoc1}, the triangle inequality, i.e., $\sqrt{\sigma_V(s,a)}\leq\sqrt{\sigma_{V'}(s,a)}+\sqrt{\sigma_{V-V'}(s,a)} \ \forall V,V'\in\R^{|\bS|}$, and the relation $\sqrt{\sigma_V(s,a)}\leq\|V\|_\infty \ \forall V\in\R^{|\bS|}$, for any $(s,a)\in\bS\times\bA$, we have
    \begin{align}
    \begin{cases}
        \sqrt{\sigma_{V_0}(s,a)}\leq\sqrt{\sigma_{V^{\pi_1}}(s,a)}+\|V^{\pi_1}-V_0\|_\infty\leq\sqrt{\sigma_{V^{\pi_1}}(s,a)}+u_0+w \quad\text{if } k=0; \\
        \sqrt{\sigma_{V^{\pi_0}}(s,a)}\leq\sqrt{\sigma_{V^{\pi_1}}(s,a)}+\|V^{\pi_1}-V^{\pi_0}\|_\infty\leq\sqrt{\sigma_{V^{\pi_1}}(s,a)}+u_k+w \quad\text{if } k>0.
    \end{cases}
    \label{eq15:lem_monotone_stoc1}
    \end{align} Combining~\eqref{induction3_stoc1}$(a)$ with~\eqref{eq15:lem_monotone_stoc1}, we obtain
    \begin{align}
        \big|\Tilde{\P}_{s,a}^{0\tr}V_0-\P_{s,a}^\tr V_0\big|\leq\tfrac{1}{2c}\Big[\tfrac{(1-\gamma)^{3/2}}{1+\Bar{h}}\big(\sqrt{\sigma_{V^{\pi_1}}(s,a)}+u_k+w\big)+(1-\gamma)\Big]\min\{u_k,1\} \quad \forall (s,a)\in\bS\times\bA,
        \label{eq17:lem_monotone_stoc1}
    \end{align}
    and it follows from~\eqref{eq17:lem_monotone_stoc1} that
    \begin{align}
        &\big\|(I-\gamma\P_{\pi_1})^{-1}\big\langle\Tilde{Q}^0-Q^0,\pi_1\big\rangle\big\|_\infty \nn\\
        &\leq \tfrac{\gamma(1-\gamma)^{3/2}}{2c(1+\Bar{h})}\big[\big\|(I-\gamma\P_{\pi_1})^{-1}\big\langle\sqrt{\sigma_{V^{\pi_1}}},\pi_1\big\rangle\big\|_\infty+(1-\gamma)^{-1}(u_k+w)\big]\min\{u_k,1\}+\tfrac{\gamma (1-\gamma)}{2c}\min\{u_k,1\} \nn\\
        &\leq\tfrac{\gamma(1-\gamma)^{3/2}}{2c(1+\Bar{h})}\Big\{(1+\Bar{h})(1+\gamma)^{\tfrac{1}{2}}\gamma^{-1}(1-\gamma)^{-\tfrac{3}{2}}+(1-\gamma)^{-1}\big[u_k+\tfrac{1}{2c}(1-\gamma)^{-\tfrac{1}{2}}+\tfrac{1}{2c}\big]\Big\}\min\{u_k,1\}+\tfrac{1}{2c}u_k \nn\\
        &\leq\tfrac{\sqrt{2}}{2c}u_k+\tfrac{1}{2c}\times\big(1+\tfrac{1}{2c}+\tfrac{1}{2c}\big)u_k+\tfrac{1}{2c}u_k=\tfrac{(2+\sqrt{2}+1/c) u_k}{2c} \leq \tfrac{5u_k}{2c},
        \label{eq18:lem_monotone_stoc1}
    \end{align}
    where the first inequality applies~\eqref{eq17:lem_monotone_stoc1}, the second one utilizes~\eqref{law_of_total_var}, Lemma~\ref{lem:total_variance} and the definition of $w$ in~\eqref{eq11:lem_monotone_stoc1}, and the last one uses $c\geq 5$. This implies~\eqref{eq7:lem_monotone_stoc1} and completes the proof.
\end{proof}
\vgap

Next, we proceed to the second part of our induction. Specifically, we show that for any $t'\in\{1,2,...,T-1\}$, if~\eqref{induction3_stoc1} holds and
\begin{align}
    &(a) \ \big|\hat{\P}_{s,a}^{(m_{k,2})\tr}(V_t-V_0)-\P_{s,a}^\tr (V_t-V_0)\big|\leq \tfrac{1}{2c}(1-\gamma)u_k \quad \forall (s,a)\in\bS\times\bA, \nn\\
    &(b) \ V_{t+1}\geq V^{\pi_{t+1}}-\tfrac{5u_k}{2c}\cdot\mathbf{1}_{\bS}
    \label{induction4_stoc1}
\end{align}
are satisfied for all iterations $t=1,2,...,t'-1$, then~\eqref{induction4_stoc1} holds for $t=t'$ with probability at least $1-\delta/[2(T-1)]$. Note that in Lemma~\ref{lem:induc2_stoc1}, we proved that~\eqref{induction3_stoc1} holds with high probability by utilizing Lemma~\ref{lem:error_pv_stoc1}$(a)$ and Lemma~\ref{lem:error_pv_stoc1}$(b)$, both of which rely on Bernstein's inequality, together with the bound $\|V_0\|_\infty\leq\cO((1-\gamma)^{-1})$ or $\|V^{\pi_0}\|_\infty\leq\cO((1-\gamma)^{-1})$. By contrast, to show that~\eqref{induction4_stoc1} holds for $t=t'$ with high probability, we apply Lemma~\ref{lem:error_pv_stoc1}$(c)$, which is based on Hoeffding's inequality, and use the bound $\|V_{t'}-V_0\|_\infty=\cO(u_k)$, which will be established later.

\begin{lemma}
    In any epoch $k\in\{0,1,...,K-1\}$ of Algorithm~\ref{algo:VMD2}, for any iteration $t'\in\{1,2,...,T-1\}$ and $\delta\in(0,1)$, suppose that $\|V^{\pi_0}-V^*\|_\infty\leq u_k$,~\eqref{induction3_stoc1} and~\eqref{induction4_stoc1} for $t=1,2,...,t'-1$ are satisfied, and the algorithmic parameter
    \begin{align}
        m_{k,2}\geq \big(4\sqrt{2}c\big)^2(1-\gamma)^{-2}\ln\big[4(T-1)|\bS||\bA|\delta^{-1}\big],
        \label{param:lem_induc4_stoc1}
    \end{align}
    then~\eqref{induction4_stoc1} holds for $t=t'$ with probability at least $1-\delta/[2(T-1)]$.
    \label{lem:induc4_stoc1}
\end{lemma}
\begin{proof}
    We first show that~\eqref{induction4_stoc1}$(a)$ holds for $t=t'$ with probability at least $1-\delta/[2(T-1)]$. For this purpose, let us bound $\|V_{t'}-V_0\|_\infty$. Notice that
    \begin{align}
        V_0\geq V_{t'}\geq V^{\pi_{t'}}-\tfrac{5u_k}{2c}\cdot\mathbf{1}_{\bS}\geq V^*-\tfrac{u_k}{2}\cdot\mathbf{1}_{\bS}\geq V^{\pi_0}-\tfrac{3}{2}u_k\cdot\mathbf{1}_{\bS}\geq V_0-2u_k\cdot\mathbf{1}_{\bS},
        \label{eq24:lem_monotone_stoc1}
    \end{align}
    where the second inequality uses the inductive hypothesis~\eqref{induction3_stoc1}$(b)$ or~\eqref{induction4_stoc1}$(b)$ for $t=t'-1$, the third one holds due to $c\geq 5$, the fourth one applies $\|V^{\pi_0}-V^*\|_\infty\leq u_k$, and the fifth one utilizes $\|V_0-V^{\pi_0}\|_\infty\leq u_k/2$ by the inductive hypothesis~\eqref{induction3_stoc1}$(a)$. It then follows from~\eqref{eq24:lem_monotone_stoc1} that $\|V_{t'}-V_0\|_\infty\leq 2u_k$, which, together with Lemma~\ref{lem:error_pv_stoc1}$(c)$ and the choice of $m_{k,2}$ in~\eqref{param:lem_induc4_stoc1}, implies that~\eqref{induction4_stoc1}$(a)$ holds for $t=t'$ with probability at least $1-\delta/[2(T-1)]$.
    
    Next, we show that~\eqref{induction4_stoc1}$(b)$ holds for $t=t'$. In view of~\eqref{eq:approxmonotonebound} and Lemma~\ref{lem:induc1_stoc1}, it suffices to show that
    \begin{align}
        \big\|(I-\gamma\P_{\pi_{t'+1}})^{-1}\beta_{t'+1}\big\|_\infty\leq\tfrac{5u_k}{2c},
        \label{eq25:lem_monotone_stoc1}
    \end{align}
    where each entry $\beta_{t'+1}(s)=\langle \Tilde{Q}^{\tau(s,t'+1)}(s,\cdot)-Q^{\tau(s,t'+1)}(s,\cdot),\pi_{t'+1}(\cdot|s)\rangle$ and $\tau(s,t'+1)\in\{0,1,...,t'\}$ by Lemma~\ref{lem:induc1_stoc1}. For any $s\in\bS$, let $\tau=\tau(s,t'+1)$ for notational simplicity, and we have
    \begin{align}
        &|\beta_{t'+1}(s)|\leq\big\langle \big|\Tilde{Q}^{\tau}(s,\cdot)-Q^{\tau}(s,\cdot)\big|,\pi_{t'+1}(\cdot|s)\big\rangle \nn\\
        &\leq\gamma\big\langle \big|\Tilde{\P}^0_sV_0-\P_s V_0\big|,\pi_{t'+1}(\cdot|s)\big\rangle+\gamma\big\langle \big|\hat{\P}^{(m_{k,2})}_s(V_{\tau}-V_0)-\P_s (V_{\tau}-V_0)\big|,\pi_{t'+1}(\cdot|s)\big\rangle \nn\\
        &\leq\gamma\big\langle \big|\Tilde{\P}^0_sV_0-\P_s V_0\big|,\pi_{t'+1}(\cdot|s)\big\rangle+\tfrac{1}{2c}(1-\gamma)u_k,
        \label{eq26:lem_monotone_stoc1}
    \end{align}
    where $|\cdot|$ denotes the elementwise absolute value, the second inequality follows from the construction of $\Tilde{\P}^{\tau}_sV_{\tau}$ in Algorithm~\ref{algo:VMD2}, and the third inequality uses the inductive hypothesis~\eqref{induction4_stoc1}$(a)$ for $t=\tau$.
    Combining~\eqref{eq10:lem_monotone_stoc1} with~\eqref{eq26:lem_monotone_stoc1}, we have
    \begin{align}
        \big\|(I-\gamma\P_{\pi_{t'+1}})^{-1}\beta_{t'+1}\big\|_\infty&\leq (1-\gamma)^{-1}\big[w_0+\tfrac{1}{2c}(1-\gamma)u_k\big]=\tfrac{1}{2c}\big[(1-\gamma)^{-\tfrac{1}{2}}+1+u_k\big]=:w',
        \label{eq27:lem_monotone_stoc1}
    \end{align}
    which, together with~\eqref{eq:approxmonotonebound} and Lemma~\ref{lem:induc1_stoc1}, yields
    \begin{align}
        V_{t'+1}\geq V^{\pi_{t'+1}}-w'\cdot\mathbf{1}_{\bS}.
        \label{eq28:lem_monotone_stoc1}
    \end{align}
    Then we bound $\|V^{\pi_{t'+1}}-V_0\|_\infty$ if epoch $k=0$, or $\|V^{\pi_{t'+1}}-V^{\pi_0}\|_\infty$ if $k>0$, using
    \begin{align}
    \begin{cases}
        &V_0+w'\cdot\mathbf{1}_{\bS}\geq V_{t'+1}+w'\cdot\mathbf{1}_{\bS}\overset{(a)}{\geq} V^{\pi_{t'+1}}\geq V^*\overset{(b)}{\geq}V_0-u_0\cdot\mathbf{1}_{\bS} \quad\text{if } k=0; \\
        &V^{\pi_0}+(u_k+w')\cdot\mathbf{1}_{\bS}\overset{(c)}{\geq} V_0+w'\cdot\mathbf{1}_{\bS}\geq V_{t'+1}+w'\cdot\mathbf{1}_{\bS}\overset{(d)}{\geq} V^{\pi_{t'+1}}\geq V^*\overset{(e)}{\geq} V^{\pi_0}-u_k\cdot\mathbf{1}_{\bS} \quad\text{if } k>0,
    \end{cases}
    \label{eq29:lem_monotone_stoc1}
    \end{align}
    where $(a)$ and $(d)$ follows from~\eqref{eq28:lem_monotone_stoc1}, $(b)$ applies $\|V_0-V^*\|_\infty\leq u_0$ when $k=0$, $(c)$ uses $\|V_0-V^{\pi_0}\|_\infty\leq u_k/c\leq u_k$ by the inductive hypothesis~\eqref{induction3_stoc1}$(a)$, and $(e)$ utilizes $\|V^{\pi_0}-V^*\|_\infty\leq u_k$ when $k>0$. By~\eqref{eq29:lem_monotone_stoc1}, we obtain
    \begin{align}
        \big\|V^{\pi_{t'+1}}-V_0\big\|_\infty\leq u_0+w' \quad\text{if } k=0; \qquad \big\|V^{\pi_{t'+1}}-V^{\pi_0}\big\|_\infty\leq u_k+w' \quad\text{if } k>0.
        \label{eq29.5:lem_monotone_stoc1}
    \end{align}
    Following the same procedure from~\eqref{eq14:lem_monotone_stoc1} to~\eqref{eq17:lem_monotone_stoc1}, we obtain
    \begin{align}
        \big|\Tilde{\P}_{s,a}^{0\tr}V_0-\P_{s,a}^\tr V_0\big|\leq \tfrac{1}{2c}\Big[\tfrac{(1-\gamma)^{3/2}}{(1+\Bar{h})}\Big(\sqrt{\sigma_{V^{\pi_{t'+1}}}(s,a)}+u_k+w'\Big)+(1-\gamma)\Big]\min\{u_k,1\} \quad \forall (s,a)\in\bS\times\bA.
        \label{eq30:lem_monotone_stoc1}
    \end{align}
    Finally, combining~\eqref{eq26:lem_monotone_stoc1} and~\eqref{eq30:lem_monotone_stoc1}, we have
    \begin{align}
        &\big\|(I-\gamma\P_{\pi_{t'+1}})^{-1}\beta_{t'+1}\big\|_\infty\leq\gamma\big\|(I-\gamma\P_{\pi_{t'+1}})^{-1}\big\langle \Tilde{P}^0V_0-PV_0,\pi_{t'+1}\big\rangle\big\|_\infty+\tfrac{1}{2c}u_k \nn\\
        &\leq\tfrac{\gamma(1-\gamma)^{3/2}}{2c(1+\Bar{h})}\big[\big\|(I-\gamma\P_{\pi_{t'+1}})^{-1}\left\langle\sqrt{\sigma_{V^{\pi_{t'+1}}}},\pi_{t'+1}\right\rangle\big\|_\infty+(1-\gamma)^{-1}(u_k+w')\big]\min\{u_k,1\}+\tfrac{1}{c}u_k \nn\\
        &\leq\tfrac{\gamma(1-\gamma)^{3/2}}{2c(1+\Bar{h})}\Big\{(1+\Bar{h})(1+\gamma)^{\tfrac{1}{2}}\gamma^{-1}(1-\gamma)^{-\tfrac{3}{2}}+(1-\gamma)^{-1}\big[\tfrac{2c+1}{2c}u_k+\tfrac{1}{2c}(1-\gamma)^{-\tfrac{1}{2}}+\tfrac{1}{2c}\big]\Big\}\min\{u_k,1\}+\tfrac{1}{c}u_k \nn\\
        &\leq\tfrac{\sqrt{2}}{2c}u_k+\tfrac{1}{2c}\times\tfrac{2c+3}{2c}u_k+\tfrac{1}{c}u_k=\tfrac{3+\sqrt{2}+3/(2c)}{2c}u_k\leq\tfrac{5u_k}{2c},
        \label{eq31:lem_monotone_stoc1}
    \end{align}
    where the third inequality uses~\eqref{law_of_total_var}, Lemma~\ref{lem:total_variance} and the definition of $w'$ in~\eqref{eq27:lem_monotone_stoc1}. This implies~\eqref{eq25:lem_monotone_stoc1} and completes the proof.
\end{proof}
\vgap

Now we conclude the induction by combining Lemma~\ref{lem:induc2_stoc1} and Lemma~\ref{lem:induc4_stoc1}.

\begin{lemma}
    For any epoch $k\in\{0,1,...,K-1\}$ of Algorithm~\ref{algo:VMD2}, $\delta\in(0,1)$ and $c\geq 5$, if $\|V^{\pi_0}-V^*\|_\infty\leq u_k$ holds, and the algorithmic parameters are set to
    \begin{align}
        &m_{k,1}\geq\tfrac{[10\sqrt{2}c(1+\Bar{h})]^2\ln(12|\bS||\bA|\delta^{-1})}{(1-\gamma)^{3}\min\{u_k,1\}^{2}}  \quad \text{and} \quad m_{k,2}\geq\tfrac{(4\sqrt{2}c)^2\ln[4(T-1)|\bS||\bA|\delta^{-1}]}{(1-\gamma)^{2}},
        \label{param:lem_monotone_stoc1}
    \end{align}
    then we have
    \begin{align}
        &(a) \ \begin{cases}
            \big|\Tilde{\P}_{s,a}^{0\tr}V_0-\P_{s,a}^\tr V_0\big|\leq\tfrac{\varepsilon_k(V_0)}{2c} \quad \forall (s,a)\in\bS\times\bA \quad\text{if } k=0; \\
            \big|\Tilde{\P}_{s,a}^{0\tr}V_0-\P_{s,a}^\tr V_0\big|\leq\tfrac{\varepsilon_k(V^{\pi_0})}{2c} \quad \forall (s,a)\in\bS\times\bA \quad\text{and}\quad \|V_0-V^{\pi_0}\|_\infty\leq\tfrac{u_k}{c} \quad\text{if } k>0;
        \end{cases} \nn\\
        &(b) \ \big|\hat{\P}_{s,a}^{(m_{k,2})\tr}(V_t-V_0)-\P_{s,a}^\tr (V_t-V_0)\big|\leq \tfrac{1}{2c}(1-\gamma)u_k \quad \forall (s,a)\in\bS\times\bA \quad\text{for }t=1,2,...,T-1; \nn\\
        &(c) \ V_t\geq V^{\pi_t}-\tfrac{5u_k}{2c}\cdot\mathbf{1}_{\bS}\quad\text{for }t=1,2,...,T
        \label{res:lem_monotone_stoc1}
    \end{align}
    with probability at least $1-\delta$, where $\varepsilon_k(\cdot)$ is defined in~\eqref{def:varepsilon1_stoc1}.
    \label{lem:monotone_stoc1}
\end{lemma}
\begin{proof}
    This result follows directly from Lemma~\ref{lem:induc2_stoc1} and Lemma~\ref{lem:induc4_stoc1}.
\end{proof}
\vgap

Lemma~\ref{lem:monotone_stoc1} implies that, for every epoch $k$ of Algorithm~\ref{algo:VMD2}, with high probability, $V_t$ bounds $V^{\pi_t}$ up to an error of order $\cO(u_k)$ for all $t=1,2,...,T$. Recall that Lemma~\ref{lem:converge_stoc1} bounds the cumulative sum of $V_t-V^*$. With the help of Lemma~\ref{lem:converge_stoc1} and Lemma~\ref{lem:monotone_stoc1}, we obtain the following single-epoch convergence result for Algorithm~\ref{algo:VMD2}.

\begin{proposition}
    For any epoch $k\in\{0,1,...,K-1\}$ of Algorithm~\ref{algo:VMD2} and $\delta\in(0,1)$, suppose that $\|V^{\pi_0}-V^*\|_\infty\leq u_k$ holds, and the algorithmic parameters are set to
    \begin{align}
        &T=\Big\lceil\tfrac{9}{1-\gamma}\Big\rceil,\quad \eta_t=\tfrac{D_0}{u_k} \quad\text{for}\ t=0,1,...,T-1, \nn\\
        &m_{k,1}\geq\tfrac{[200\sqrt{2}(1+\Bar{h})]^2\ln(12|\bS||\bA|\delta^{-1})}{(1-\gamma)^{3}\min\{u_k,1\}^{2}}\quad\text{and}\quad m_{k,2}\geq\tfrac{(80\sqrt{2})^2\ln[4(T-1)|\bS||\bA|\delta^{-1}]}{(1-\gamma)^{2}},
        \label{param:prop_stoc1}
    \end{align}
    then with probability at least $1-\delta$, we have
    \begin{align}
        (a) \ \big\|V^{\pi_T}-V^*\big\|_\infty\leq\tfrac{u_k}{2}=u_{k+1}; \quad (b) \ \big\|(I-\gamma\P_{\pi^*})^{-1}D_{\Tilde{\pi}_t}^{\pi^*}\big\|_\infty\leq\tfrac{5D_0}{1-\gamma} \quad\text{for all}\quad t=0,1,...,T.
        \label{res:prop_stoc1}
    \end{align}
    \label{prop:stoc1}
\end{proposition}
\begin{proof}
    First, we know from Lemma~\ref{lem:converge_stoc1} and the choice of $\eta_t$ in~\eqref{param:prop_stoc1} that
    \begin{align}
        (V_T-V^*)+\tsum_{t=1}^{T-1}(I-\gamma\P_{\pi^*})(V_t-V^*)\leq\tfrac{u_k}{D_0}D_{\Tilde{\pi}_0}^{\pi^*}+\gamma\P_{\pi^*}(V_0-V^*)+\tsum_{t=0}^{T-1}\big\langle\Tilde{Q}^t-Q^t,\pi^*\big\rangle.
        \label{eq1:prop_stoc1}
    \end{align}
    Since $V_0\geq V_1\geq\cdots\geq V_T$ holds, and by Lemma~\ref{lem:monotone_stoc1}, $V_T\geq V^{\pi_T}-(u_k/8)\cdot\mathbf{1}_{\bS}$ is satisfied with probability at least $1-\delta$, then~\eqref{eq1:prop_stoc1} can be simplified to
    \begin{align}
        &(I-\gamma\P_{\pi^*})^{-1}\big(V^{\pi_T}-\tfrac{u_k}{8}\cdot\mathbf{1}_{\bS}-V^*\big)+(T-1)\big(V^{\pi_T}-\tfrac{u_k}{8}\cdot\mathbf{1}_{\bS}-V^*\big) \nn\\
        &\leq (I-\gamma\P_{\pi^*})^{-1}\big[\tfrac{u_k}{D_0}D_{\Tilde{\pi}_0}^{\pi^*}+\gamma\P_{\pi^*}(V_0-V^*)+\tsum_{t=0}^{T-1}\big\langle\Tilde{Q}^t-Q^t,\pi^*\big\rangle\big].
        \label{eq2:prop_stoc1}
    \end{align}
    By utilizing $(I-\gamma\P_{\pi^*})^{-1}(V^{\pi_T}-V^*)\geq V^{\pi_T}-V^*$ and moving the terms involving $(u_k/8)\cdot\mathbf{1}_{\bS}$ to the right-hand side in~\eqref{eq2:prop_stoc1}, we obtain
    \begin{align}
        &\left\|V^{\pi_T}-V^*\right\|_\infty\leq\tfrac{1}{T}\Big(\underbrace{\tfrac{u_k}{D_0}\big\|(I-\gamma\P_{\pi^*})^{-1}D_{\Tilde{\pi}_0}^{\pi^*}\big\|_\infty}_{A_1}+\underbrace{\gamma\big\|(I-\gamma\P_{\pi^*})^{-1}\P_{\pi^*}(V_0-V^*)\big\|_\infty}_{A_2} \nn\\
        &+\underbrace{\tsum_{t=0}^{T-1}\big\|(I-\gamma\P_{\pi^*})^{-1}\big\langle\Tilde{Q}^t-Q^t,\pi^*\big\rangle\big\|_\infty}_{A_3}+\underbrace{\big\|\tfrac{u_k}{8}\big[(T-1)I+(I-\gamma\P_{\pi^*})^{-1}\big]\cdot\mathbf{1}_{\bS}\big\|_\infty}_{A_4}\Big).
        \label{eq3:prop_stoc1}
    \end{align}
    Now, let us bound the terms $A_1/T,A_2/T,A_3/T$ and $A_4/T$ in~\eqref{eq3:prop_stoc1}, respectively. First, it is straightforward to check from the choice of $T$ in~\eqref{param:prop_stoc1} that
    \begin{align}
        \tfrac{A_1}{T}\leq\tfrac{u_k}{(1-\gamma)T}\leq\tfrac{u_k}{9},\quad \tfrac{A_2}{T}\leq\tfrac{\|V_0-V^*\|_\infty}{(1-\gamma)T}\leq\tfrac{7u_k}{60},\quad\text{and}\quad \tfrac{A_4}{T}\leq\tfrac{u_k}{8}\big(\tfrac{T-1}{T}+\tfrac{1}{(1-\gamma)T}\big)\leq\tfrac{5u_k}{36},
        \label{eq4:prop_stoc1}
    \end{align}
    where the first relation utilizes $\|(I-\gamma\P_{\pi^*})^{-1}D_{\Tilde{\pi}_0}^{\pi^*}\|_\infty\leq (1-\gamma)^{-1}D_0$ since $\Tilde{\pi}_0(\cdot|s)=\hat{\pi}_0(\cdot|s)=\mathbf{1}_{\bA}/|\bA|$, and the second relation applies $\|V_0-V^*\|_\infty\leq u_k$ when the epoch $k=0$, and $\|V_0-V^*\|_\infty\leq\|V^{\pi_0}-V^*\|_\infty+\|V_0-V^{\pi_0}\|_\infty\leq(21/20)u_k$ (see~\eqref{res:lem_monotone_stoc1}$(a)$ in Lemma~\ref{lem:monotone_stoc1}) when $k>0$. Moreover, for any $t\in\{0,1,...,T-1\}$ in the term $A_3$ of~\eqref{eq3:prop_stoc1}, we have
    \begin{align}
        &\big\|(I-\gamma\P_{\pi^*})^{-1}\big\langle\Tilde{Q}^t-Q^t,\pi^*\big\rangle\big\|_\infty\overset{(a)}{=}\big\|(I-\gamma\P_{\pi^*})^{-1}\gamma\big\langle\big(\Tilde{P}^0V_0-PV_0\big)+\big[\hat{P}^{(m_{k,2})}(V_t-V_0)-P(V_t-V_0)\big],\pi^*\big\rangle\big\|_\infty \nn\\
        &\leq\gamma\big\|(I-\gamma\P_{\pi^*})^{-1}\big\langle\Tilde{P}^0V_0-PV_0,\pi^*\big\rangle\big\|_\infty+\gamma\big\|(I-\gamma\P_{\pi^*})^{-1}\big\langle\hat{P}^{(m_{k,2})}(V_t-V_0)-P(V_t-V_0),\pi^*\big\rangle\big\|_\infty \nn\\
        &\overset{(b)}{\leq}\gamma\Big\|(I-\gamma\P_{\pi^*})^{-1}\Big[\tfrac{(1-\gamma)^{3/2}}{40(1+\Bar{h})}\big(\big\langle\sqrt{\sigma_{V^*}},\pi^*\big\rangle+u_k\cdot\mathbf{1}_{\bS}\big)\min\{u_k,1\}+\tfrac{1-\gamma}{40}\min\{u_k,1\}\cdot\mathbf{1}_{\bS}\Big]\Big\|_\infty+\tfrac{u_k}{40} \nn\\
        &\leq \tfrac{\gamma(1-\gamma)^{3/2}}{40(1+\Bar{h})}\big\|(I-\gamma\P_{\pi^*})^{-1}\big\langle\sqrt{\sigma_{V^*}},\pi^*\big\rangle\big\|_\infty u_k+\tfrac{3u_k}{40}\overset{(c)}{\leq} \tfrac{\gamma(1-\gamma)^{3/2}}{40(1+\Bar{h})}\times\tfrac{(1+\Bar{h})(1+\gamma)^{1/2}}{\gamma(1-\gamma)^{3/2}}u_k+\tfrac{3}{40}u_k\leq\tfrac{1}{8}u_k,
        \label{eq7:prop_stoc1}
    \end{align}
    where $(a)$ follows from the construction of $\Tilde{P}^tV_t$ in Algorithm~\ref{algo:VMD2}, $(b)$ uses the error bounds in~\eqref{res:lem_monotone_stoc1}$(a)$ and~\eqref{res:lem_monotone_stoc1}$(b)$, $\sqrt{\sigma_{V_0}(s,a)}-\sqrt{\sigma_{V^*}(s,a)}\leq\sqrt{\sigma_{V_0-V^*}(s,a)}\leq\|V_0-V^*\|_\infty\leq u_0$ if epoch $k=0$, and $\sqrt{\sigma_{V^{\pi_0}}(s,a)}-\sqrt{\sigma_{V^*}(s,a)}\leq\sqrt{\sigma_{V^{\pi_0}-V^*}(s,a)}\leq\|V^{\pi_0}-V^*\|_\infty\leq u_k$ if $k>0$, and $(c)$ utilizes~\eqref{law_of_total_var} and Lemma~\ref{lem:total_variance}. It follows from~\eqref{eq7:prop_stoc1} that $A_3/T\leq u_k/8$, which, together with~\eqref{eq3:prop_stoc1} and~\eqref{eq4:prop_stoc1}, yields $\|V^{\pi_T}-V^*\|_\infty\leq (177/360)u_k\leq u_k/2$. This completes the first part of this proof.

    Next, we show~\eqref{res:prop_stoc1}$(b)$. For any $t'\in\{1,2,...,T\}$, let $\eta=\eta_t$, then Lemma~\ref{lem:converge_stoc1} implies that
    \begin{align}
        (V_{t'}-V^*)+\tsum_{t=1}^{t'-1}(I-\gamma\P_{\pi^*})(V_t-V^*)+\tfrac{1}{\eta}D_{\Tilde{\pi}_{t'}}^{\pi^*}\leq\tfrac{1}{\eta}D_{\Tilde{\pi}_0}^{\pi^*}+\gamma\P_{\pi^*}(V_0-V^*)+\tsum_{t=0}^{t'-1}\big\langle\Tilde{Q}^t-Q^t,\pi^*\big\rangle,
        \label{eq8:prop_stoc1}
    \end{align}
    which yields
    \begin{align}
        &\eta(I-\gamma\P_{\pi^*})^{-1}\big(V^{\pi_{t'}}-V^*\big)+\eta(t'-1)\big(V^{\pi_{t'}}-V^*\big)+(I-\gamma\P_{\pi^*})^{-1}D_{\Tilde{\pi}_{t'}}^{\pi^*} \nn\\
        &\leq(I-\gamma\P_{\pi^*})^{-1}D_{\Tilde{\pi}_0}^{\pi^*}+\eta(I-\gamma\P_{\pi^*})^{-1}\Big[\gamma\P_{\pi^*}(V_0-V^*)+\tsum_{t=0}^{t'-1}\big\langle\Tilde{Q}^t-Q^t,\pi^*\big\rangle\Big]+\eta\big[\tfrac{1}{1-\gamma}+(t'-1)\big]\cdot\tfrac{u_k}{8}\cdot\mathbf{1}_{\bS}
        \label{eq9:prop_stoc1}
    \end{align}
    by $V_0\geq V_1\geq\cdots\geq V_T$ and $V_T\geq V^{\pi_T}-(u_k/8)\cdot\mathbf{1}_{\bS}$. Since $V^{\pi_{t'}}\geq V^*$ and $\eta=D_0/u_k$, then we simplify~\eqref{eq9:prop_stoc1} to
    \begin{align}
        \big\|(I-\gamma\P_{\pi^*})^{-1}D_{\Tilde{\pi}_{t'}}^{\pi^*}\big\|_\infty\leq\tfrac{D_0}{1-\gamma}+\tfrac{D_0}{u_k}\cdot\tfrac{1}{1-\gamma}\cdot\tfrac{21u_k}{20}+\tfrac{D_0}{u_k}\cdot\tfrac{t'u_k}{8}+\tfrac{D_0}{u_k}\cdot\big(\tfrac{1}{1-\gamma}+t'-1\big)\cdot\tfrac{u_k}{8}\leq\tfrac{5D_0}{1-\gamma},
        \label{eq10:prop_stoc1}
    \end{align}
    where the first inequality uses $\|(I-\gamma\P_{\pi^*})^{-1}D_{\Tilde{\pi}_0}^{\pi^*}\|_\infty\leq D_0/(1-\gamma)$, $\|V_0-V^*\|_\infty\leq(21/20)u_k$ and~\eqref{eq7:prop_stoc1}, and the second inequality applies the parameter choice of $T$ in~\eqref{param:prop_stoc1} and $t' \leq T$. This completes the proof.
\end{proof}
\vgap

Proposition~\ref{prop:stoc1} shows that each epoch $k\in\{0,1,...,K-1\}$ of Algorithm~\ref{algo:VMD2} reduces the value error from $u_k$ to $u_{k+1}=u_k/2$ with high probability, while maintaining control of the Bregman-divergence term. This result is the stochastic analogue of Proposition~\ref{prop:deter_inexact} and implies that Algorithm~\ref{algo:VMD2} computes an $\epsilon$-optimal policy $\hat{\pi}_K$ in $K:=\lceil\log_2(u_0/\epsilon)\rceil$ epochs with high probability. As a consequence, we are now ready to prove Theorem~\ref{thm:stoc1}.

\paragraph{Proof of Theorem~\ref{thm:stoc1}}
    This proof follows directly from Proposition~\ref{prop:stoc1}. In each epoch $k\in\{0,1,...,K-1\}$, in view of the parameter setting~\eqref{param:thm_stoc1},~\eqref{res:prop_stoc1} occurs with probability at least $1-\delta/K$. As a result, with probability at least $1-\delta$, we will have $\|V^{\hat{\pi}_K}-V^*\|_\infty\leq u_K=u_0/2^K\leq\epsilon$ and $\|D_{\hat{\pi}_K}^{\pi^*}\|_\infty\leq\max_{k=0,1,...,K-1,t=0,1,...,T}\|(I-\gamma\P_{\pi^*})^{-1}D_{\Tilde{\pi}_{k,t}}^{\pi^*}\|_\infty\leq 5D_0/(1-\gamma)$, where $\Tilde{\pi}_{k,t}$ denotes $\Tilde{\pi}_t$ in epoch $k$. This completes the proof.
\vgap

It now remains to show the proof of Corollary~\ref{cor:stoc1}.

\paragraph{Proof of Corollary~\ref{cor:stoc1}}
    This proof uses the techniques in the proofs of Theorem~\ref{thm:deter_inexact} and Proposition~\ref{prop:stoc1}. First,~\eqref{res:cor_stoc1}(a) and~\eqref{res:cor_stoc1}(b) immediately follows from Theorem~\ref{thm:stoc1}. Moreover,~\eqref{res:cor_stoc1}(c) follows from Theorem~\ref{thm:stoc1},~\eqref{eq2:thm_deter_inexact} with $\pi=\hat{\pi}_K$ and~\eqref{eq3:thm_deter_inexact}. Below we derive~\eqref{res:cor_stoc1}(d). For any epoch $k\in\{0,1,...,K-1\}$ of Algorithm~\ref{algo:VMD2} and any iteration $t'\in\{1,2,...,T\}$, by considering Lemma~\ref{lem:threepoint2} with $\pi=\Bar{\pi}^*$ for all $t=0,1,...,t'-1$ and $s\in\bS$, taking $\eta_t=D_0/u_k$, following an argument that is similar to Lemma~\ref{lem:converge_stoc1} and accounting for the kernel inexactness, we obtain
    \begin{align}
        &(V_{t'}-\Bar{V}^*)+\tsum_{t=1}^{t'-1}(I-\gamma\P_{\Bar{\pi}^*})(V_t-\Bar{V}^*)+\tfrac{u_k}{D_0}D_{\Tilde{\pi}_{t'}}^{\Bar{\pi}^*} \nn\\
        &\leq \tfrac{u_k}{D_0}D_{\Tilde{\pi}_0}^{\Bar{\pi}^*}+\gamma\P_{\pi^*}(V_0-\Bar{V}^*)+\tsum_{t=0}^{t'-1}\big\langle\Tilde{Q}^t-Q^t,\Bar{\pi}^*\big\rangle+\gamma\tsum_{t=0}^{t'-1}(\P_{\Bar{\pi}^*}-\Bar{\P}_{\Bar{\pi}^*})\Bar{V}^*.
        \label{eq1:cor_stoc1}
    \end{align}
    Next, we denote $\tau'=\gamma(1+\Bar{h})\tau/(1-\gamma)^2$. By following a similar derivation from~\eqref{eq8:prop_stoc1} to~\eqref{eq10:prop_stoc1}, and utilizing $\|V^*-\Bar{V}^*\|_\infty \leq \tau'$ by~\eqref{eq3:thm_deter_inexact} and $\|(I-\gamma\P_{\Bar{\pi}^*})^{-1}\langle\Tilde{Q}^t-Q^t,\Bar{\pi}^*\rangle\|_\infty \leq u_k/8+\tau'/40$ for $t=0,1,...,t'-1$, which follows from a similar derivation in~\eqref{eq7:prop_stoc1} and uses $\sqrt{\sigma_{V^*}(s,a)}-\sqrt{\sigma_{V^{\Bar{\pi}^*}}(s,a)}\leq\sqrt{\sigma_{V^*-V^{\Bar{\pi}^*}}(s,a)}\leq\|V^*-V^{\Bar{\pi}^*}\|_\infty\leq \tau'$,~\eqref{eq1:cor_stoc1} yields
    \begin{align}
        \big\|(I-\gamma\P_{\Bar{\pi}^*})^{-1}D_{\Tilde{\pi}_{t'}}^{\pi^*}\big\|_\infty &\leq \tfrac{5D_0}{1-\gamma}+\tfrac{D_0}{u_k}\Big[\big(\tfrac{2}{1-\gamma}+\tfrac{41}{40}t'-1\big)\tfrac{\gamma(1+\Bar{h})\tau}{(1-\gamma)^2}+\gamma t'\tfrac{(1+\Bar{h})\tau}{1-\gamma}\Big] \leq \tfrac{5D_0}{1-\gamma}+\tfrac{44(1+\Bar{h})D_0\tau}{(1-\gamma)^3\epsilon},
        \label{eq2:cor_stoc1}
    \end{align}
    where the second inequality utilizes the parameter choice of $T$ in~\eqref{param:thm_stoc1}, $t' \leq T$, and $u_k \geq u_K \geq \epsilon/2$ by the definition of $u_k$ in~\eqref{def:uk} and the choice of $K$ in~\eqref{param:thm_stoc1}. This implies~\eqref{res:cor_stoc1}(d) and completes the proof.
\vgap
\section{Stochastic Value Mirror Descent under Strongly Convex Regularizers}
\label{section:stoc2}

In this section, we establish the convergence of stochastic value mirror descent when the regularizer $h$ is $\mu$-strongly convex. This additional curvature improves the $\epsilon$-dependence of the sample complexity of our method from $\cO(\epsilon^{-2})$ to $\cO(\epsilon^{-1})$. Throughout this section, we assume that the Bregman divergence is associated with the $l_1$-norm.

\subsection{Algorithm and  Convergence Result}

Algorithm~\ref{algo:VMD3} is a simplified version of Algorithm~\ref{algo:VMD2}, obtained by removing the step~\eqref{eq3:algoVMD2}, which determines whether in each iteration to accept the policy update based on the comparison in~\eqref{eq2:algoVMD2}. We will show that every policy computed by~\eqref{eq1:algoVMD3} is sufficiently well-behaved to be accepted under the strongly convex regularized setting.

For value-iteration type methods, obtaining an $\epsilon$-optimal policy with high probability using only $\cO(\epsilon^{-1})$ samples is challenging, since standard sample-average arguments typically require $\cO(\epsilon^{-2})$ samples to control the estimation error within $\epsilon$. To overcome this difficulty, for any epoch $k\in\{0,1,...,K-1\}$, we introduce a sequence of auxiliary value vectors, denoted by $\{\Bar{V}_t\}_{t=0}^T$, which serve as proxies of the computed value updates $\{V_t\}_{t=0}^T$. Specifically, we define $\Bar{V}_0:=V_0$ if epoch $k=0$, $\Bar{V}_0:=V^{\pi_0}$ if $k>0$, and
\begin{align}
     \Bar{V}_{t+1}(s):=\min\left\{\left\langle c(s,\cdot)+\gamma\P_s\Bar{V}_t,\pi_{t+1}(\cdot|s)\right\rangle+h(\pi_{t+1}(\cdot|s)),\Bar{V}_t(s)\right\} \ \forall s\in\bS \quad \text{for } t=0,1,...,T-1.
    \label{def:barvt}
\end{align}
Moreover, for every $t\in\{0,1,...,T-1\}$, we define $\Tilde{\P}^t\Bar{V}_t$ by
\begin{align}
    \big(\Tilde{\P}^t\Bar{V}_t\big)_{s,a}:=\big(\Tilde{\P}^0V_0\big)_{s,a}+\big(\hat{\P}^{(m_{k,2})}(\Bar{V}_t-V_0)\big)_{s,a}
    \label{def:tildePt_barvt}
\end{align}
for all $(s,a)\in\bS\times\bA$, where $\hat{\P}^{(m_{k,2})}$ uses the same $m_{k,2}$ observations of state transitions for computing $\Tilde{\P}^t(V_t-V_0)$. We also define $\Bar{Q}^t\in\R^{|\bS|\times|\bA|}$ by $\Bar{Q}^t(s,\cdot):=c(s,\cdot)+\gamma\P_s\Bar{V}_t\in\R^{|\bA|}$ for all $s\in\bS$.

The proxy iterates $\{\Bar{V}_t\}_{t=0}^T$ inspire us to establish the approximate monotonicity property $\Bar{V}_t\geq\Gamma_{\pi_t}\Bar{V}_t-\beta_t$, rather than $V_t\geq\Gamma_{\pi_t}V_t-\beta_t$. Later we will bound the noise $\beta_t$ in a different manner, by controlling $\|\Tilde{Q}^t-\Bar{Q}^t\|_{\max}$ and $\|V_t-\Bar{V}_t\|_\infty$ throughout the epoch, and utilizing the strong convexity of $h$ together with the Bregman divergence.

\begin{algorithm}[H]
\caption{Stochastic Value Mirror Descent (SVMD) for DMDPs with Strongly Convex Regularizers}
\label{algo:VMD3}
\begin{algorithmic}
\State{\textbf{Input:} $\hat{\pi}_0(\cdot|s)=\mathbf{1}_{\bA}/|\bA| \ \forall s\in\bS$, $V_0=(1+\Bar{h})/(1-\gamma)\cdot\mathbf{1}_{\bS}\in\R^{|\bS|}$, $K>0$, $T>0$, $\{\eta_{k,t}\}_{t=0}^{T-1}$ for $k=0,1,...,K-1$, $\{m_{k,1},m_{k,2}\}_{k=0}^{K-1}>0$.}
\For{$k=0,1,...,K-1$}
\State{Set $\pi_0=\hat{\pi}_k$, $\{\eta_t\}_{t=0}^{T-1}=\{\eta_{k,t}\}_{t=0}^{T-1}$, and compute $\Tilde{\P}^0=\hat{\P}^{(m_{k,1})}$ for each $(s,a)\in\bS\times\bA$.}
\State{If $k>0$, then compute $V_0=(I-\gamma\Tilde{\P}^0_{\pi_0})^{-1}(c_{\pi_0}+h(\pi_0))$.}
\For{$t=0,1,...,T-1$}
\State{If $t>0$, then compute $\Tilde{\P}^tV_t\in\R^{|\bS|\times|\bA|}$ by $(\Tilde{\P}^tV_t)_{s,a}=(\Tilde{\P}^0V_0)_{s,a}+\hat{\P}_{s,a}^{(m_{k,2})\tr}(V_t-V_0)$.}
\State{For all $s\in\bS$, update}
\State{\begin{align}
    &\pi_{t+1}(\cdot|s)=\argmin_{\pi(\cdot|s)\in\Delta_{\bA}}\Big\{\eta_t\big[\big\langle c(s,\cdot)+\gamma\Tilde{\P}^t_s V_t,\pi(\cdot|s)\big\rangle+h(\pi(\cdot|s))\big]+D^{\pi}_{\pi_t}(s)\Big\}, \label{eq1:algoVMD3}\\
    &V_{t+1}(s)=\min\Big\{\Tilde{V}_{t+1}(s):=\big\langle c(s,\cdot)+\gamma\Tilde{\P}^t_sV_t,\pi_{t+1}(\cdot|s)\big\rangle+h(\pi_{t+1}(\cdot|s)),V_t(s)\Big\}. \label{eq2:algoVMD3}
\end{align}}
\EndFor
\State{Set $\hat{\pi}_{k+1}=\pi_T$.}
\EndFor
\State{\textbf{Output:} $\hat{\pi}_K$.}
\end{algorithmic}
\end{algorithm}
\vgap

As in the previous sections, we assume $\|D_{\hat{\pi}_0}^{\pi}\|_\infty\leq D_0$. Moreover, we redefine
\begin{align}
    u_k:=\tfrac{\max\{1+\Bar{h},\mu D_0\}}{2^k(1-\gamma)} \qquad\text{for}\quad k=0,1,...,K.
    \label{redef:uk}
\end{align}
The convergence of Algorithm~\ref{algo:VMD3} is established as follows.

\begin{theorem}
    For any $\delta\in(0,1)$, suppose that the algorithmic parameters in Algorithm~\ref{algo:VMD3} are set to
    \begin{align}
        &K=\Big\lceil\log_2\big(\tfrac{u_0}{\epsilon}\big)\Big\rceil,\quad T=\Big\lceil\tfrac{22}{1-\gamma}\Big\rceil,\quad \eta_{k,t}=\tfrac{2}{\mu(t+1)} \quad \text{for} \quad t=0,1,...,T-1, \nn\\
        &m_{k,1}\geq 200\max\Big\{\tfrac{64(1+\Bar{h})^2}{\mu}(1-\gamma)^{-5}u_k^{-1}(\ln T+1),(1-\gamma)^{-1}\Big\}\ln(24K|\bS||\bA|\delta^{-1}) \nn\\
        &\text{and}\quad m_{k,2}\geq \Big[\tfrac{400\max\{1+\Bar{h},\mu D_0\}}{\mu}(\ln T+1)+16\Big](1-\gamma)^{-4}\ln[4K(T-1)|\bS||\bA|\delta^{-1}] \quad \text{for} \quad k=0,1,...,K-1,
        \label{param:thm_stoc2}
    \end{align}
    then with probability at least $1-\delta$, $\hat{\pi}_K$ is an $\epsilon$-optimal policy, and $\|D_{\hat{\pi}_k}^{\pi^*}\|_\infty\leq\epsilon/\mu$.
    \label{thm:stoc2}
\end{theorem}
\vgap

Theorem~\ref{thm:stoc2} implies that, in the strongly convex regularized setting, Algorithm~\ref{algo:VMD3} achieves sample complexity
\begin{align}
    \cO\Big(|\bS||\bA|\tsum_{k=0}^{K-1}\big[m_{k,1}+(T-1)m_{k,2}\big]\Big)=\Tilde{\cO}\big(|\bS||\bA|(1-\gamma)^{-5}\epsilon^{-1}\big).
    \label{complexity:stoc2}
\end{align}
To the best of our knowledge, this is the first value iteration-type method to exhibit an $\cO(\epsilon^{-1})$ dependence in this setting. Moreover, the Bregman divergence between the output policy and $\pi^*$ is bounded by $\cO(\epsilon)$, which implies convergence of the generated policy and provides a stronger stability guarantee than in the general convex setting.

\subsection{Convergence Analysis}

The first lemma is the strongly convex analogue of Lemma~\ref{lem:threepoint2}.

\begin{lemma}
    In any epoch $k\in\{0,1,...,K-1\}$ of Algorithm~\ref{algo:VMD3}, for any iteration $t\in\{0,1,...,T-1\}$, state $s\in\bS$ and policy $\pi(\cdot|s)\in\Delta_{\bA}$, we have
    \begin{align*}
        \eta_t\big[\big\langle c(s,\cdot)+\gamma\Tilde{\P}^t_sV_t,\pi_{t+1}(\cdot|s)-\pi(\cdot|s)\big\rangle+h(\pi_{t+1}(\cdot|s))-h(\pi(\cdot|s))\big]+D^{\pi_{t+1}}_{\pi_t}(s)\leq D^{\pi}_{\pi_t}(s)-(1+\eta_t\mu)D^{\pi}_{\pi_{t+1}}(s).
    \end{align*}
    \label{lem:threepoint3}
\end{lemma}
\begin{proof}
    This proof directly follows from~\cite[Lemma 3.4]{lan2020first}.
\end{proof}
\vgap

The next result builds on Lemma~\ref{lem:threepoint3} and establishes a single-epoch convergence property of Algorithm~\ref{algo:VMD3} in terms of the proxy iterates $\Bar{V}_t$.

\begin{lemma}
    In any epoch $k\in\{0,1,...,K-1\}$ of Algorithm~\ref{algo:VMD3}, taking $\eta_t=2/[\mu(t+1)]$ for $t=0,1,...,T-1$ yields
    \begin{align}
        &(\Bar{V}_T-V^*)+\tsum_{t=1}^{T-1}(I-\gamma\P_{\pi^*})(\Bar{V}_t-V^*)+\tfrac{\mu(T+2)}{2}D_{\pi_T}^{\pi^*} \nn\\
        &\leq\mu D_{\pi_0}^{\pi^*}+\gamma\P_{\pi^*}(\Bar{V}_0-V^*)+\tfrac{2}{\mu}\tsum_{t=0}^{T-1}\big\|\Tilde{Q}^t-\Bar{Q}^t\big\|_{\max}^2\cdot\mathbf{1}_{\bS}.
        \label{res:lem3p3}
    \end{align}
    \label{lem:converge_stoc2}
\end{lemma}
\begin{proof}
    For any iteration $t\in\{0,1,...,T-1\}$ and state $s\in\bS$, taking $\pi=\pi^*$ in Lemma~\ref{lem:threepoint3}, we have
    \begin{align}
        &\eta_t\big[\big\langle c(s,\cdot)+\gamma\Tilde{\P}_sV_t,\pi_{t+1}(\cdot|s)-\pi^*(\cdot|s)\big\rangle+h(\pi_{t+1}(\cdot|s))-h(\pi^*(\cdot|s))\big]+D^{\pi_{t+1}}_{\pi_t}(s)\nn\\ 
        &\leq D^{\pi^*}_{\pi_t}(s)-(1+\eta_t\mu)D^{\pi^*}_{\pi_{t+1}}(s).
        \label{eq1:lem_converge_stoc2}
    \end{align}
    It follows from~\eqref{eq1:lem_converge_stoc2} and the definitions $\Tilde{Q}^t(s,\cdot):=c(s,\cdot)+\gamma\Tilde{\P}_s^tV_t$ and $\Bar{Q}^t(s,\cdot)=c(s,\cdot)+\gamma\P_s\Bar{V}_t$ that
    \begin{align}
        &\eta_t\big[\big\langle c(s,\cdot)+\gamma\P_s\Bar{V}_t,\pi_{t+1}(s)-\pi^*(s)\big\rangle+h(\pi_{t+1}(\cdot|s))-h(\pi^*(\cdot|s))\big]+D^{\pi_{t+1}}_{\pi_t}(s) \nn\\
        &\leq D^{\pi^*}_{\pi_t}(s)-(1+\eta_t\mu)D^{\pi^*}_{\pi_{t+1}}(s)+\eta_t\big\langle\Tilde{Q}^t(s,\cdot)-\Bar{Q}^t(s,\cdot),\pi^*(\cdot|s)-\pi_{t+1}(\cdot|s)\big\rangle.
        \label{eq2:lem_converge_stoc2}
    \end{align}
    Since $\Bar{V}_{t+1}(s)\leq\langle c(s,\cdot)+\gamma\P_s^t\Bar{V}_t,\pi_{t+1}(\cdot|s)\rangle+h(\pi_{t+1}(\cdot|s))$ by~\eqref{def:barvt}, and 
    \begin{align}
        \left\langle c(s,\cdot)+\gamma\P_s\Bar{V}_t,\pi^*(\cdot|s)\right\rangle+h(\pi^*(\cdot|s))&=\left\langle c(s,\cdot)+\gamma\P_sV^*,\pi^*(\cdot|s)\right\rangle+h(\pi^*(\cdot|s))+\left\langle\gamma\P_s(\Bar{V}_t-V^*),\pi^*(\cdot|s)\right\rangle \nn\\
        &=V^*(s)+\gamma(\P_{\pi^*})_s(\Bar{V}_t-V^*),
        \label{eq3:lem_converge_stoc2}
    \end{align}
    where the second equality uses $V^*=c_{\pi^*}+\gamma P_{\pi^*}V^*+h(\pi^*)$, ~\eqref{eq2:lem_converge_stoc2} becomes
    \begin{align}
        \eta_t \big(\Bar{V}_{t+1}(s)-V^*(s)\big)+D^{\pi_{t+1}}_{\pi_t}(s)\leq &D^{\pi^*}_{\pi_t}(s)-(1+\eta_t\mu)D^{\pi^*}_{\pi_{t+1}}(s)+\eta_t\gamma(\P_{\pi^*})_s(\Bar{V}_t-V^*) \nn\\
        &+\eta_t\big\langle\Tilde{Q}^t(s,\cdot)-\Bar{Q}^t(s,\cdot),\pi^*(\cdot|s)-\pi_{t+1}(\cdot|s)\big\rangle.
        \label{eq4:lem_converge_stoc2}
    \end{align}
    Next, we bound the term $\eta_t\langle\Tilde{Q}^t(s,\cdot)-\Bar{Q}^t(s,\cdot),\pi^*(\cdot|s)-\pi_{t+1}(\cdot|s)\rangle$ as
    \begin{align}
        &\eta_t\big\langle\Tilde{Q}^t(s,\cdot)-\Bar{Q}^t(s,\cdot),\pi^*(\cdot|s)-\pi_{t+1}(\cdot|s)\big\rangle \nn\\
        &=\eta_t\big\langle\Tilde{Q}^t(s,\cdot)-\Bar{Q}^t(s,\cdot),\pi^*(\cdot|s)-\pi_t(\cdot|s)\big\rangle+\eta_t\big\langle\Tilde{Q}^t(s,\cdot)-\Bar{Q}^t(s,\cdot),\pi_t(\cdot|s)-\pi_{t+1}(\cdot|s)\big\rangle \nn\\
        &\leq\tfrac{\eta_t}{\mu}\big\|\Tilde{Q}^t(s,\cdot)-\Bar{Q}^t(s,\cdot)\big\|_\infty^2+\tfrac{\eta_t\mu}{4}\|\pi^*(\cdot|s)-\pi_t(\cdot|s)\|_1^2+\tfrac{\eta_t^2}{2}\big\|\Tilde{Q}^t(s,\cdot)-\Bar{Q}^t(s,\cdot)\big\|_\infty^2+\tfrac{1}{2}\|\pi_t(\cdot|s)-\pi_{t+1}(\cdot|s)\|_1^2 \nn\\
        &\leq \eta_t\big(\tfrac{1}{\mu}+\tfrac{\eta_t}{2}\big)\big\|\Tilde{Q}^t(s,\cdot)-\Bar{Q}^t(s,\cdot)\big\|_\infty^2+\tfrac{\eta_t\mu}{2}D_{\pi_t}^{\pi^*}(s)+D_{\pi_t}^{\pi_{t+1}}(s),
        \label{eq5:lem_converge_stoc2}
    \end{align}
    where the first inequality applies Hölder's inequality and Young's inequality. Combining~\eqref{eq4:lem_converge_stoc2} and~\eqref{eq5:lem_converge_stoc2}, we have
    \begin{align}
        \Bar{V}_{t+1}(s)-V^*(s)\leq\tfrac{1+\eta_t\mu/2}{\eta_t}D_{\pi_t}^{\pi^*}(s)-\tfrac{1+\eta_t\mu}{\eta_t}D_{\pi_{t+1}}^{\pi^*}(s)+\gamma(\P_{\pi^*})_s(\Bar{V}_t-V^*)+\big(\tfrac{1}{\mu}+\tfrac{\eta_t}{2}\big)\big\|\Tilde{Q}^t(s,\cdot)-\Bar{Q}^t(s,\cdot)\big\|_\infty^2.
        \label{eq6:lem_converge_stoc2}
    \end{align}
    Taking $\eta_t=2/[\mu(t+1)]$, we obtain
    \begin{align}
        \Bar{V}_{t+1}(s)-V^*(s)\leq\tfrac{\mu(t+2)}{2}D_{\pi_t}^{\pi^*}(s)-\tfrac{\mu(t+3)}{2}D_{\pi_{t+1}}^{\pi^*}(s)+\gamma(\P_{\pi^*})_s(\Bar{V}_t-V^*)+\tfrac{1}{\mu}\big(1+\tfrac{1}{t+1}\big)\big\|\Tilde{Q}^t(s,\cdot)-\Bar{Q}^t(s,\cdot)\big\|_\infty^2
        \label{eq7:lem_converge_stoc2}
    \end{align}
    Since~\eqref{eq7:lem_converge_stoc2} holds for every $s\in\bS$, we have
    \begin{align}
        \Bar{V}_{t+1}-V^*\leq\tfrac{\mu(t+2)}{2}D_{\pi_t}^{\pi^*}-\tfrac{\mu(t+3)}{2}D_{\pi_{t+1}}^{\pi^*}+\gamma\P_{\pi^*}(\Bar{V}_t-V^*)+\tfrac{2}{\mu}\big\|\Tilde{Q}^t-\Bar{Q}^t\big\|_{\max}^2\cdot\mathbf{1}_{\bS},
        \label{eq8:lem_converge_stoc2}
    \end{align}
    and summing up~\eqref{eq8:lem_converge_stoc2} for $t=0,1,...,T-1$ yields~\eqref{res:lem3p3}.
\end{proof}
\vgap

By applying $(I-\gamma\P_{\pi^*})^{-1}$ to both sides of~\eqref{res:lem3p3},
Lemma~\ref{lem:converge_stoc2} bounds the sum of the proxy value gaps $\Bar{V}_t-V^*,t=1,2,...,T-1$, the transformed terminal proxy value gap $(I-\gamma\P_{\pi^*})^{-1}(\Bar{V}_T-V^*)$, and the transformed terminal Bregman divergence $(I-\gamma\P_{\pi^*})^{-1}D_{\pi_T}^{\pi^*}$. Next, in order to bound $V^{\pi_T}-V^*$, we build up the relation between $\Bar{V}_t$ and $V^{\pi_t}$ for $t=1,2,...,T$, by establishing the approximate monotonicity property $\Bar{V}_t\geq\Gamma_{\pi_t}\Bar{V}_t-\beta_t$ with $\|\beta_t\|_\infty=\cO((1-\gamma)u_k)$. As a consequence, we will have $\Bar{V}_t\geq V^{\pi_t}-(I-\gamma\P_{\pi_t})^{-1}\beta_t$ where $\|(I-\gamma\P_{\pi_t})^{-1}\beta_t\|_\infty=\cO(u_k)$.

For this purpose, we define
\begin{align}
    b:=\sqrt{\tfrac{\mu}{16(\ln T+1)}}
    \label{def:b_stoc2}
\end{align}
and show that for any epoch $k\in\{0,1,...,K-1\}$ of Algorithm~\ref{algo:VMD3},
\begin{align}
    &(a) \ \big\|\Tilde{\P}^t\Bar{V}_t-\P\Bar{V}_t\big\|_{\max}\leq b(1-\gamma)^{\tfrac{3}{2}}u_k^{\tfrac{1}{2}}; \quad (b) \ \big\|\Tilde{Q}^t-\Bar{Q}^t\big\|_{\max}\leq b(1-\gamma)^{\tfrac{1}{2}}u_k^{\tfrac{1}{2}}; \quad (c) \ \|V_{t+1}-\Bar{V}_{t+1}\|_{\infty}\leq b(1-\gamma)^{\tfrac{1}{2}}u_k^{\tfrac{1}{2}}; \nn\\
    &(d) \ \Bar{V}_{t+1}\geq\Gamma_{\pi_{t+1}}\Bar{V}_{t+1}-\beta_{t+1}; \quad (e) \ \Bar{V}_{t+1}\geq V^{\pi_{t+1}}-\tfrac{u_k}{32}\cdot\mathbf{1}_{\bS}
    \label{induc_state:stoc2}
\end{align}
are satisfied for $t=0,1,...,T-1$ with high probability, where
\begin{align}
    \beta_t:=\tsum_{i=1}^t\tfrac{b^2}{2i\mu}(1-\gamma)u_k\cdot\mathbf{1}_{\bS}.
    \label{def:betak_stoc2}
\end{align}

Again, as in Section~\ref{section:stoc1}, we proceed by mathematical induction and divide our analysis into two parts. We first show that~\eqref{induc_state:stoc2} holds for $t=0$ with high probability.

\begin{lemma}
    For any epoch $k\in\{0,1,...,K-1\}$ of Algorithm~\ref{algo:VMD3} and $\delta\in(0,1)$, if $\|\Bar{V}_0-V^*\|_\infty\leq u_k$ holds, and the algorithmic parameters
    \begin{align}
        m_{k,1}\geq 200\max\Big\{\tfrac{4(1+\Bar{h})^2}{b^2}(1-\gamma)^{-5}u_k^{-1},(1-\gamma)^{-1}\Big\}\ln(24|\bS||\bA|\delta^{-1}) \quad\text{and}\quad \eta_t=\tfrac{2}{\mu(t+1)} \quad \text{for} \quad t=0,1,...,T-1,
        \label{param:lem_induc1_stoc2}
    \end{align}
    where $b$ is defined in~\eqref{def:b_stoc2}, then with probability at least $1-\delta/2$,~\eqref{induc_state:stoc2} holds for $t=0$, and we have
    \begin{align}
        \big\|\Tilde{\P}^0 V_0-\P V_0\big\|_{\max}\leq \tfrac{b}{2}(1-\gamma)^{\tfrac{3}{2}}u_k^{\tfrac{1}{2}} \quad\text{and}\quad \left\|V_0-\Bar{V}_0\right\|_\infty\leq \tfrac{b}{2}(1-\gamma)^{\tfrac{1}{2}}u_k^{\tfrac{1}{2}}.
        \label{res:lem_induc1_stoc2}
    \end{align}
    \label{lem:induc1_stoc2}
\end{lemma}
\begin{proof}
    First, in view of the choice $m_{k,1}$ in~\eqref{param:lem_induc1_stoc2}, $\|\Bar{V}_0\|_\infty\leq (1+\Bar{h})/(1-\gamma)$ and Lemma~\ref{lem:concentration_bounds}$(a)$,~\eqref{induc_state:stoc2}$(a)$ holds for $t=0$ with probability at least $1-\delta/4$.
    
    Next, we show that~\eqref{res:lem_induc1_stoc2} also holds with probability at least $1-\delta/4$. Notice that it trivially holds when the epoch $k=0$, because $V_0=\Bar{V}_0$; when $k>0$, this follows from Lemma~\ref{lem:error_pv_stoc1}$(b)$ with $u=\min\{(b/2)(1-\gamma)u_k^{1/2},(1+\Bar{h})/(1-\gamma)\}$, $\Bar{V}_0=V^{\pi_0}$ by definition, the choice $m_{k,1}$ in~\eqref{param:lem_induc1_stoc2} and $\|\sqrt{\sigma_{V^{\pi_0}}}\|_{\max}\leq\|V^{\pi_0}\|_\infty\leq(1+\Bar{h})/(1-\gamma)$.
    
    Then, we obtain that~\eqref{induc_state:stoc2}$(b)$ holds for $t=0$ by
    \begin{align}
        &\big\|\Tilde{Q}^0-\Bar{Q}^0\big\|_{\max}\leq \gamma\big(\big\|\Tilde{\P}^0V_0-\Tilde{\P}^0\Bar{V}_0\big\|_{\max}+\big\|\Tilde{\P}^0\Bar{V}_0-\P\Bar{V}_0\big\|_{\max}\big)\leq\gamma\left\|V_0-\Bar{V}_0\right\|_\infty+b(1-\gamma)^{\tfrac{3}{2}}u_k^{\tfrac{1}{2}} \nn\\
        &\leq b\gamma(1-\gamma)^{\tfrac{1}{2}}u_k^{\tfrac{1}{2}}+b(1-\gamma)^{\tfrac{3}{2}}u_k^{\tfrac{1}{2}}=b(1-\gamma)^{\tfrac{1}{2}}u_k^{\tfrac{1}{2}},
        \label{eq2:lem_monotone_stoc2}
    \end{align}
    where the second inequality uses~\eqref{induc_state:stoc2}$(a)$ for $t=0$, and the third inequality applies~\eqref{res:lem_induc1_stoc2}. Moreover,~\eqref{induc_state:stoc2}$(c)$ holds for $t=0$ by
    \begin{align}
        &|V_1(s)-\Bar{V}_1(s)|=\big|\min\big\{\big\langle\Tilde{Q}^0(s,\cdot),\pi_1(\cdot|s)\big\rangle+h(\pi_1(\cdot|s)),V_0(\cdot|s)\big\}-\min\big\{\big\langle\Bar{Q}^0(s,\cdot),\pi_1(\cdot|s)\big\rangle+h(\pi_1(\cdot|s)),\Bar{V}_0(s)\big\}\big| \nn\\
        &\leq\max\big\{\big|\big\langle\Tilde{Q}^0(s,\cdot)-\Bar{Q}^0(s,\cdot),\pi_1(\cdot|s)\big\rangle\big|,|V_0(s)-\Bar{V}_0(s)|\big\}\leq\max\big\{\big\|\Tilde{Q}^0(s,\cdot)-\Bar{Q}^0(s,\cdot)\big\|_{\infty},|V_0(s)-\Bar{V}_0(s)|\big\} \nn\\
        &\leq b(1-\gamma)^{\tfrac{1}{2}}u_k^{\tfrac{1}{2}} \quad\forall s\in\bS,
        \label{eq3:lem_monotone_stoc2}
    \end{align}
    where the first inequality holds since the map $(x,y)\mapsto\min\{x,y\}$ is $1$-Lipschitz with respect to $\|\cdot\|_\infty$, and the third inequality follows from~\eqref{res:lem_induc1_stoc2} and~\eqref{eq2:lem_monotone_stoc2}.
    
    Next, in order to show~\eqref{induc_state:stoc2}$(d)$ for $t=0$, we show $\Bar{V}_1(s)\geq(\Gamma_{\pi_1}\Bar{V}_1)(s)-\beta_1(s)$ for any $s\in\bS$ under two cases specified by~\eqref{def:barvt}. In the first case $\Bar{V}_1(s)=\langle c(s,\cdot)+\gamma\P_s\Bar{V}_0,\pi_{1}(\cdot|s)\rangle+h(\pi_{1}(\cdot|s))$, we directly have $\Bar{V}_1(s)=(\Gamma_{\pi_1}\Bar{V}_0)(s)\geq (\Gamma_{\pi_1}\Bar{V}_1)(s)$ by applying $\Bar{V}_0 \geq \Bar{V}_1$. Now we consider the second case $\Bar{V}_1(s)=\Bar{V}_0(s)$. For any $t\in\{0,1,...,T-1\}$, taking $\pi=\pi_t$ in Lemma~\ref{lem:threepoint3} yields
    \begin{align}
        \eta_t\big[\big\langle\Tilde{Q}^t(s,\cdot),\pi_{t+1}(\cdot|s)-\pi_t(\cdot|s)\big\rangle+h(\pi_{t+1}(\cdot|s))-h(\pi_t(\cdot|s))\big]+D^{\pi_{t+1}}_{\pi_t}(s)\leq -(1+\eta_t\mu)D^{\pi_t}_{\pi_{t+1}}(s),
        \label{eq4:lem_monotone_stoc2}
    \end{align}
    which can be further simplified as
    \begin{align}
        (\Gamma_{\pi_{t+1}}\Bar{V}_t)(s)-(\Gamma_{\pi_t}\Bar{V}_t)(s)&\leq \big\langle\Tilde{Q}^t(s,\cdot)-\Bar{Q}^t(s,\cdot),\pi_t(\cdot|s)-\pi_{t+1}(\cdot|s)\big\rangle-\tfrac{1}{\eta_t}\big[D^{\pi_{t+1}}_{\pi_t}(s)+(1+\eta_t\mu)D^{\pi_t}_{\pi_{t+1}}(s)\big] \nn\\
        &\leq\tfrac{\eta_t}{4}\big\|\Tilde{Q}^t(s,\cdot)-\Bar{Q}^t(s,\cdot)\big\|_\infty^2+\tfrac{1}{\eta_t}\|\pi_t(\cdot|s)-\pi_{t+1}(\cdot|s)\|_1^2-\tfrac{1}{\eta_t}\big[D^{\pi_{t+1}}_{\pi_t}(s)+D^{\pi_t}_{\pi_{t+1}}(s)\big] \nn\\
        &\leq\tfrac{1}{2\mu(t+1)}\big\|\Tilde{Q}^t(s,\cdot)-\Bar{Q}^t(s,\cdot)\big\|_\infty^2,
        \label{eq5:lem_monotone_stoc2}
    \end{align}
    where the second inequality applies Hölder's inequality and Young's inequality, and the third inequality follows from the choice of $\eta_t$ in~\eqref{param:lem_induc1_stoc2} and the property $D_{\pi}^{\pi'}(s) \geq \|\pi(\cdot|s)-\pi'(\cdot|s)\|_1^2/2$ for any $\pi,\pi'\in\Pi$. Hence, we have
    \begin{align}
        \Bar{V}_1(s)=\Bar{V}_0(s)=\left(\Gamma_{\pi_0}\Bar{V}_0\right)(s)\geq\left(\Gamma_{\pi_1}\Bar{V}_0\right)(s)-\tfrac{b^2}{2\mu}(1-\gamma)u_k\geq\left(\Gamma_{\pi_1}\Bar{V}_1\right)(s)-\beta_1(s),
        \label{eq6:lem_monotone_stoc2}
    \end{align}
    where the first inequality uses~\eqref{eq2:lem_monotone_stoc2} and~\eqref{eq5:lem_monotone_stoc2} with $t=0$, and the second inequality utilizes $\Bar{V}_0\geq\Bar{V}_1$ and the definition of $\beta_t$ in~\eqref{def:betak_stoc2}. Combining the two cases above implies~\eqref{induc_state:stoc2}$(d)$ for $t=1$.
    
    Finally,~\eqref{induc_state:stoc2}$(e)$ for $t=1$ follows from~\eqref{induc_state:stoc2}$(d)$, i.e.,
    \begin{align}
        \Bar{V}_1\geq\Gamma_{\pi_1}\Bar{V}_1-\beta_1\geq V^{\pi_1}-(I-\gamma\P_{\pi_1})^{-1}\beta_1\geq V^{\pi_1}-\tfrac{b^2}{2\mu}u_k\cdot\mathbf{1}_{\bS}\geq V^{\pi_1}-\tfrac{u_k}{32}\cdot\mathbf{1}_{\bS},
        \label{eq7:lem_monotone_stoc2}
    \end{align}
    where the last inequality applies $b=\sqrt{\mu/[16(\ln T+1)]}\leq\sqrt{\mu/16}$ by~\eqref{def:b_stoc2}. This completes the proof.
\end{proof}
\vgap

Next, we proceed to the second part of the induction and establish the high probability result~\eqref{induc_state:stoc2} for the subsequent iterations $t=1,2,...,T-1$.

\begin{lemma}
    For any epoch $k\in\{0,1,...,K-1\}$ of Algorithm~\ref{algo:VMD3} and $\delta\in(0,1)$, if $\|\Bar{V}_0-V^*\|_\infty\leq u_k$ holds, and the algorithmic parameters are set to
    \begin{align}
        &\eta_t=\tfrac{2}{\mu(t+1)} \ \text{for} \ t=0,1,...,T-1, \quad m_{k,1}\geq 200\max\Big\{\tfrac{4(1+\Bar{h})^2}{b^2}(1-\gamma)^{-5}u_k^{-1},(1-\gamma)^{-1}\Big\}\ln(24|\bS||\bA|\delta^{-1}) \quad\text{and}\quad \nn\\
        &m_{k,2}\geq \big(\tfrac{25u_0}{b^2}+16\big)(1-\gamma)^{-3}\ln[4(T-1)|\bS||\bA|\delta^{-1}],
        \label{param:lem_monotone_stoc2}
    \end{align}
    where $b$ is defined in~\eqref{def:b_stoc2}, then~\eqref{induc_state:stoc2} holds for $t=0,1,...,T-1$ with probability at least $1-\delta$.
    \label{lem:monotone_stoc2}
\end{lemma}
\begin{proof}
    In view of Lemma~\ref{lem:induc1_stoc2}, for any $t'\in\{1,2,...,T-1\}$, let us suppose that~\eqref{induc_state:stoc2} holds for $t=0,1,...,t'-1$, and~\eqref{res:lem_induc1_stoc2} is satisfied. Then we show that~\eqref{induc_state:stoc2} holds for $t=t'$ with probability at least $1-\delta/[2(T-1)]$. In order to show~\eqref{induc_state:stoc2}$(a)$ for $t=t'$, we first have
    \begin{align}
        \|\Bar{V}_{t'}-V_0\|_\infty\leq\|\Bar{V}_{t'}-\Bar{V}_0\|_\infty+\|\Bar{V}_0-V_0\|_\infty\leq\tfrac{5}{4}u_k+b(1-\gamma)^{\tfrac{1}{2}}u_k^{\tfrac{1}{2}}\leq\big(\tfrac{5}{4}u_0^{\tfrac{1}{2}}+b\big)u_k^{\tfrac{1}{2}},
        \label{eq8:lem_monotone_stoc2}
    \end{align}
    where the second inequality follows from~\eqref{res:lem_induc1_stoc2} and $\|\Bar{V}_{t'}-\Bar{V}_0\|_\infty\leq 5u_k/4$, which holds due to $\Bar{V}_0\geq\Bar{V}_{t'}$ and $\Bar{V}_{t'}\geq V^{\pi_{t'}}-(u_k/4)\cdot\mathbf{1}_{\bS}\geq V^*-(u_k/4)\cdot\mathbf{1}_{\bS}\geq \Bar{V}_0-(5u_k/4)\cdot\mathbf{1}_{\bS}$ by employing~\eqref{induc_state:stoc2}$(e)$ for $t=t'-1$, and the third inequality in~\eqref{eq8:lem_monotone_stoc2} uses $u_k\leq u_0$. Then, in view of Lemma~\ref{lem:concentration_bounds}$(a)$, the choice of $m_{k,2}$ in~\eqref{param:lem_monotone_stoc2}, and $\|\Bar{V}_{t'}-V_0\|_\infty^2\leq [(25/8)u_0+2b^2]u_k$ from~\eqref{eq8:lem_monotone_stoc2}, we have
    \begin{align}
        \big\|\hat{\P}^{(m_{k,2})}(\Bar{V}_{t'}-V_0)-\P(\Bar{V}_{t'}-V_0)\big\|_{\max}\leq m_{k,2}^{-\tfrac{1}{2}}\sqrt{2\ln[4(T-1)|\bS||\bA|\delta^{-1}]}\|\Bar{V}_{t'}-V_0\|_\infty\leq \tfrac{b}{2}(1-\gamma)^{\tfrac{3}{2}}u_k^{\tfrac{1}{2}}
        \label{eq9:lem_monotone_stoc2}
    \end{align}
    with probability at least $1-\delta/[2(T-1)]$. Then, in view of the definition of $\Tilde{\P}^t\Bar{V}_t$ in~\eqref{def:tildePt_barvt}, we obtain~\eqref{induc_state:stoc2}$(a)$ for $t=t'$ by combining~\eqref{res:lem_induc1_stoc2} and~\eqref{eq9:lem_monotone_stoc2}. Next, we show that~\eqref{induc_state:stoc2}$(b)$ holds for $t=t'$ by
    \begin{align}
        &\big\|\Tilde{Q}^{t'}-\Bar{Q}^{t'}\big\|_{\max}=\gamma\big\|\Tilde{\P}^{t'}V_{t'}-\Tilde{\P}^{t'}\Bar{V}_{t'}+\Tilde{\P}^{t'}\Bar{V}_{t'}-\P\Bar{V}_{t'}\big\|_{\max} \nn\\
        &\leq\gamma\big\|\big[\Tilde{\P}^0V_0+\hat{\P}^{(m_{k,2})}(V_{t'}-V_0)\big]-\big[\Tilde{\P}^0V_0+\hat{\P}^{(m_{k,2})}(\Bar{V}_{t'}-V_0)\big]\big\|_{\max}+\gamma\big\|\Tilde{\P}^{t'}\Bar{V}_{t'}-\P\Bar{V}_{t'}\big\|_{\max} \nn\\
        &\leq\gamma\big\|\hat{\P}^{(m_{k,2})}(V_{t'}-\Bar{V}_{t'})\big\|_{\max}+b(1-\gamma)^{\tfrac{3}{2}}u_k^{\tfrac{1}{2}}\leq\gamma\|V_{t'}-\Bar{V}_{t'}\|_\infty+b(1-\gamma)^{\tfrac{3}{2}}u_k^{\tfrac{1}{2}}\leq b(1-\gamma)^{\tfrac{1}{2}}u_k^{\tfrac{1}{2}},
        \label{eq10:lem_monotone_stoc2}
    \end{align}
    where the second inequality follows from~\eqref{induc_state:stoc2}$(a)$ for $t=t'$, and the fourth inequality uses $\|V_{t'}-\Bar{V}_{t'}\|_\infty\leq b(1-\gamma)^{1/2}u_k^{1/2}$ by~\eqref{induc_state:stoc2}$(c)$ for $t=t'-1$. Moreover, we obtain~\eqref{induc_state:stoc2}$(c)$ for $t=t'$ by
    \begin{align}
        &\|V_{t'+1}(s)-\Bar{V}_{t'+1}(s)\| \nn\\
        &=\big|\min\big\{\big\langle\Tilde{Q}^{t'}(s,\cdot),\pi_{t'+1}(\cdot|s)\big\rangle+h(\pi_{t'+1}(\cdot|s)),V_{t'}(s)\big\}-\min\big\{\big\langle\Bar{Q}^{t'}(s,\cdot),\pi_{t'+1}(\cdot|s)\big\rangle+h(\pi_{t'+1}(\cdot|s)),\Bar{V}_{t'}(s)\big\}\big| \nn\\
        &\leq\max\big\{\big|\big\langle\Tilde{Q}^{t'}(s,\cdot)-\Bar{Q}^{t'}(s,\cdot),\pi_{t'+1}(\cdot|s)\big\rangle\big|,|V_{t'}(s)-\Bar{V}_{t'}(s)|\big\}\leq\max\big\{\big\|\Tilde{Q}^{t'}(s,\cdot)-\Bar{Q}^{t'}(s,\cdot)\big\|_{\infty},|V_{t'}(s)-\Bar{V}_{t'}(s)|\big\} \nn\\
        &\leq b(1-\gamma)^{\tfrac{1}{2}}u_k^{\tfrac{1}{2}} \quad\forall s\in\bS,
        \label{eq11:lem_monotone_stoc2}
    \end{align}
    where the first inequality holds because the map $(x,y)\mapsto\min\{x,y\}$ is $1$-Lipschitz with respect to $\|\cdot\|_\infty$, and the third inequality uses~\eqref{induc_state:stoc2}$(b)$ for $t=t'$ and~\eqref{induc_state:stoc2}$(c)$ for $t=t'-1$.
    
    Now we show~\eqref{induc_state:stoc2}$(d)$ for $t=t'$, i.e., $\Bar{V}_{t'+1}(s)\geq(\Gamma_{\pi_{t'+1}}V_{t'+1})(s)-\beta_{t'+1}(s)$ for any $s\in\bS$, under two cases specified by~\eqref{def:barvt}. The first case $\Bar{V}_{t'+1}(s)=\left\langle c(s,\cdot)+\gamma\P_s\Bar{V}_{t'},\pi_{t'+1}(s)\right\rangle+h(\pi_{t'+1}(s))=(\Gamma_{\pi_{t'+1}}\Bar{V}_{t'})(s)\geq (\Gamma_{\pi_{t'+1}}\Bar{V}_{t'+1})(s)$ trivially follows from $\Bar{V}_{t'} \geq \Bar{V}_{t'+1}$. For the second case $\Bar{V}_{t'+1}(s)=\Bar{V}_{t'}(s)$, we have
    \begin{align}
        \Bar{V}_{t'+1}(s)=\Bar{V}_{t'}(s)\geq (\Gamma_{\pi_{t'}}\Bar{V}_{t'})(s)-\beta_{t'}(s)\geq (\Gamma_{\pi_{t'+1}}\Bar{V}_{t'})(s)-\beta_{t'}(s)-\tfrac{b^2(1-\gamma)u_k}{2\mu(t'+1)}\geq (\Gamma_{\pi_{t'+1}}\Bar{V}_{t'+1})(s)-\beta_{t'+1}(s),
        \label{eq12:lem_monotone_stoc2}
    \end{align}
    where the first inequality follows from~\eqref{induc_state:stoc2}$(d)$ for $t=t'-1$, the second inequality uses~\eqref{eq5:lem_monotone_stoc2} with $t=t'$ and~\eqref{eq10:lem_monotone_stoc2}, and the third inequality applies $\Bar{V}_{t'}\geq\Bar{V}_{t'+1}$ and the definitions of $\beta_{t'}$ and $\beta_{t'+1}$ in~\eqref{def:betak_stoc2}. Hence,~\eqref{induc_state:stoc2}$(d)$ holds for $t=t'$ by combining the two cases.
    
    Finally,~\eqref{induc_state:stoc2}$(d)$ for $t=t'$ directly implies~\eqref{induc_state:stoc2}$(e)$ for $t=t'$ by
    \begin{align}
        \big\|(I-\gamma\P_{\pi_{t'+1}})^{-1}\beta_{t'+1}\big\|_\infty\leq (1-\gamma)^{-1}\|\beta_{t'+1}\|_\infty\leq\tfrac{b^2}{2\mu}u_k\cdot\big(\tsum_{i=1}^{t'+1}\tfrac{1}{i}\big)\leq\tfrac{b^2(\ln T+1)}{2\mu}u_k=\tfrac{u_k}{32},
        \label{eq13:lem_monotone_stoc2}
    \end{align}
    where the second inequality uses the definition of $\beta_{t'+1}$ in~\eqref{def:betak_stoc2}, the third inequality utilizes $t' \leq T-1$, and the equality applies the definition of $b$ in~\eqref{def:b_stoc2}. This concludes the inductive arguments and completes the proof.
\end{proof}
\vgap

Combining Lemma~\ref{lem:converge_stoc2} and Lemma~\ref{lem:monotone_stoc2} yields the following single-epoch convergence result of Algorithm~\ref{algo:VMD3}.

\begin{proposition}
    For any epoch $k\in\{0,1,...,K-1\}$ of Algorithm~\ref{algo:VMD3} and $\delta\in(0,1)$, suppose that $\|\Bar{V}_0-V^*\|_\infty\leq u_k$ and $\|(I-\gamma\P_{\pi^*})^{-1}D_{\pi_0}^{\pi^*}\|_\infty\leq u_k/\mu$ hold, and the algorithmic parameters are set to
    \begin{align}
        &T=\left\lceil\tfrac{22}{1-\gamma}\right\rceil, \quad \eta_t=\tfrac{2}{\mu(t+1)} \ \text{for} \ t=0,1,...,T-1, \nn\\
        &m_{k,1}\geq 200\max\Big\{\tfrac{4(1+\Bar{h})^2}{b^2}(1-\gamma)^{-5}u_k^{-1},(1-\gamma)^{-1}\Big\}\ln(24|\bS||\bA|\delta^{-1}), \nn\\
        &\text{and}\quad m_{k,2}\geq\Big(\tfrac{25\max\{1+\Bar{h},\mu D_0\}}{b^2}+16\Big)(1-\gamma)^{-4}\ln[4(T-1)|\bS||\bA|\delta^{-1}],
        \label{param:prop_stoc2}
    \end{align}
    where $b$ is defined in~\eqref{def:b_stoc2}, then we have
    \begin{align}
        \big\|V^{\pi_T}-V^*\big\|_\infty\leq u_{k+1}=\tfrac{u_k}{2} \quad\text{and}\quad \big\|(I-\gamma\P_{\pi^*})^{-1}D_{\pi_T}^{\pi^*}\big\|_\infty\leq\tfrac{u_{k+1}}{\mu}=\tfrac{u_k}{2\mu}
        \label{res:prop_stoc2}
    \end{align}
    with probability at least $1-\delta$.
    \label{prop:stoc2}
\end{proposition}
\begin{proof}
    We first know from Lemma~\ref{lem:converge_stoc2} that
    \begin{align}
        (\Bar{V}_T-V^*)+\tsum_{t=1}^{T-1}(I-\gamma\P_{\pi^*})(\Bar{V}_t-V^*)+\tfrac{\mu(T+2)}{2}D_{\pi_T}^{\pi^*}\leq\mu D_{\pi_0}^{\pi^*}+\gamma\P_{\pi^*}(\Bar{V}_0-V^*)+\tsum_{t=0}^{T-1}\tfrac{2}{\mu}\left\|\Tilde{Q}^t-\Bar{Q}^t\right\|_{\max}^2\cdot\mathbf{1}_{\bS}.
        \label{eq1:prop_stoc2}
    \end{align}
    In view of Lemma~\ref{lem:monotone_stoc2}, $\|\Tilde{Q}^t-\Bar{Q}^t\|_{\max}\leq b(1-\gamma)^{1/2}u_k^{1/2}$ and $\Bar{V}_{t+1}\geq V^{\pi_{t+1}}-u_k/32\cdot\mathbf{1}_{\bS}$ hold for $t=0,1,...,T-1$ with probability at least $1-\delta$, where $b=\sqrt{\mu/[16(\ln T+1)]}$ by~\eqref{def:b_stoc2}. Together with $\Bar{V}_0\geq\Bar{V}_1\geq\cdots\geq\Bar{V}_T$,~\eqref{eq1:prop_stoc2} can be simplified to
    \begin{align}
        &(I-\gamma\P_{\pi^*})^{-1}\big(V^{\pi_T}-\tfrac{u_k}{32}\cdot\mathbf{1}_{\bS}-V^*\big)+(T-1)\big(V^{\pi_T}-\tfrac{u_k}{32}\cdot\mathbf{1}_{\bS}-V^*\big)+\tfrac{\mu(T+2)}{2}(I-\gamma\P_{\pi^*})^{-1}D_{\pi_T}^{\pi^*} \nn\\
        &\leq (I-\gamma\P_{\pi^*})^{-1}\big[\mu D_{\pi_0}^{\pi^*}+\gamma\P_{\pi^*}(\Bar{V}_0-V^*)+T\cdot\tfrac{2}{\mu}b^2(1-\gamma)u_k\cdot\mathbf{1}_{\bS}\big].
        \label{eq2:prop_stoc2}
    \end{align}
    Since we have $\|\Bar{V}_0-V^*\|_\infty\leq u_k$, $\|(I-\gamma\P_{\pi^*})^{-1}D_{\pi_0}^{\pi^*}\|_\infty\leq u_k/\mu$, and $b^2\leq\mu/16$,~\eqref{eq2:prop_stoc2} implies
    \begin{align}
        &T\big(V^{\pi_T}-V^*\big)+\tfrac{\mu(T+2)}{2}(I-\gamma\P_{\pi^*})^{-1}D_{\pi_T}^{\pi^*}-\big(\tfrac{1}{1-\gamma}+T-1\big)\cdot\tfrac{u_k}{32}\cdot\mathbf{1}_{\bS} \nn\\
        &\leq u_k\cdot\mathbf{1}_{\bS}+\tfrac{u_k}{1-\gamma}\cdot\mathbf{1}_{\bS}+\tfrac{Tu_k}{8}\cdot\mathbf{1}_{\bS}\leq\big(\tfrac{2}{1-\gamma}+\tfrac{T}{8}\big)u_k\cdot\mathbf{1}_{\bS},
        \label{eq3:prop_stoc2}
    \end{align}
    and hence
    \begin{align}
        T\left(V^{\pi_T}-V^*\right)+\tfrac{\mu(T+2)}{2}(I-\gamma\P_{\pi^*})^{-1}D_{\pi_T}^{\pi^*}\leq\left[\tfrac{65}{32}(1-\gamma)^{-1}+\tfrac{5T}{32}\right]u_k\cdot\mathbf{1}_{\bS}.
        \label{eq4:prop_stoc2}
    \end{align}
    Finally, by~\eqref{eq4:prop_stoc2} and the choice of $T$ in~\eqref{param:prop_stoc2}, we obtain~\eqref{res:prop_stoc2} and conclude the proof.
\end{proof}
\vgap

Finally, Theorem~\ref{thm:stoc2} follows by applying Proposition~\ref{prop:stoc2} recursively over all epochs.

\paragraph{Proof of Theorem~\ref{thm:stoc2}}
    By using induction that follows from Proposition~\ref{prop:stoc2}, it is straightforward to obtain that for any $k\in\{0,1,...,K\}$,
    \begin{align}
        \big\|V^{\hat{\pi}_k}-V^*\big\|_\infty\leq u_k \quad\text{and}\quad \big\|(I-\gamma\P_{\pi^*})^{-1}D_{\hat{\pi}_k}^{\pi^*}\big\|_\infty\leq\tfrac{u_k}{\mu}
        \label{eq1:thm_stoc2}
    \end{align}
    hold with probability at least $1-(k/K)\delta$. In view of the choice of $K$ in~\eqref{param:thm_stoc2}, we have $\|V^{\hat{\pi}_K}-V^*\|_\infty\leq u_K\leq\epsilon$ and $\|D_{\hat{\pi}_k}^{\pi^*}\|_\infty\leq\|(I-\gamma\P_{\pi^*})^{-1}D_{\hat{\pi}_k}^{\pi^*}\|_\infty\leq u_K/\mu\leq\epsilon/\mu$.
\vgap
\section{Numerical Illustration}

In this section, we report numerical experiments under three different problem settings. First, in the model misspecification setting where the probability transition kernel is known but inexact, we compare VMD (see Algorithm~\ref{algo:VMD1_modified}) with BVMD (see Algorithm~\ref{algo:VMD1}) and classical value iteration (VI) to demonstrate the stability of the policies produced by VMD. Second, in the unregularized stochastic setting where a generative sampling model is available, we compare SVMD (see Algorithm~\ref{algo:VMD2}) under different stepsize choices to study both policy value convergence and policy stability. Third, in a strongly convex regularized setting where we have access to a generative model, we compare SVMD (see Algorithm~\ref{algo:VMD3}) with the variance-reduced Q-value iteration (QVI) method proposed in~\cite{sidford2018near} and demonstrate faster convergence of SVMD. We use the Tsallis divergence with $p=0.5$ as the Bregman divergence in the updates of BVMD, VMD and SVMD. In the unregularized settings, all methods are initialized with the uniform policy, which selects each action with equal probability at every state.

We generate our MDPs using a 2D GridWorld environment, in which an agent aims to learn the shortest path to the destination on a grid map. The environment consists of $n\times n$ states, with four possible actions available at each state. When an action is taken, the agent moves in the corresponding direction with probability $\alpha\in(0,1]$, and moves in each of the other three directions with probability $(1-\alpha)/3$. The environment also contains $m$ wall grids. If an action causes the agent to hit a wall or the boundary of the map, then the agent remains at its current grid and incurs a cost of $1$. Otherwise, moving to a blank grid incurs a cost of $0.8$, while moving to the unique destination grid incurs a cost of $0$. The wall grids are randomly generated and the destination is reachable from every blank grid on the map. For all experiments, we set $n=20$, $m=40$, and $\alpha=0.9$.

In our experiments, we compute the optimal policies by running value iteration for $1,000$ iterations. We then evaluate the policy value gap and the Bregman divergence to the optimal policy along the policy iterates of each method, and report their values normalized by their respective initial values in the plots. Each plot is obtained by averaging over $20$ independent trials.

\subsection{Inexact Transition Kernel}

Our first experiment compares VMD (Algorithm~\ref{algo:VMD1_modified}), BVMD (Algorithm~\ref{algo:VMD1}) and VI in the model misspecification setting where only an inexact transition kernel is available. For each $(s,a)\in\bS\times\bA$, we uniformly generate each entry of $p\in\R^{|\bS|}$ in $[0,1]$, compute $p_{\mathrm{noise}}=p/\|p\|_1\in\Delta_{\bS}$, and set $\P_{s,a}=(1-\rho)\Bar{\P}_{s,a}+\rho p_{\mathrm{noise}}$ with $\rho=10^{-3}$. Here, $\Bar{\P}$ denotes the true transition kernel, while $\P$ denotes the accessible inexact kernel. We run VMD and BVMD for $10$ epochs with $15$ iterations per epoch, using the stepsize choice prescribed by Theorem~\ref{thm_deter} for BVMD, and considering three different stepsize choices for VMD: one prescribed by Theorem~\ref{thm:deter_inexact}, and two more aggressive choices that are $10^2$ and $10^4$ times larger, respectively. As a comparison, we run VI for $150$ iterations.

\begin{figure}[H]
    \centering

    \includegraphics[width=0.32\textwidth]{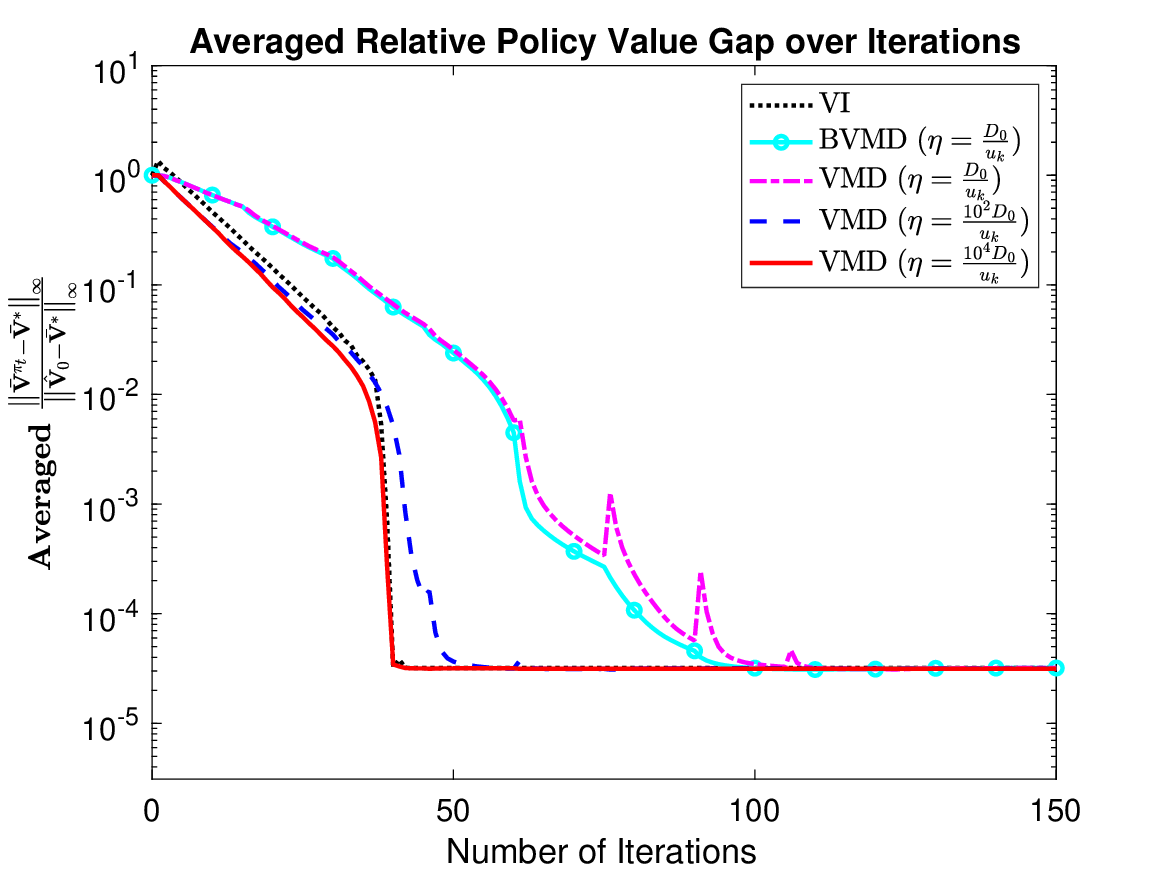}
    \hfill
    \includegraphics[width=0.32\textwidth]{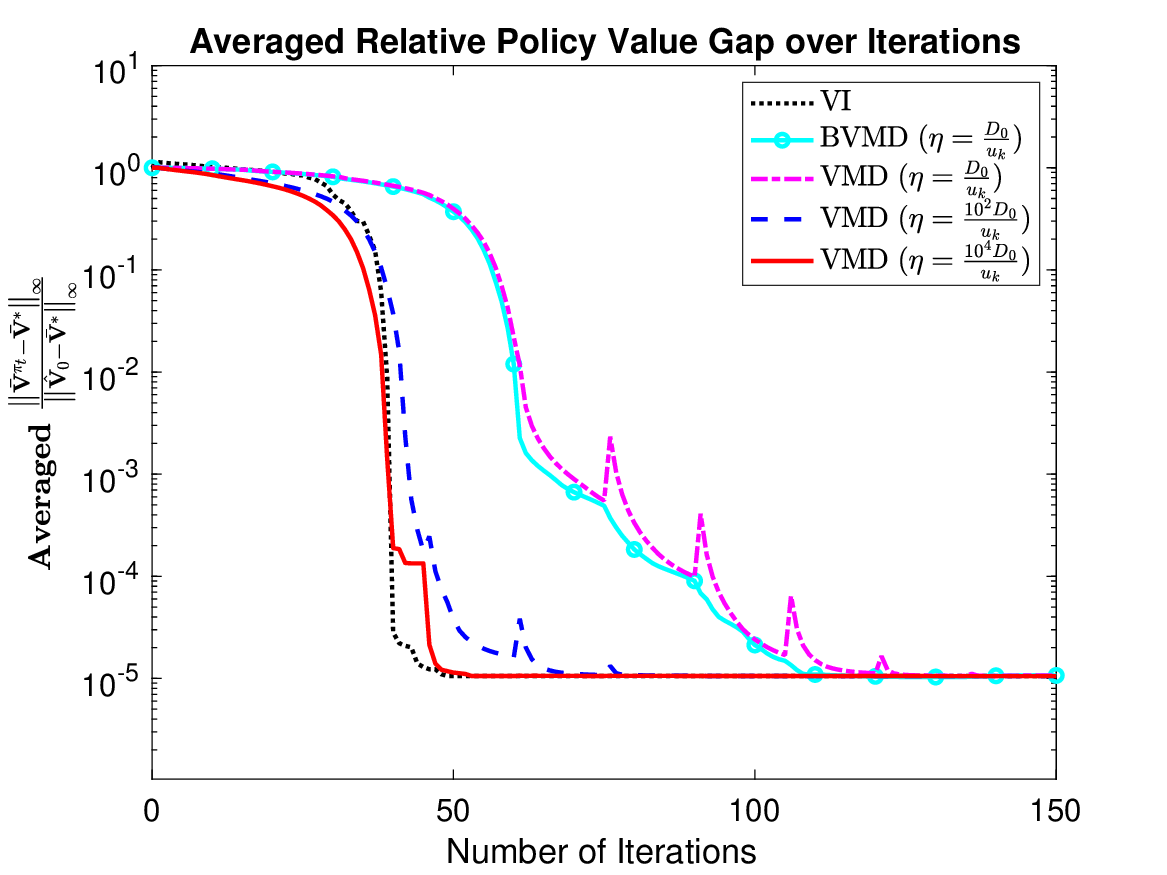}
    \hfill
    \includegraphics[width=0.32\textwidth]{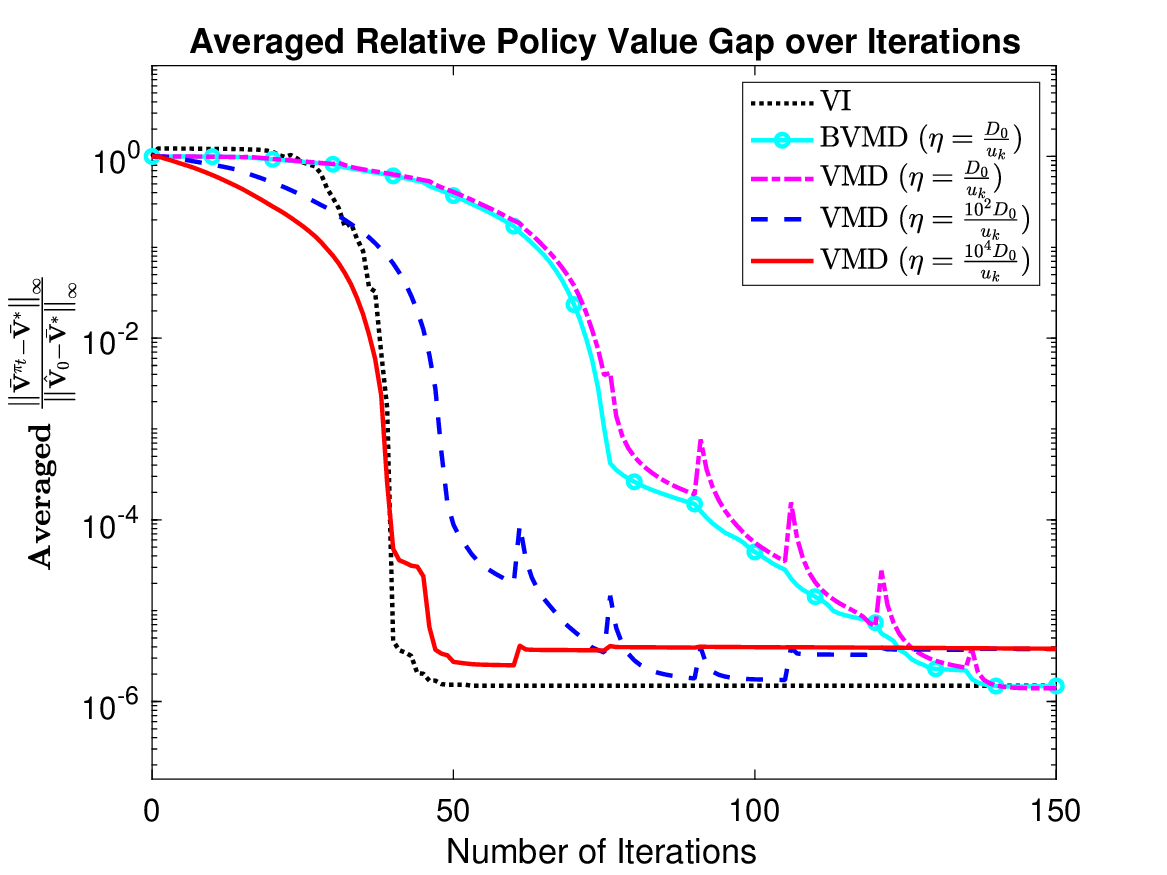}

    \vspace{0.3cm}

    \includegraphics[width=0.32\textwidth]{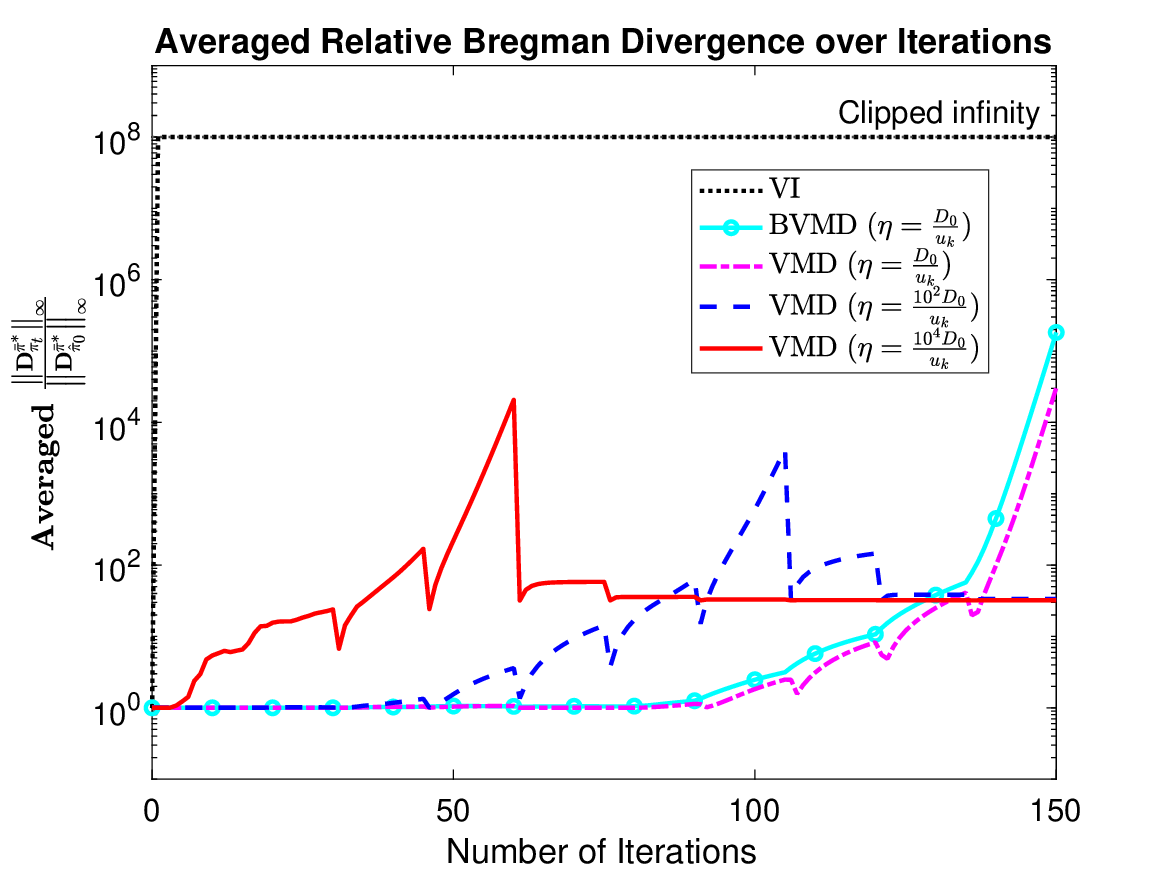}
    \hfill
    \includegraphics[width=0.32\textwidth]{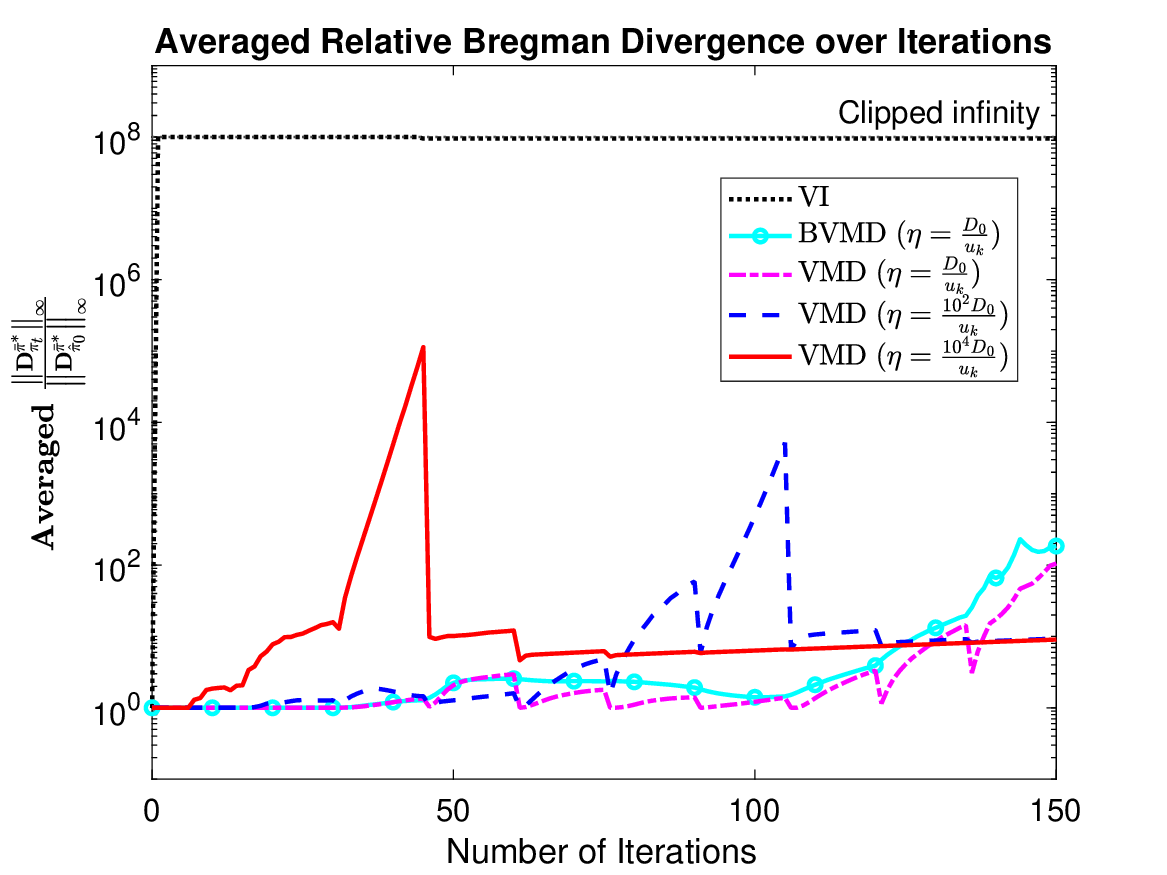}
    \hfill
    \includegraphics[width=0.32\textwidth]{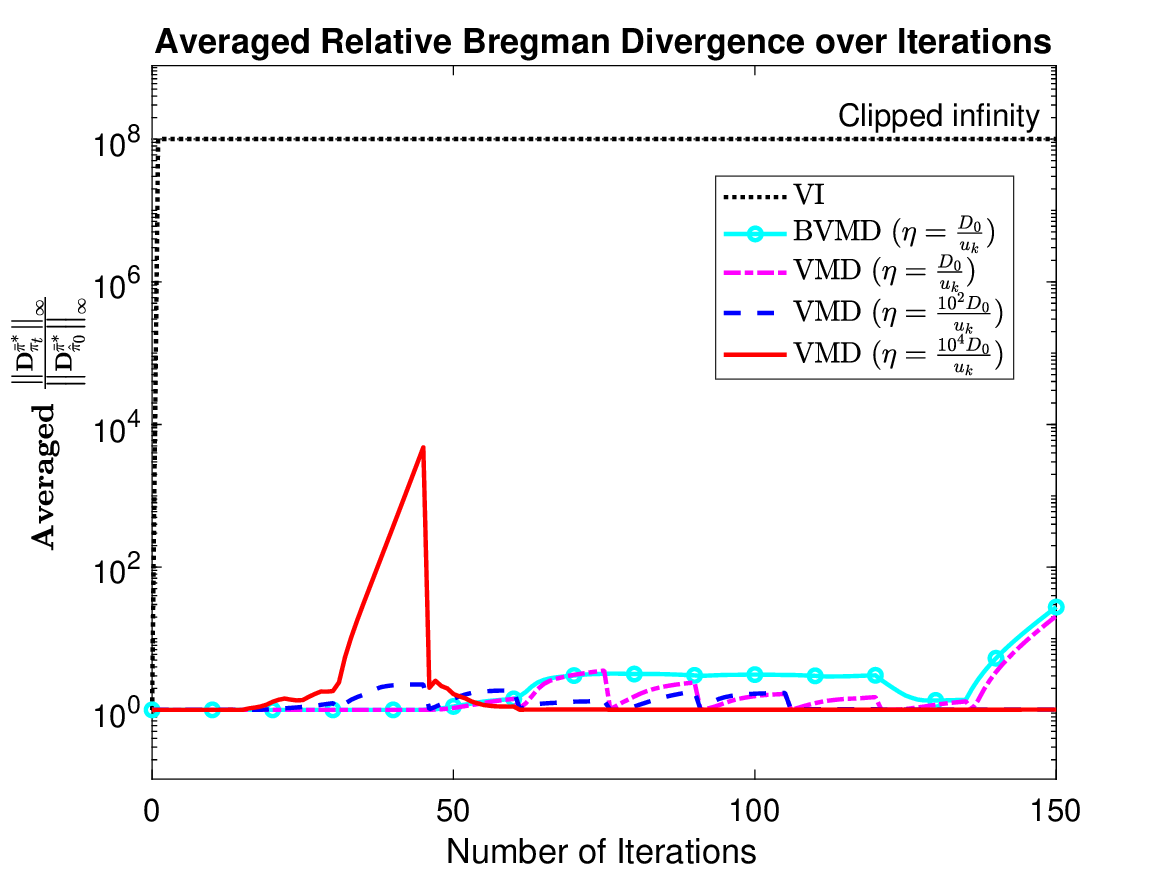}

    \caption{Comparison of VMD under three different choices of stepsizes, BVMD and VI in a model misspecification setting. From left to right, $\gamma$ is set to $0.9$, $0.99$ and $0.999$, respectively.}
    \label{plots:experiment1}
\end{figure}
\vgap

As shown in Figure~\ref{plots:experiment1}, VMD with the more aggressive stepsizes achieve convergence rates comparable to VI. Moreover, all VMD variants and BVMD appear to maintain relatively stable policies, whereas VI does not exhibit policy stability. In particular, under the same stepsize choices, the policies produced by BVMD are less stable than those produced by VMD.

\subsection{Unregularized Stochastic Setting with Generative Model}

In our second experiment, we compare SVMD (Algorithm~\ref{algo:VMD2}) under four different stepsize choices: the stepsizes prescribed by Theorem~\ref{thm:stoc1}, two more aggressive choices that are $10^2$ and $10^4$ times larger, respectively, and a limiting case of infinite stepsizes. In the infinite-stepsize case, SVMD can be viewed as a stochastic value iteration method, where the prox-mapping step in~\eqref{eq1:algoVMD2} is replaced by a hard minimization over all actions. We run SVMD for $10$ epochs with $15$ iterations per epoch. For each epoch $k\in\{0,1,...,9\}$, we empirically set the algorithmic parameters to $m_{k,1}=\min\{\max\{10^3,4^{k+1}\},10^5\}$ and $m_{k,2}=10^3$.

\begin{figure}[H]
    \centering

    \includegraphics[width=0.32\textwidth]{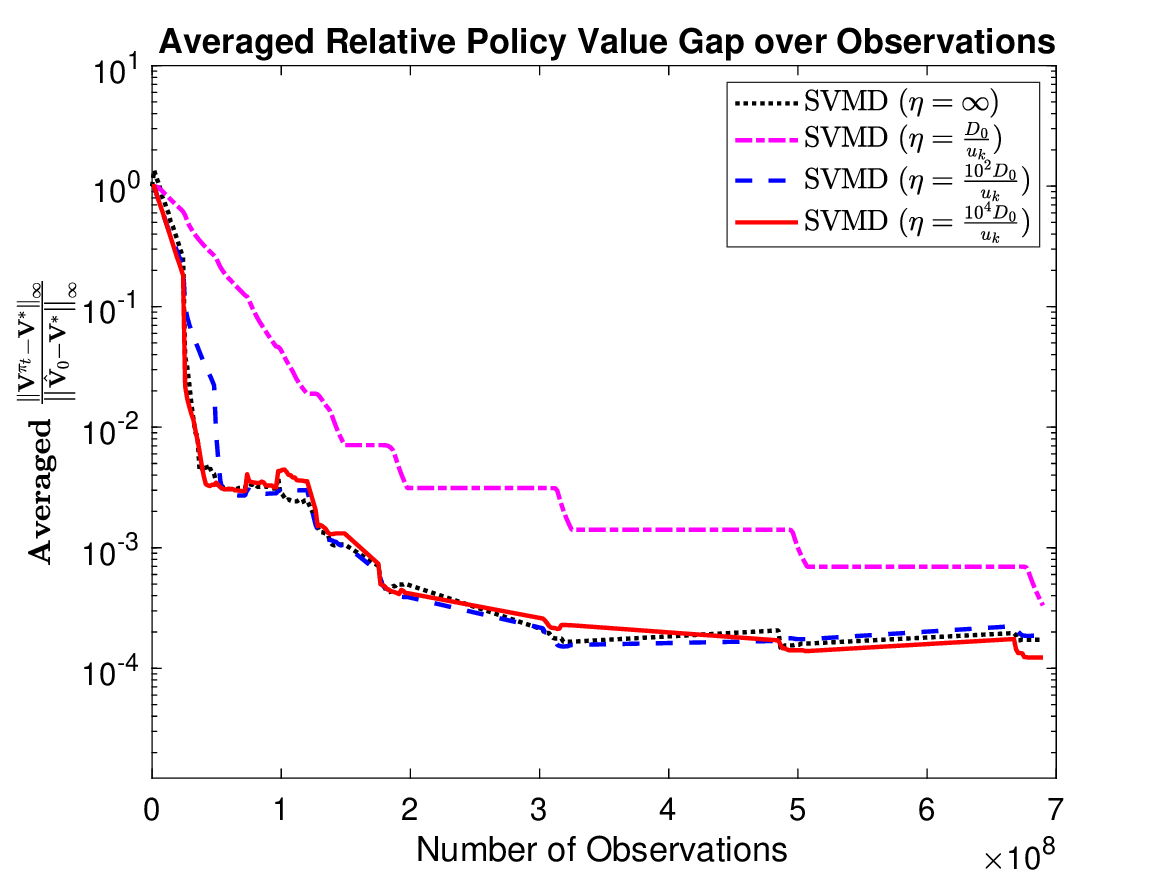}
    \hfill
    \includegraphics[width=0.32\textwidth]{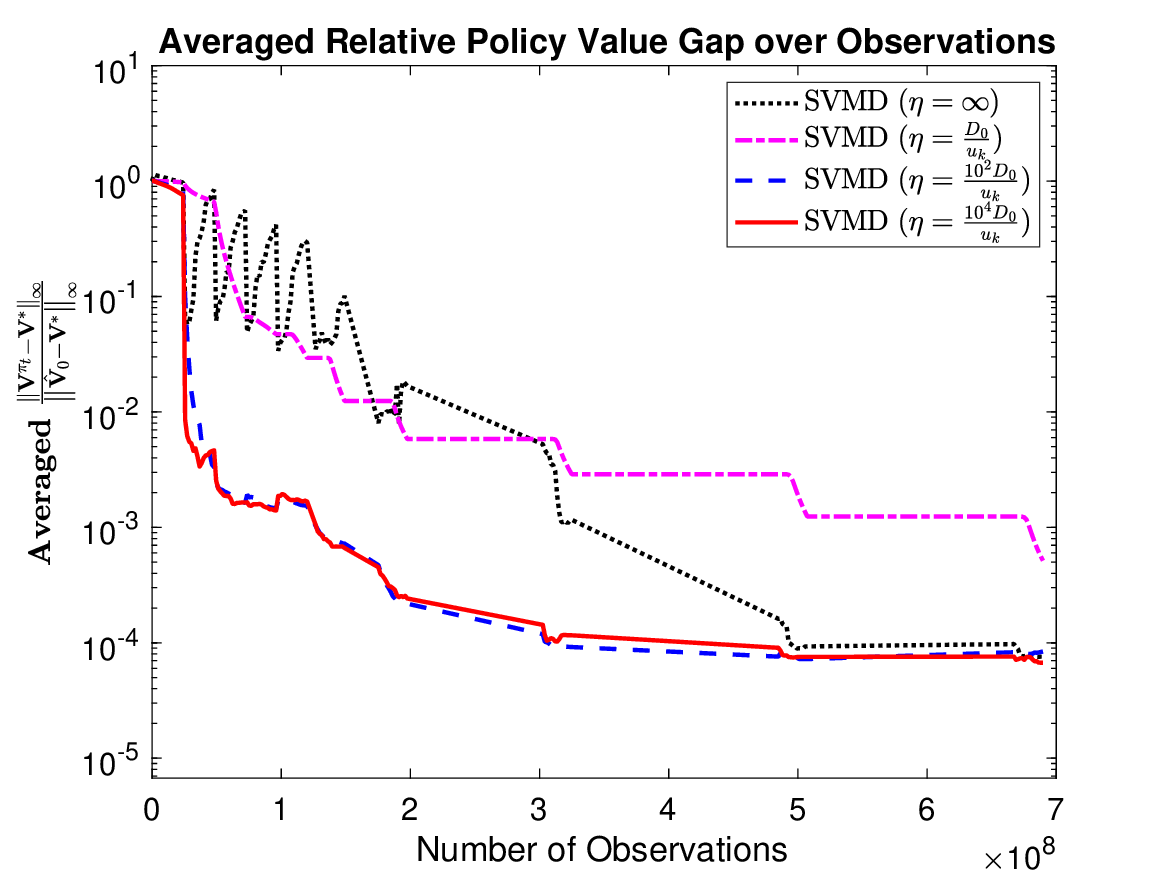}
    \hfill
    \includegraphics[width=0.32\textwidth]{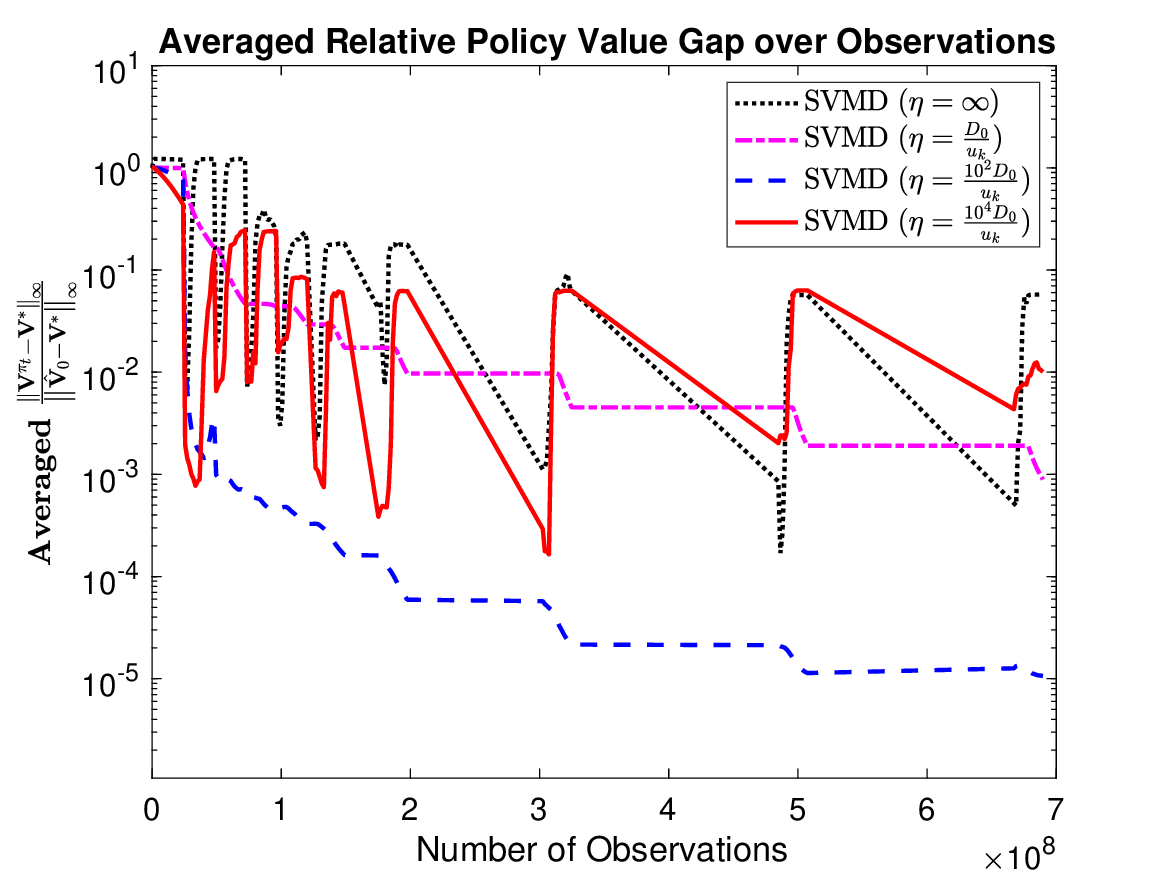}

    \vspace{0.3cm}

    \includegraphics[width=0.32\textwidth]{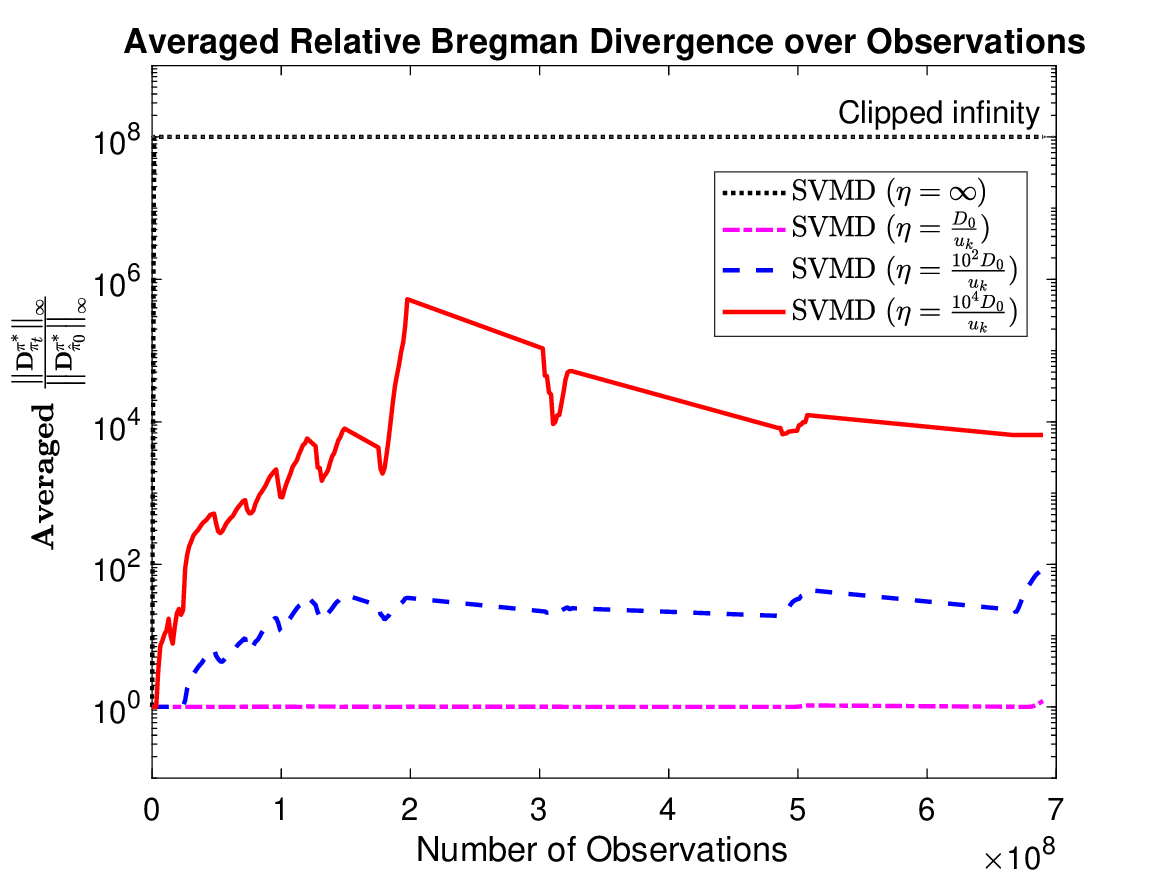}
    \hfill
    \includegraphics[width=0.32\textwidth]{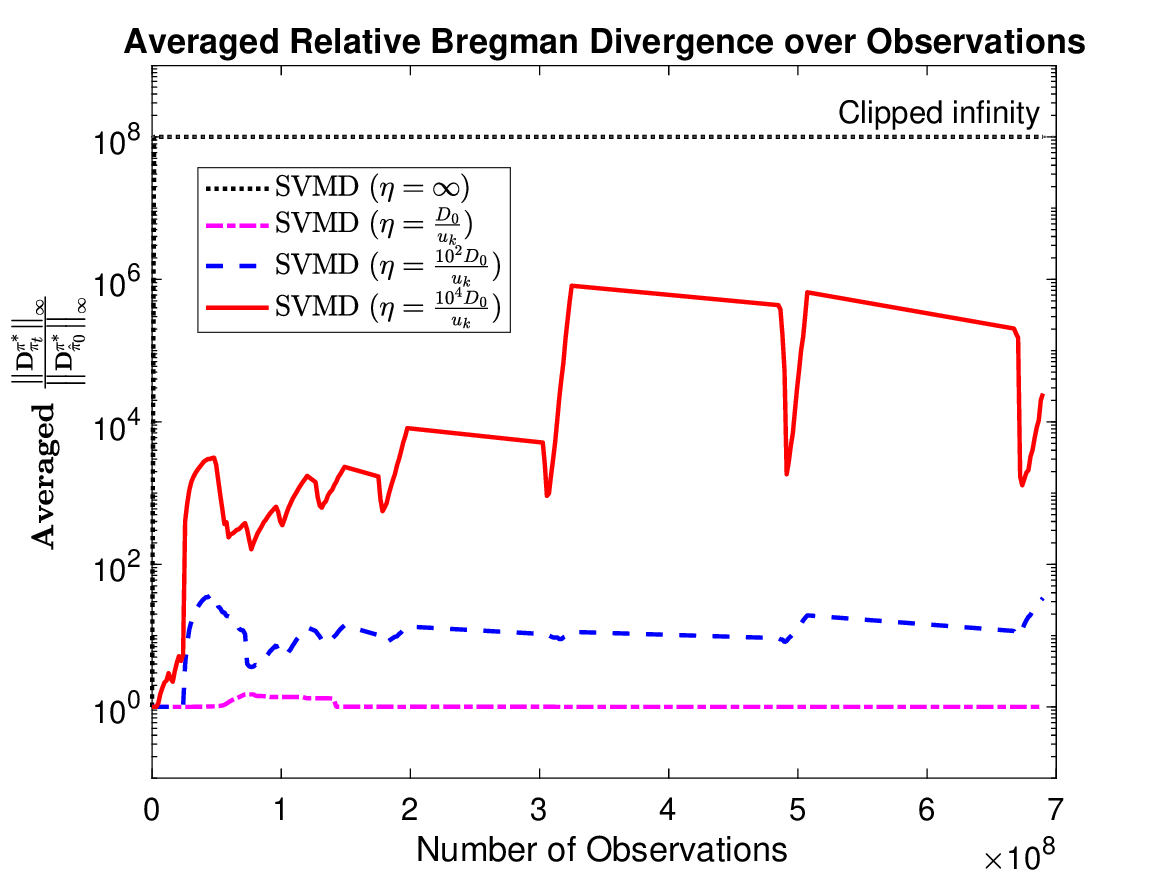}
    \hfill
    \includegraphics[width=0.32\textwidth]{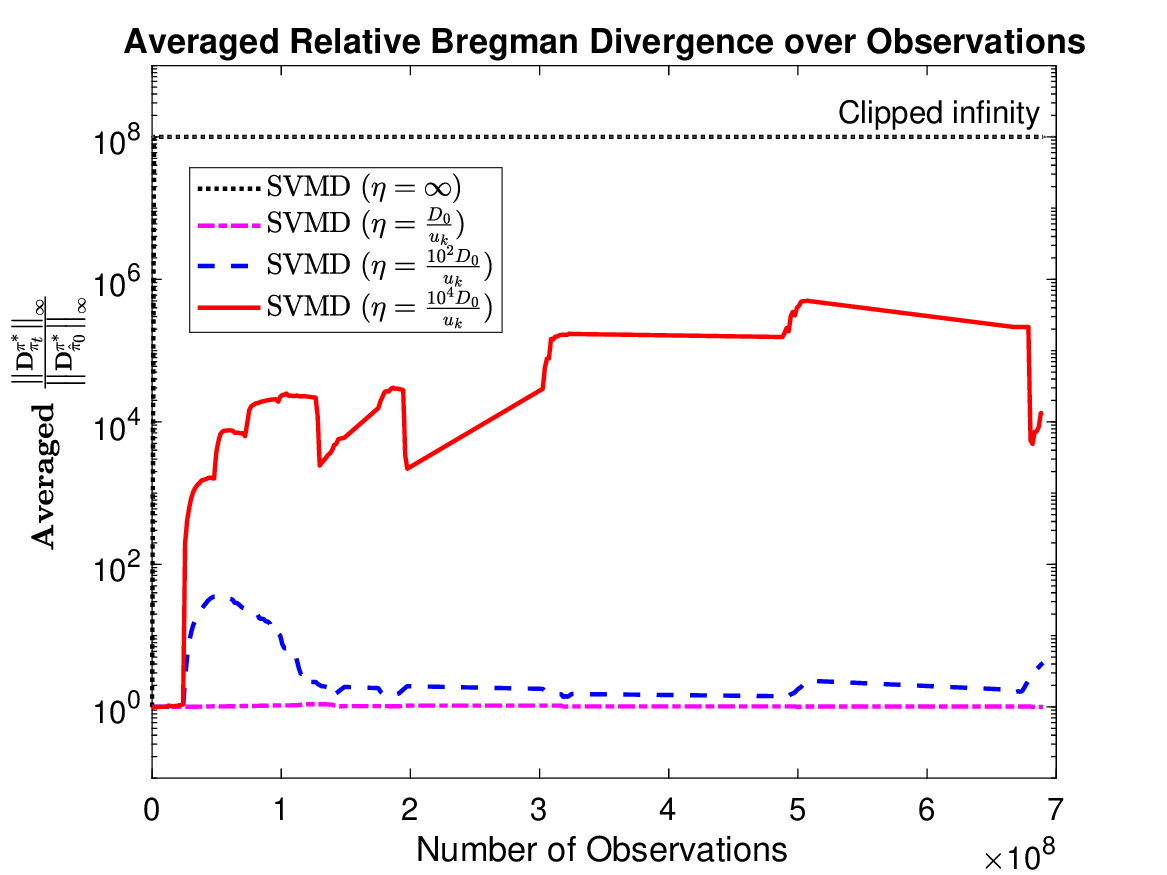}

    \caption{Comparison of SVMD under four different choices of stepsizes in an unregularized stochastic setting, where a generative sampling model is available. From left to right, $\gamma$ is set to $0.9$, $0.99$ and $0.999$, respectively.}
    \label{plots:experiment2}
\end{figure}
\vgap

Figure~\ref{plots:experiment2} shows that SVMD with stepsizes enlarged by a factor of $10^2$ converges the fastest. When the stepsizes are enlarged by a factor of $10^4$ or set to infinity, the convergence of SVMD becomes less stable as $\gamma$ increases. Moreover, the policies produced in all finite-stepsize cases appear to be more stable than those produced in the infinite-stepsize case. In general, SVMD with more conservative stepsizes yields more stable policies.

\subsection{Regularized Stochastic Setting with Generative Model}

The last experiment compares SVMD (Algorithm~\ref{algo:VMD3}) with the variance-reduced Q-value iteration (QVI) method proposed in~\cite{sidford2018near}. In order to adapt QVI to the regularized RL setting, we incorporate the regularization term into its policy update step. Here, we use the Tsallis divergence with $p=0.5$ between a reference policy $\pi_{\mathrm{ref}}(\cdot|s)$ and the current policy $\pi(\cdot|s)$ as a $1$-strongly convex regularizer at each state $s\in\bS$. The reference policy is set to a convex combination of the uniform policy and the optimal policy in the unregularized setting, with weights $0.5$ and $0.5$, respectively. We run SVMD with three different choices of stepsizes: the one prescribed by Theorem~\ref{thm:stoc2} and two more aggressive choices that are $10^2$ and $10^4$ times larger, respectively. QVI does not require a stepsize; however, at each iteration $t$, it computes and adds a bias term $b_t^{(s,a)}$ to each $(s,a)\in\bS\times\bA$ component of the Q-function. We compute these bias terms $b_t$ following~\cite[Algorithm 1]{sidford2018near}. Since $b_t$ can be conservative in practice and may slow down the convergence of QVI, we also run QVI with smaller bias terms $b_t/10^3$, $b_t/10^4$, and $b_t/10^5$, in the problem settings where $\gamma=0.9$, $\gamma=0.99$, and $\gamma=0.999$, respectively. Both SVMD and QVI are initialized at $\pi_{\mathrm{ref}}$, and run for $5$ epochs with $10$ iterations per epoch. The algorithmic parameters are set empirically as $m_{k,1}=10^3\cdot2^{k+1}$ and $m_{k,2}=10^3$ for each epoch $k\in\{0,1,...,4\}$.

\begin{figure}[H]
    \centering

    \includegraphics[width=0.32\textwidth]{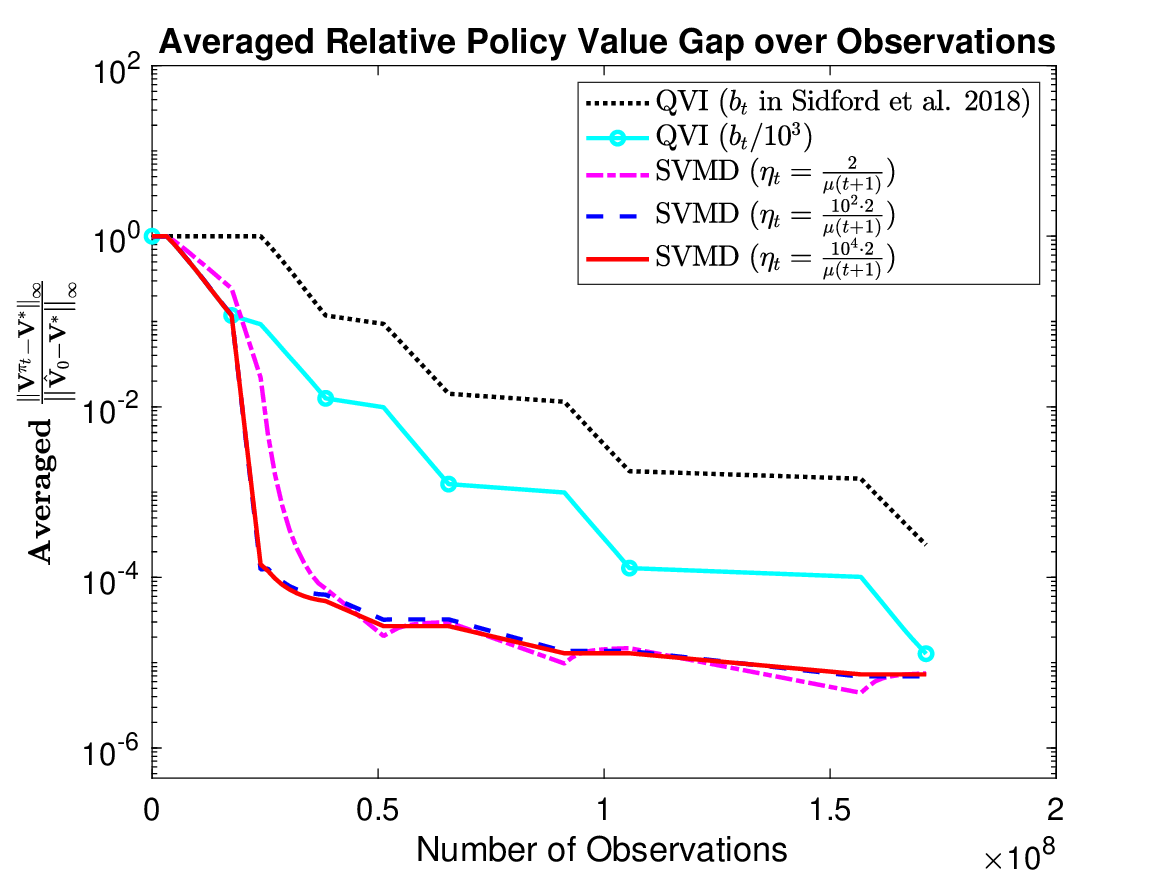}
    \hfill
    \includegraphics[width=0.32\textwidth]{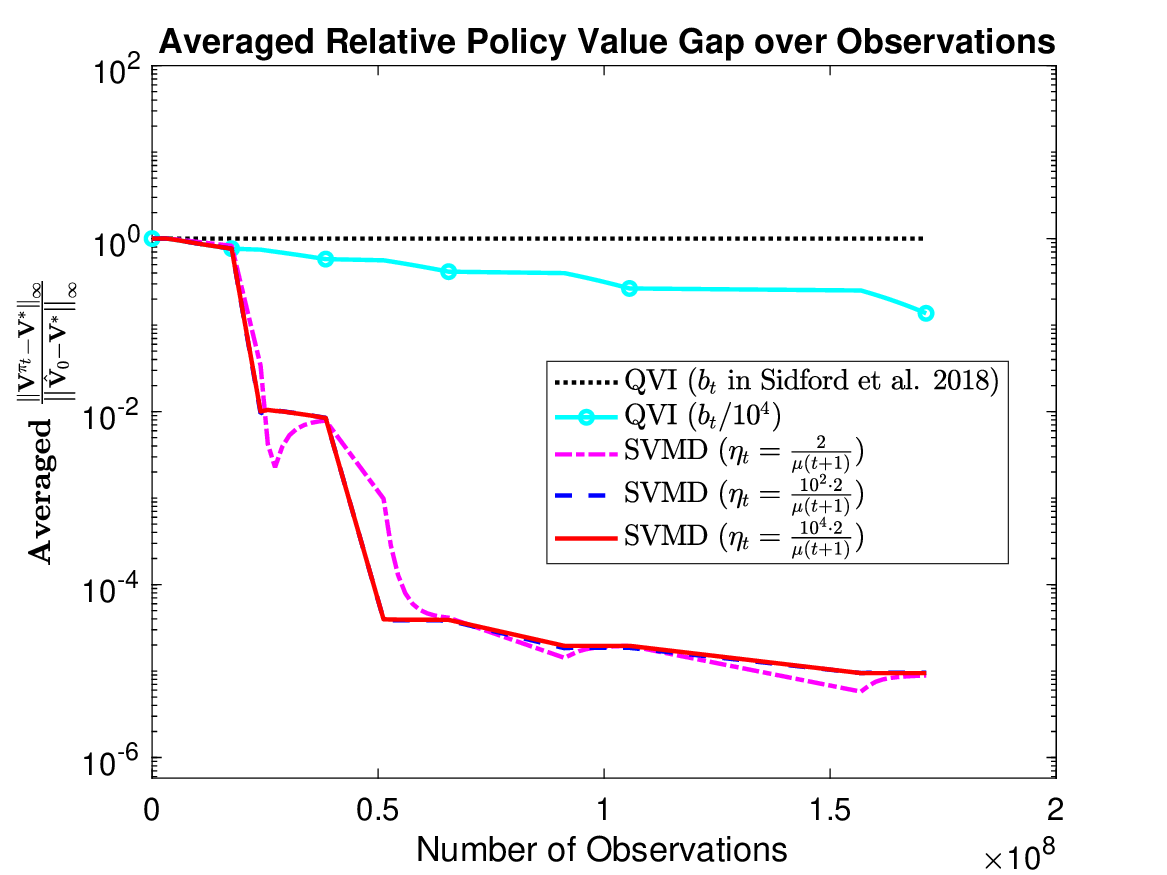}
    \hfill
    \includegraphics[width=0.32\textwidth]{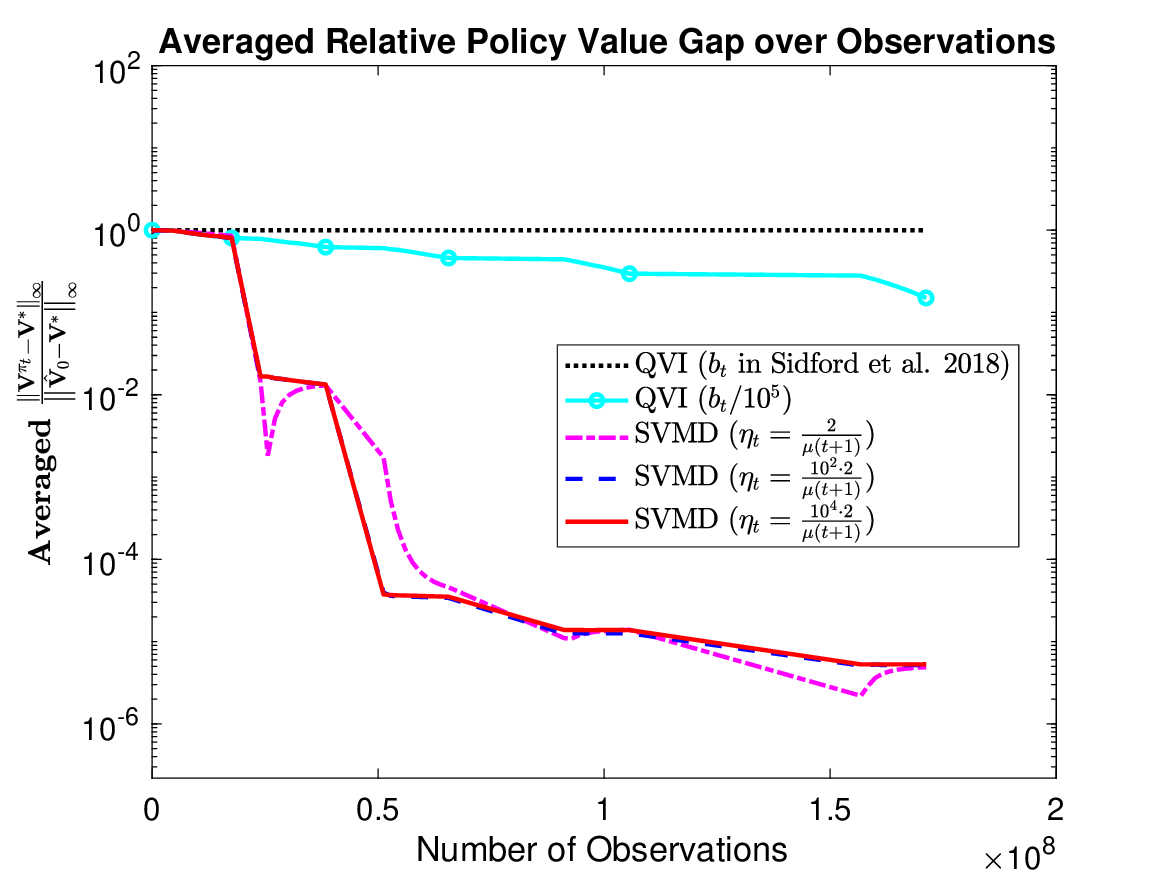}

    \vspace{0.3cm}

    \includegraphics[width=0.32\textwidth]{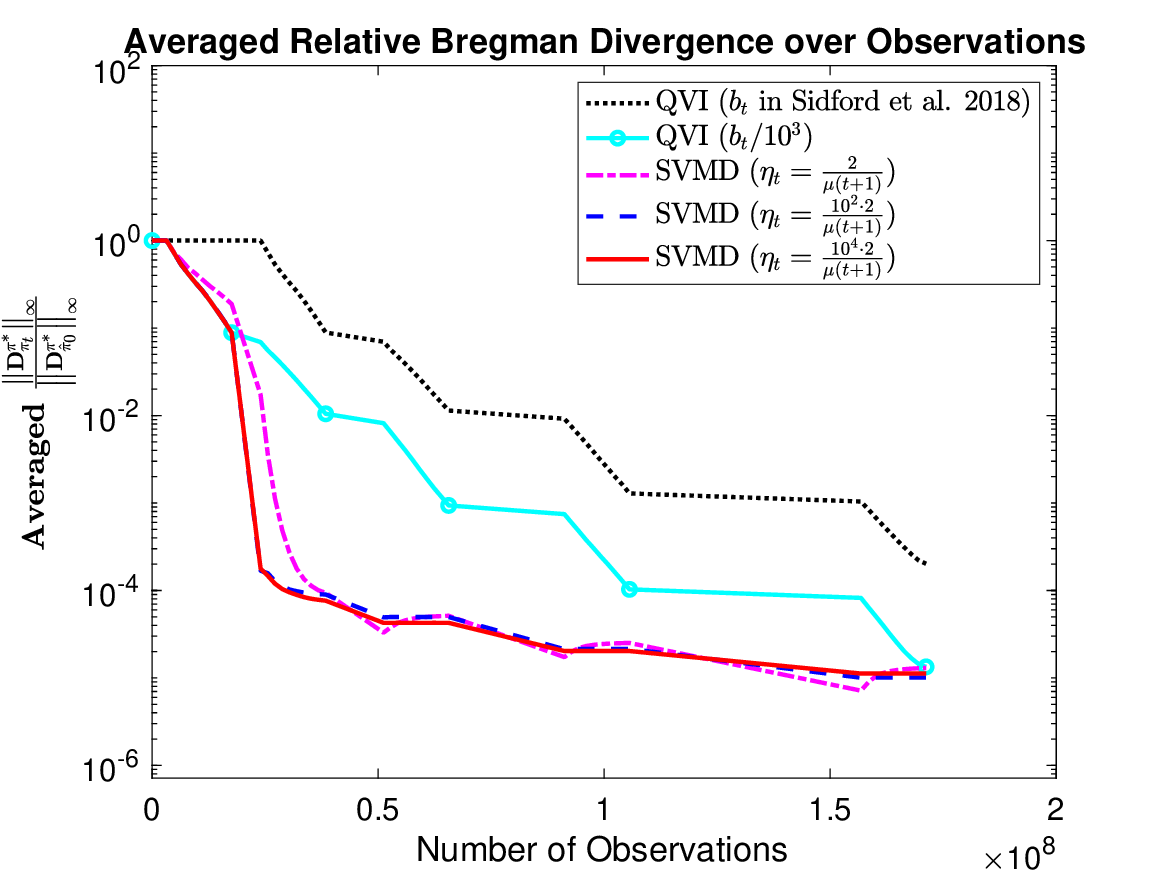}
    \hfill
    \includegraphics[width=0.32\textwidth]{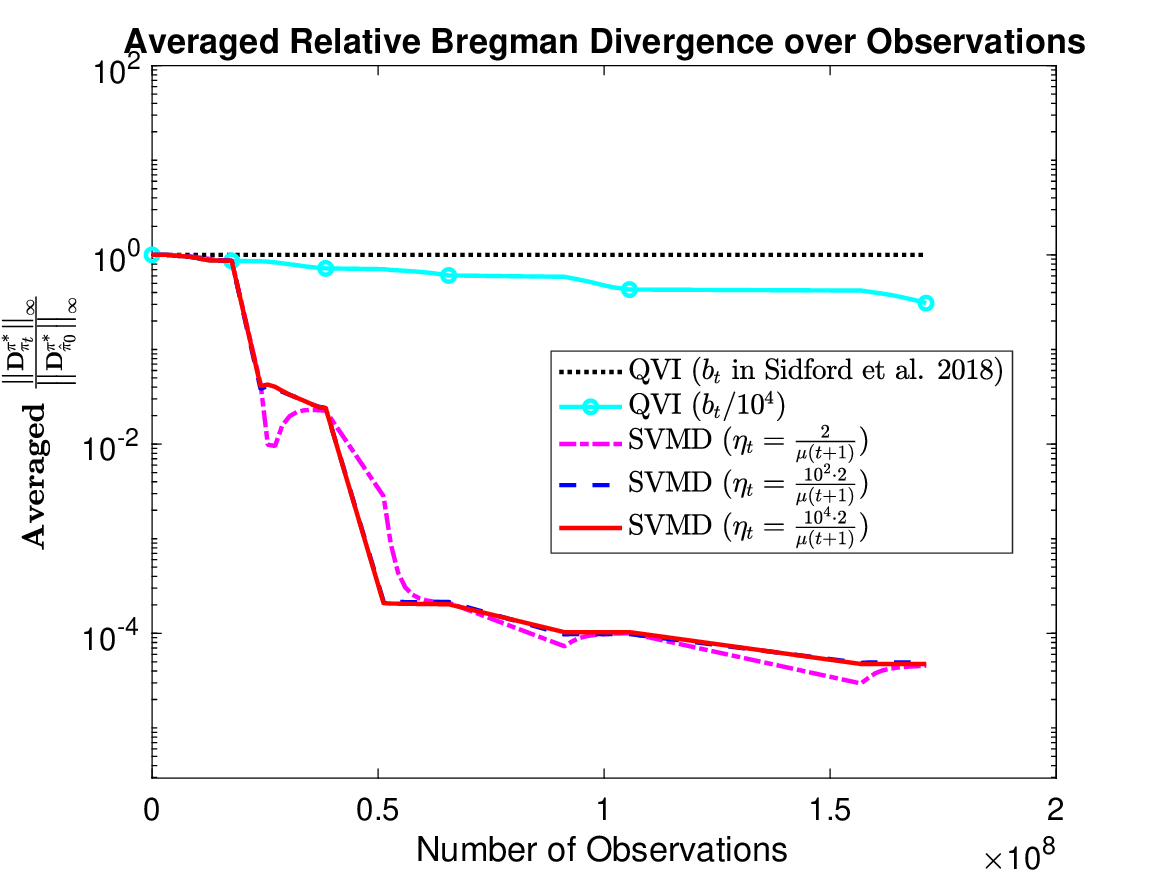}
    \hfill
    \includegraphics[width=0.32\textwidth]{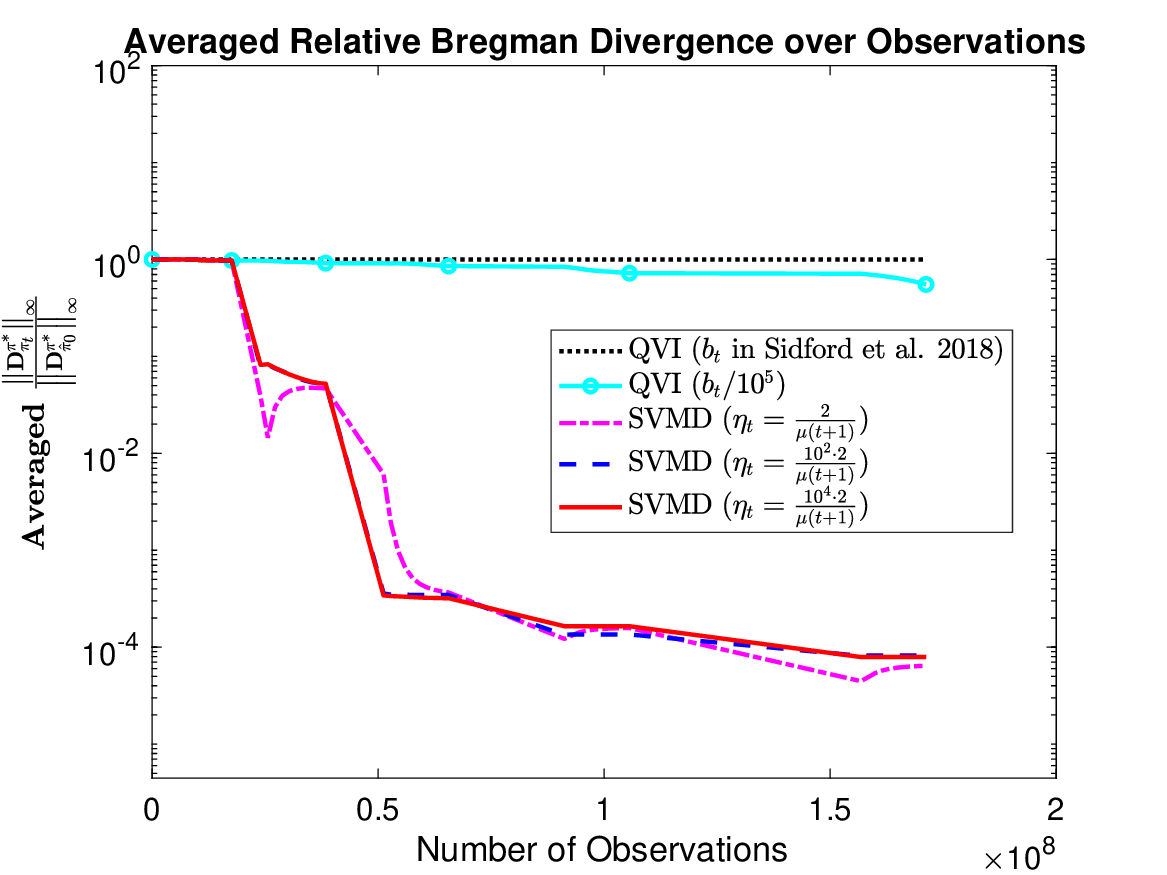}

    \caption{Comparison of SVMD under three different choices of stepsizes and variance-reduced QVI under two difference choices of added bias terms, in a regularized stochastic setting where a generative sampling model is available. From left to right, $\gamma$ is set to $0.9$, $0.99$ and $0.999$, respectively.}
    \label{plots:experiment3}
\end{figure}
\vgap

Figure~\ref{plots:experiment3} suggests that SVMD converges faster than QVI across all tested cases. Moreover, all SVMD variants appear to generate convergent policies, whereas QVI only clearly exhibits policy convergence when $\gamma=0.9$. Such policy convergence becomes less evident for QVI when $\gamma=0.99$ or $0.999$.
\section{Conclusion}
\label{section:conclusion}

In this paper, we propose a new value mirror descent method that converges linearly to an $\epsilon$-optimal policy for DMDPs. We then develop a stochastic value mirror descent method for solving RL problems under a generative sampling model. For RL problems with general convex regularizers, we show that with high probability, our method computes an $\epsilon$-optimal policy using $\Tilde{\cO}(|\bS||\bA|(1-\gamma)^{-3}\epsilon^{-2})$ samples, which is nearly optimal up to logarithmic factors. Moreover, for strongly convex regularized RL problems, our method achieves a sample complexity of $\Tilde{\cO}(|\bS||\bA|(1-\gamma)^{-5}\epsilon^{-1})$. To the best of our knowledge, this is the first value iteration-type method to attain an $\cO(\epsilon^{-1})$ dependence on $\epsilon$. Furthermore, our method guarantees a bounded or convergent Bregman divergence between the output approximate optimal policy and the optimal policy even with the misspecified models. This property enables the policy produced by our offline pre-training procedure to serve as a stable starting point for continual online adaptation and on-policy learning.

\bibliographystyle{plain}
\bibliography{bibliography}

\section*{Appendix A: Deferred Proofs}

\paragraph{Proof of Lemma~\ref{lem:total_variance}}

Although the Bellman-like equation established in this proof is similar to those in~\cite{gheshlaghi2013minimax,munos1999influence,sidford2018near,sobel1982variance}, the regularized term $h(\cdot)$ is incorporated in our case, and our policy $\pi$ is chosen as a randomized policy in general, while those works only discuss deterministic policies. Moreover, unlike~\cite{gheshlaghi2013minimax,sidford2018near}, we study the variance of the discounted cumulative costs for every state, rather than for every state-action pair.

For any policy $\pi$ and $s\in\bS$, we first define $\Sigma^{\pi}\in\R^{|\bS|}$ with each entry
\begin{align}
    \Sigma^{\pi}(s):=\E_{\pi}\Big[\big[\tsum_{t=0}^{\infty}\gamma^t\left(c_{\pi}(s_t)+h(\pi(\cdot|s_t))\right)-V^{\pi}(s)\big]^2\mid s_0=s\Big],
    \label{def:Sigmapi}
\end{align}
where $a_t\sim\pi(\cdot|s_t)$ for all $t=0,1,2,...$. Notice that $\Sigma^{\pi}(s)$ represents the variance of the discounted cumulative terms $c_{\pi}(s_t)+h(\pi(\cdot|s_t))$ over infinite steps when the initial $s_0=s$ is given, and the following relation holds:
\begin{align}
    \Sigma^{\pi}(s)=\E_{\pi}\Big[\big[\tsum_{t=0}^{\infty}\gamma^t(c_{\pi}(s_t)+h(\pi(\cdot|s_t)))\big]^2\mid s_0=s\Big]-(V^{\pi}(s))^2.
    \label{eq1:lem_totalvar}
\end{align}

For any policy $\pi$, let us now prove the Bellman equation below for the variance:
\begin{align}
    \Sigma^{\pi}=\gamma^2\sigma^{\pi}_{V^{\pi}}+\gamma^2\P_{\pi}\Sigma^{\pi}.
    \label{Bellman_eq_var}
\end{align}
To this end, we first simplify the first term on the right-hand side of~\eqref{def:Sigmapi}. Specifically, we have
\begin{align}
    &\E_{\pi}\Big[\big[\tsum_{t=0}^{\infty}\gamma^t(c_{\pi}(s_t)+h(\pi(\cdot|s_t)))\big]^2\mid s_0=s\Big] \nn\\
    &=\tsum_{a\in\bA,s'\in\bS}\pi(a|s)\P_{s,a}(s')\E_{\pi}\Big[\big[c_{\pi}(s)+h(\pi(\cdot|s))+\gamma\tsum_{t=1}^{\infty}\gamma^{t-1}(c_{\pi}(s_t)+h(\pi(\cdot|s_t)))\big]^2\mid s_1=s'\Big] \nn\\
    &=\tsum_{a\in\bA,s'\in\bS}\pi(a|s)\P_{s,a}(s')\Big\{(c_{\pi}(s)+h(\pi(\cdot|s)))^2+\gamma^2\E_{\pi}\Big[\big[\tsum_{t=1}^{\infty}\gamma^{t-1}(c_{\pi}(s_t)+h(\pi(\cdot|s_t)))\big]^2\mid s_1=s'\Big] \nn\\
    &\quad +2\gamma (c_{\pi}(s)+h(\pi(\cdot|s)))\E_{\pi}\big[\tsum_{t=1}^{\infty}\gamma^{t-1}(c_{\pi}(s_t)+h(\pi(\cdot|s_t)))\mid s_1=s'\big]\Big\} \nn\\
    &=\underbrace{(c_{\pi}(s)+h(\pi(\cdot|s)))^2}_{A}+\gamma^2\tsum_{a\in\bA,s'\in\bS}\pi(a|s)\P_{s,a}(s')\big[\Sigma^{\pi}(s')+\big(V^{\pi}(s')\big)^2\big] \nn\\
    &\quad +\underbrace{2\gamma(c_{\pi}(s)+h(\pi(\cdot|s)))\tsum_{a\in\bA,s'\in\bS}\pi(a|s)\P_{s,a}(s')V^{\pi}(s')}_{B},
    \label{eq2:lem_totalvar}
\end{align}
where the last equality utilizes~\eqref{eq1:lem_totalvar}. In~\eqref{eq2:lem_totalvar}, since we know that
\begin{align}
    A+B&=\big(c_{\pi}(s)+h(\pi(\cdot|s))+\gamma\tsum_{a\in\bA,s'\in\bS}\pi(a|s)\P_{s,a}(s')V^{\pi}(s')\big)^2-\big(\gamma\tsum_{a\in\bA,s'\in\bS}\pi(a|s)\P_{s,a}(s')V^{\pi}(s')\big)^2 \nn\\
    &=\big(V^{\pi}(s)\big)^2-\gamma^2\big(\tsum_{s'\in\bS}\P_{\pi}(s,s')V^{\pi}(s')\big)^2,
    \label{eq3:lem_totalvar}
\end{align}
then incorporating~\eqref{eq3:lem_totalvar} into~\eqref{eq2:lem_totalvar} yields
\begin{align}
    &\E_{\pi}\Big[\big[\tsum_{t=0}^{\infty}\gamma^t(c_{\pi}(s_t)+h(\pi(\cdot|s_t)))\big]^2\mid s_0=s\Big] \nn\\
    &=\big(V^{\pi}(s)\big)^2+\gamma^2\tsum_{s'\in\bS}\P_{\pi}(s,s')\Sigma^{\pi}(s')+\Big[\gamma^2\tsum_{s'\in\bS}\P_{\pi}(s,s')\big(V^{\pi}(s')\big)^2-\gamma^2\big(\tsum_{s'\in\bS}\P_{\pi}(s,s')V^{\pi}(s')\big)^2\Big] \nn\\
    &=\big(V^{\pi}(s)\big)^2+\gamma^2\P_{\pi}(s,\cdot)\Sigma^{\pi}+\gamma^2\sigma^{\pi}_{V^{\pi}}.
    \label{eq4:lem_totalvar}
\end{align}
Therefore,~\eqref{Bellman_eq_var} holds by combining~\eqref{eq1:lem_totalvar} and~\eqref{eq4:lem_totalvar}. It then follows from~\eqref{Bellman_eq_var} that
\begin{align}
    \gamma^2(I-\gamma^2\P_{\pi})^{-1}\sigma^{\pi}_{V^{\pi}}=\Sigma^{\pi}\leq\tfrac{(1+\Bar{h})^2}{(1-\gamma)^2}\cdot\mathbf{1}_{\bS},
    \label{eq5:lem_totalvar}
\end{align}
where the inequality uses the definition of $\Sigma^{\pi}$ in~\eqref{def:Sigmapi}. Finally, we directly apply~\cite[Lemma C.3]{sidford2018near}, to~\eqref{eq5:lem_totalvar} and obtain
\begin{align}
    \big\|(I-\gamma\P_{\pi})^{-1}\sqrt{\sigma^{\pi}_{V^{\pi}}}\big\|_\infty^2\leq\tfrac{1+\gamma}{1-\gamma}\big\|(I-\gamma^2\P_{\pi})^{-1}\sigma^{\pi}_{V^{\pi}}\big\|_\infty\leq\tfrac{(1+\Bar{h})^2(1+\gamma)}{\gamma^2(1-\gamma)^3}.
    \label{eq6:lem_totalvar}
\end{align}
Note that a result similar to~\eqref{eq6:lem_totalvar} also appeared earlier in~\cite{gheshlaghi2013minimax}. This completes the proof.
\vgap

\paragraph{Proof of Lemma~\ref{lem:error_pv_stoc1}$(b)$}

First, for any epoch $k\in\{1,2,...,K-1\}$ of Algorithm~\ref{algo:VMD2}, let us denote $m=m_{k,1},\pi=\pi_0,V=V^{\pi_0},\Tilde{V}=V_0$ and $\Tilde{\P}=\Tilde{P}^0$ for notational simplicity. According to the definitions $V=c_{\pi}+\gamma\P_{\pi}V+h(\pi)$ and $\Tilde{V}=c_{\pi}+\gamma\Tilde{\P}_{\pi}\Tilde{V}+h(\pi)$, we can obtain
\begin{align}
    \Tilde{V}-V=\gamma(I-\gamma\P_{\pi})^{-1}(\Tilde{\P}_{\pi}-\P_{\pi})\Tilde{V}.
    \label{eq1:appdx_error}
\end{align}
Notice that the dependence between $\Tilde{\P}$ and $\Tilde{V}$ prevents us from directly applying Lemma~\ref{lem:concentration_bounds}$(b)$ to bound the error term $(\Tilde{\P}_{\pi}-\P_{\pi})\Tilde{V}$ on the right-hand side of~\eqref{eq1:appdx_error}. To handle this difficulty, we adopt and develop the result in~\cite{pananjady2020instance}. Specifically, we separate~\eqref{eq1:appdx_error} as
\begin{align}
    \Tilde{V}-V=\gamma(I-\gamma\P_{\pi})^{-1}\big[(\Tilde{\P}_{\pi}-\P_{\pi})V+(\Tilde{\P}_{\pi}-\P_{\pi})(\Tilde{V}-V)\big],
    \label{eq2:appdx_error}
\end{align}
where $(\Tilde{\P}_{\pi}-\P_{\pi})V$ can be bounded using Lemma~\ref{lem:concentration_bounds}$(b)$, and we focus on the term $(\Tilde{\P}_{\pi}-\P_{\pi})(\Tilde{V}-V)$. To properly decouple the dependence between $\Tilde{\P}$ and $\Tilde{V}$, we define a perturbed probability transition kernel $\widehat{\P}^{(s,a)}$ by
\begin{align}
    \widehat{\P}^{(s,a)}_{s',a'}:=\begin{cases}
        \Tilde{\P}_{s',a'} &\text{if } (s',a')\neq (s,a); \\
        \P_{s',a'} &\text{otherwise}
    \end{cases}
    \label{def:perturb_transker}
\end{align}
for all $(s,a)\in\bS\times\bA$. Accordingly, we define
\begin{align}
    \widehat{V}^{(s,a)}:=\big(I-\gamma\widehat{\P}^{(s,a)}_{\pi}\big)^{-1}(c_{\pi}+h(\pi))
    \label{def:perturb_v}
\end{align}
so that $\widehat{V}^{(s,a)}=c_{\pi}+\gamma\widehat{\P}^{(s,a)}_{\pi}\widehat{V}^{(s,a)}+h(\pi)$ is satisfied, where $\widehat{\P}^{(s,a)}_{\pi}\in\R^{|\bS|\times|\bS|}$ is defined by
\begin{align}
    \widehat{\P}_{\pi}^{(s,a)}(s_1,s_2):=\tsum_{a'\in\bA}\pi(a'|s_1)\widehat{\P}^{(s,a)}(s_2|s_1,a').
    \label{def:perturb_transker_pi}
\end{align}
for all $s_1,s_2\in\bS$. Next, we further decouple the term $(\Tilde{\P}_{\pi}-\P_{\pi})(\Tilde{V}-V)$ in~\eqref{eq2:appdx_error}. Specifically, for any $(s,a)\in\bS\times\bA$, we have
\begin{align}
    (\Tilde{\P}_{s,a}-\P_{s,a})^{\tr}(\Tilde{V}-V)=\underbrace{(\Tilde{\P}_{s,a}-\P_{s,a})^{\tr}\big(\widehat{V}^{(s,a)}-V\big)}_{A}+\underbrace{(\Tilde{\P}_{s,a}-\P_{s,a})^{\tr}\big(\Tilde{V}-\widehat{V}^{(s,a)}\big)}_{B}.
    \label{eq3:appdx_error}
\end{align}
Since $\Tilde{\P}_{s,a}$ are independently computed for all $(s,a)\in\bS\times\bA$, then in view of~\eqref{def:perturb_transker} and~\eqref{def:perturb_transker_pi}, $\widehat{\P}^{(s,a)}_{\pi}$ is independent of $\Tilde{\P}_{s,a}$. As a consequence, $\widehat{V}^{(s,a)}$ computed from $\widehat{\P}^{(s,a)}_{\pi}$ (see~\eqref{def:perturb_v}) is also independent of $\Tilde{\P}_{s,a}$, and hence for the term $A$ in~\eqref{eq3:appdx_error}, $\Tilde{\P}_{s,a}-\P_{s,a}$ and $\widehat{V}^{(s,a)}-V$ are independent. Moreover, for the term $B$ in~\eqref{eq3:appdx_error}, $\widehat{V}^{(s,a)}$ can be regarded as a perturbed policy value vector that is close to $\Tilde{V}$.

Below we divide the rest of this proof into three parts. In the first two parts, we provide high probability bounds for $A$ and $B$ in~\eqref{eq3:appdx_error}, respectively. Then, utilizing these two bounds, we complete the proof of Lemma~\ref{lem:error_pv_stoc1}$(2)$.
\vgap

\paragraph{Bound for $A$ in~\eqref{eq3:appdx_error}} Let us first consider the term $A$. Due to the independence between $\Tilde{\P}_{s,a}-\P_{s,a}$ and $\widehat{V}^{(s,a)}-V$, we obtain from Lemma~\ref{lem:concentration_bounds}$(a)$ that
\begin{align}
    |A|\leq\sqrt{2}\big\|\widehat{V}^{(s,a)}-V\big\|_\infty m^{-\tfrac{1}{2}}\sqrt{\ln(6|\bS||\bA|\delta^{-1})}\leq\tfrac{1}{4}(1-\gamma)\big\|\widehat{V}^{(s,a)}-V\big\|_\infty
    \label{eq4:appdx_error}
\end{align}
holds with probability at least $1-\delta/(3|\bS||\bA|)$, where the second inequality utilizes $m\geq 32(1-\gamma)^{-2}\sqrt{\ln(6|\bS||\bA|\delta^{-1})}$ by~\eqref{m_error_pvki_stoc1}. In order to further derive an upper bound for $|A|$ in terms of $\|\Tilde{V}-V\|_\infty$, it follows from~\eqref{eq4:appdx_error} that
\begin{align}
    \big|(\Tilde{\P}_{s,a}-\P_{s,a})^{\tr}\widehat{V}^{(s,a)}\big|\leq\tfrac{1}{4}(1-\gamma)\big\|\widehat{V}^{(s,a)}-V\big\|_\infty+\big|(\Tilde{\P}_{s,a}-\P_{s,a})^{\tr}V\big|.
    \label{eq5:appdx_error}
\end{align}
Moreover, by the equations $\Tilde{V}=c_{\pi}+\gamma\Tilde{\P}_{\pi}\Tilde{V}+h(\pi)$ and $\widehat{V}^{(s,a)}=c_{\pi}+\gamma\widehat{\P}^{(s,a)}_{\pi}\widehat{V}^{(s,a)}+h(\pi)$, we have
\begin{align}
    \Tilde{V}-\widehat{V}^{(s,a)}&=\gamma(I-\gamma\Tilde{\P}_{\pi})^{-1}(\Tilde{\P}_{\pi}-\widehat{\P}^{(s,a)}_{\pi})\widehat{V}^{(s,a)}=\gamma(I-\gamma\Tilde{\P}_{\pi})^{-1}\big[\pi(a|s)e_s(\Tilde{\P}_{s,a}-\P_{s,a})^{\tr}\big]\widehat{V}^{(s,a)} \nn\\
    &=\gamma\pi(a|s)\big[(\Tilde{\P}_{s,a}-\P_{s,a})^{\tr}\widehat{V}^{(s,a)}\big](I-\gamma\Tilde{\P}_{\pi})^{-1}e_s,
    \label{eq6:appdx_error}
\end{align}
where the second equality holds due to the perturbation of $\widehat{\P}^{(s,a)}_{\pi}$ from $\Tilde{\P}_{\pi}$, and $e_s\in\R^{|\bS|}$ has value $1$ at entry $s$ and $0$ elsewhere. By~\eqref{eq5:appdx_error} and~\eqref{eq6:appdx_error}, we have
\begin{align}
    \big\|\Tilde{V}-\widehat{V}^{(s,a)}\big\|_\infty\leq(1-\gamma)^{-1}\big|(\Tilde{\P}_{s,a}-\P_{s,a})^{\tr}\widehat{V}^{(s,a)}\big|\leq\tfrac{1}{4}\big\|\widehat{V}^{(s,a)}-V\big\|_\infty+(1-\gamma)^{-1}\big|(\Tilde{\P}_{s,a}-\P_{s,a})^{\tr}V\big|.
    \label{eq7:appdx_error}
\end{align}
It then follows from~\eqref{eq7:appdx_error} that
\begin{align}
    \big\|\widehat{V}^{(s,a)}-V\big\|_\infty\leq\big\|\widehat{V}^{(s,a)}-\Tilde{V}\big\|_\infty+\|\Tilde{V}-V\|_\infty\leq\tfrac{1}{4}\big\|\widehat{V}^{(s,a)}-V\big\|_\infty+(1-\gamma)^{-1}\big|(\Tilde{\P}_{s,a}-\P_{s,a})^{\tr}V\big|+\|\Tilde{V}-V\|_\infty,
    \label{eq8:appdx_error}
\end{align}
which can be simplified as
\begin{align}
    \big\|\widehat{V}^{(s,a)}-V\big\|_\infty\leq\tfrac{4}{3}\big[\|\Tilde{V}-V\|_\infty+(1-\gamma)^{-1}\big|(\Tilde{\P}_{s,a}-\P_{s,a})^{\tr}V\big|\big].
    \label{eq9:appdx_error}
\end{align}
Combining~\eqref{eq4:appdx_error} with~\eqref{eq9:appdx_error}, we obtain an upper bound for the term $A$ in~\eqref{eq3:appdx_error} as
\begin{align}
    |A|\leq\tfrac{1}{3}\big[(1-\gamma)\|\Tilde{V}-V\|_\infty+\big|(\Tilde{\P}_{s,a}-\P_{s,a})^{\tr}V\big|\big].
    \label{eq10:appdx_error}
\end{align}
\vgap

\paragraph{Bound for $B$ in~\eqref{eq3:appdx_error}} Next, we discuss an upper bound for the term $B$ in~\eqref{eq3:appdx_error}. We first express the difference between $(I-\gamma\Tilde{\P}_{\pi})^{-1}$ and $(I-\gamma\widehat{\P}_{\pi}^{(s,a)})^{-1}$ using the Woodbury matrix identity, i.e., we define
\begin{align}
    M:=(I-\gamma\Tilde{\P}_{\pi})^{-1}-\big(I-\gamma\widehat{\P}_{\pi}^{(s,a)}\big)^{-1}=\gamma\pi(a|s)\tfrac{(I-\gamma\widehat{\P}_{\pi}^{(s,a)})^{-1}e_s(\Tilde{\P}_{s,a}-\P_{s,a})^{\tr}(I-\gamma\widehat{\P}_{\pi}^{(s,a)})^{-1}}{1-\gamma\pi(a|s)(\Tilde{\P}_{s,a}-\P_{s,a})^{\tr}(I-\gamma\widehat{\P}_{\pi}^{(s,a)})^{-1}e_s},
    \label{eq11:appdx_error}
\end{align}
which utilizes the fact $I-\gamma\Tilde{\P}_{\pi}=(I-\gamma\widehat{\P}_{\pi}^{(s,a)})-\pi(a|s)\gamma e_s(\Tilde{\P}_{s,a}-\P_{s,a})^{\tr}$. Then, in view of~\eqref{eq6:appdx_error}, we express the term $B$ in~\eqref{eq3:appdx_error} as
\begin{align}
    B&=(\Tilde{\P}_{s,a}-\P_{s,a})^{\tr}\big(\Tilde{V}-\widehat{V}^{(s,a)}\big)=\gamma\pi(a|s)\big[(\Tilde{\P}_{s,a}-\P_{s,a})^{\tr}\widehat{V}^{(s,a)}\big](\Tilde{\P}_{s,a}-\P_{s,a})^{\tr}(I-\gamma\Tilde{\P}_{\pi})^{-1}e_s \nn\\
    &=\gamma\pi(a|s)\big[(\Tilde{\P}_{s,a}-\P_{s,a})^{\tr}\widehat{V}^{(s,a)}\big](\Tilde{\P}_{s,a}-\P_{s,a})^{\tr}\big[\big(I-\gamma\widehat{\P}_{\pi}^{(s,a)}\big)^{-1}+M\big]e_s \nn\\
    &=\gamma\pi(a|s)\big[(\Tilde{\P}_{s,a}-\P_{s,a})^{\tr}\widehat{V}^{(s,a)}\big](\Tilde{\P}_{s,a}-\P_{s,a})^{\tr}\big(I-\gamma\widehat{\P}_{\pi}^{(s,a)}\big)^{-1}e_s \nn\\
    &\quad +(\gamma\pi(a|s))^2\big[(\Tilde{\P}_{s,a}-\P_{s,a})^{\tr}\widehat{V}^{(s,a)}\big]\tfrac{(\Tilde{\P}_{s,a}-\P_{s,a})^{\tr}(I-\gamma\widehat{\P}_{\pi}^{(s,a)})^{-1}e_s(\Tilde{\P}_{s,a}-\P_{s,a})^{\tr}(I-\gamma\widehat{\P}_{\pi}^{(s,a)})^{-1}e_s}{1-\gamma\pi(a|s)(\Tilde{\P}_{s,a}-\P_{s,a})^{\tr}(I-\gamma\widehat{\P}_{\pi}^{(s,a)})^{-1}e_s},
    \label{eq12:appdx_error}
\end{align}
where the second equality follows from~\eqref{eq6:appdx_error}, and the third and fourth equalities use~\eqref{eq11:appdx_error}. Let $x=(\Tilde{\P}_{s,a}-\P_{s,a})^{\tr}(I-\gamma\widehat{\P}_{\pi}^{(s,a)})^{-1}e_s$. Since $\widehat{\P}_{\pi}^{(s,a)}$ is independent of $\Tilde{\P}_{s,a}$, Lemma~\ref{lem:concentration_bounds}$(a)$ guarantees that
\begin{align}
    |x|\leq\sqrt{2}m^{-\tfrac{1}{2}}\sqrt{\ln(6|\bS||\bA|\delta^{-1})}\big\|\big(I-\gamma\widehat{\P}_{\pi}^{(s,a)}\big)^{-1}e_s\big\|_\infty\leq\tfrac{1}{4}
    \label{eq13:appdx_error}
\end{align}
is satisfied with probability at least $1-\delta/(3|\bS||\bA|)$, where the second inequality uses $m\geq 32(1-\gamma)^{-2}\sqrt{\ln(6|\bS||\bA|\delta^{-1})}$ by~\eqref{m_error_pvki_stoc1}. Then, it follows from~\eqref{eq12:appdx_error} that
\begin{align}
    |B|&=\Big|\big[\gamma\pi(a|s)x+\left(\gamma\pi(a|s)\right)^2\tfrac{x^2}{1-\gamma\pi(a|s)x}\big]\cdot\big[(\Tilde{\P}_{s,a}-\P_{s,a})^{\tr}\widehat{V}^{(s,a)}\big]\Big|\leq\big|(\Tilde{\P}_{s,a}-\P_{s,a})^{\tr}\widehat{V}^{(s,a)}\big|\cdot\big|\tfrac{\gamma\pi(a|s)x}{1-\gamma\pi(a|s)x}\big| \nn\\
    &\leq\tfrac{1/4}{1-1/4}\big|(\Tilde{\P}_{s,a}-\P_{s,a})^{\tr}\widehat{V}^{(s,a)}\big|=\tfrac{1}{3}\big|(\Tilde{\P}_{s,a}-\P_{s,a})^{\tr}\widehat{V}^{(s,a)}\big|\leq\tfrac{1}{3}\big(|A|+\big|(\Tilde{\P}_{s,a}-\P_{s,a})^{\tr}V\big|\big) \nn\\
    &\leq\tfrac{1}{9}(1-\gamma)\|\Tilde{V}-V\|_\infty+\tfrac{4}{9}\big|(\Tilde{\P}_{s,a}-\P_{s,a})^{\tr}V\big|,
    \label{eq14:appdx_error}
\end{align}
where the second inequality applies~\eqref{eq13:appdx_error}, the third inequality uses the definition of $A$ in~\eqref{eq3:appdx_error}, and the last inequality follows from~\eqref{eq10:appdx_error}.
\vgap

\paragraph{Proof for~\eqref{error_pvki_stoc1} in Lemma~\ref{lem:error_pv_stoc1}} Finally, in view of~\eqref{eq3:appdx_error}, and combining~\eqref{eq10:appdx_error} and~\eqref{eq14:appdx_error}, we obtain
\begin{align}
    \big|(\Tilde{\P}_{s,a}-\P_{s,a})^{\tr}(\Tilde{V}-V)\big|\leq |A|+|B|\leq\tfrac{4}{9}(1-\gamma)\left\|\Tilde{V}-V\right\|_\infty+\tfrac{7}{9}\big|(\Tilde{\P}_{s,a}-\P_{s,a})^{\tr}V\big|.
    \label{eq15:appdx_error}
\end{align}
Note that by taking a union bound for the events~\eqref{eq4:appdx_error} and~\eqref{eq13:appdx_error} over all the state-action pairs $(s,a)\in\bS\times\bA$,~\eqref{eq15:appdx_error} holds for all $(s,a)\in\bS\times\bA$ with probability at least $1-(2/3)\delta$. Then~\eqref{eq2:appdx_error} together with~\eqref{eq15:appdx_error} yields
\begin{align}
    \|\Tilde{V}-V\|_\infty&\leq \big\|(I-\gamma\P_{\pi})^{-1}\big[(\Tilde{\P}_{\pi}-\P_{\pi})V+(\Tilde{\P}_{\pi}-\P_{\pi})(\Tilde{V}-V)\big]\big\|_\infty \nn\\
    &\leq\tfrac{4}{9}\|\Tilde{V}-V\|_\infty+\tfrac{16}{9}\big\|(I-\gamma\P_{\pi})^{-1}\big\langle|(\Tilde{\P}-\P)V|,\pi\big\rangle\big\|_\infty,
    \label{eq16:appdx_error}
\end{align}
where $|(\Tilde{\P}-\P)V|$ denotes element-wise absolute value of $(\Tilde{\P}-\P)V$. Hence, it follows from~\eqref{eq16:appdx_error} that
\begin{align}
    &\left\|\Tilde{V}-V\right\|_\infty\leq \tfrac{16}{5}\gamma\big\|(I-\gamma\P_{\pi})^{-1}\big\langle|(\Tilde{\P}-\P)V|,\pi\big\rangle\big\|_\infty \nn\\
    &\leq\tfrac{16}{5}\gamma\Big\|(I-\gamma\P_{\pi})^{-1}\Big(\big\langle\sqrt{2}m^{-1/2}\sqrt{\ln(6|\bS||\bA|\delta^{-1})}\sqrt{\sigma_V},\pi\big\rangle+\tfrac{2}{3}\|V\|_\infty m^{-1}\ln(6|\bS||\bA|\delta^{-1})\cdot\mathbf{1}_{\bS}\Big)\Big\|_\infty \nn\\
    &\leq\tfrac{16\sqrt{2}}{5}\gamma m^{-\tfrac{1}{2}}\sqrt{\ln(6|\bS||\bA|\delta^{-1})}\big\|(I-\gamma\P_{\pi})^{-1}\left\langle\sqrt{\sigma_V},\pi\right\rangle\big\|_\infty+\tfrac{32}{15}\gamma(1-\gamma)^{-1}\|V\|_\infty m^{-1}\ln(6|\bS||\bA|\delta^{-1}) \nn\\
    &\leq\tfrac{16\sqrt{2}}{5}\gamma\sqrt{\ln(6|\bS||\bA|\delta^{-1})}\times\tfrac{\sqrt{2}(1+\Bar{h})}{\gamma(1-\gamma)^{3/2}}\times m^{-\tfrac{1}{2}}+\tfrac{32}{15}\gamma(1-\gamma)^{-1}\ln(6|\bS||\bA|\delta^{-1})\times\tfrac{1+\Bar{h}}{1-\gamma}\times m^{-1}\leq\tfrac{1}{2}u+\tfrac{1}{2}u=u
    \label{eq17:appdx_error}
\end{align}
holds with probability at least $1-\delta$, where the second inequality applies Lemma~\ref{lem:concentration_bounds}$(b)$ and takes a union bound over the probability $1-\delta/(3|\bS||\bA|)$ for all $(s,a)\in\bS\times\bA$, as well as the $1-(2/3)\delta$ probability of~\eqref{eq16:appdx_error}, the fourth inequality utilizes Lemma~\ref{lem:total_variance},~\eqref{law_of_total_var} and $\|V\|_\infty\leq(1+\Bar{h})/(1-\gamma)$, and the last inequality follows from $m\geq 200(1+\Bar{h})^2(1-\gamma)^{-3}u^{-2}\ln(6|\bS||\bA|\delta^{-1})$ by~\eqref{m_error_pvki_stoc1} and $u\leq (1+\Bar{h})/(1-\gamma)$. We then conclude from~\eqref{eq15:appdx_error} and~\eqref{eq17:appdx_error} that for any $(s,a)\in\bS\times\bA$,
\begin{align}
    &\big|(\Tilde{\P}_{s,a}-\P_{s,a})^{\tr}\Tilde{V}\big|\leq\big|(\Tilde{\P}_{s,a}-\P_{s,a})^{\tr}(\Tilde{V}-V)\big|+\big|(\Tilde{\P}_{s,a}-\P_{s,a})^{\tr}V\big|\leq\tfrac{4}{9}(1-\gamma)\|\Tilde{V}-V\|_\infty+\tfrac{16}{9}\big|(\Tilde{\P}_{s,a}-\P_{s,a})^{\tr}V\big| \nn\\
    &\leq\tfrac{4}{9}(1-\gamma)u+\tfrac{16}{9}\big[\sqrt{2}m^{-1/2}\sqrt{\ln(6|\bS||\bA|\delta^{-1})}\sqrt{\sigma_V(s,a)}+\tfrac{2}{3}\|V\|_\infty m^{-1}\ln(6|\bS||\bA|\delta^{-1})\big] \nn\\
    &\leq\tfrac{4}{9}(1-\gamma)u+\tfrac{16}{9}\big[\tfrac{1}{10(1+\Bar{h})}(1-\gamma)^{\tfrac{3}{2}}u\sqrt{\sigma_V(s,a)}+\tfrac{1}{300}(1-\gamma)u\big]\leq\tfrac{1}{2}\big[\tfrac{(1-\gamma)^{1/2}}{1+\Bar{h}}\sqrt{\sigma_V(s,a)}+1\big](1-\gamma)u,
    \label{eq18:appdx_error}
\end{align}
where the second inequality follows from~\eqref{eq15:appdx_error}, the third inequality utilizes~\eqref{eq17:appdx_error} and the error bound used in the second inequality of~\eqref{eq17:appdx_error} obtained from Lemma~\ref{lem:concentration_bounds}$(b)$, and the fourth inequality applies $m\geq 200(1+\Bar{h})^2(1-\gamma)^{-3}u^{-2}\ln(6|\bS||\bA|\delta^{-1})$ by~\eqref{m_error_pvki_stoc1} and $\|V\|_\infty\leq(1+\Bar{h})/(1-\gamma)$.

We conclude this proof of Lemma~\ref{lem:error_pv_stoc1}$(b)$ by combining~\eqref{eq17:appdx_error} and~\eqref{eq18:appdx_error}.

\end{document}